\documentclass[parskip=never,abstract=true]{scrartcl}

\usepackage[utf8]{inputenc}
\usepackage{csquotes}

\usepackage[T1]{fontenc}
\usepackage[table]{xcolor}

\usepackage{lmodern}
\usepackage{yfonts}
\usepackage{textcomp}
\usepackage{setspace}
\usepackage[a4paper,left=2.5cm,right=2cm,top=2.5cm,bottom=2.5cm]{geometry}

\usepackage{extarrows}
\usepackage{etoolbox}
\usepackage{xparse}

\usepackage{mathtools}
\usepackage{amsthm,thmtools}
\usepackage{dsfont}
\let\mathbb\mathds
\usepackage{amssymb}
\usepackage[scr]{rsfso} 
\usepackage{bm}
\usepackage{euscript}
\usepackage{amsbsy}
\DeclareMathAlphabet\mathbfcal{OMS}{cmsy}{b}{n}

\usepackage{microtype}
\usepackage{authblk}
\usepackage{tikz}
\usetikzlibrary{cd,quotes}
\usetikzlibrary{decorations.markings}
\usetikzlibrary{decorations.pathreplacing,calligraphy}
\usepackage{pgfplots}
\pgfplotsset{compat=1.13}
\usepackage{booktabs}
\usepackage{float}
\usepackage{environ}
\usepackage{xcolor,soul}
\usepackage{hyperref} 
\hypersetup{%
  colorlinks,%
  linkcolor={red!60!black},%
  citecolor={blue!50!black},%
  urlcolor={blue!80!black}%
}
\usepackage{enumitem,cleveref}
\usepackage{eso-pic}

\usepackage[normalem]{ulem}

\newenvironment{subproof}[1][\proofname]{%
	\begin{proof}[#1]%
	}{%
	\end{proof}%
}

\usepackage{tikz-cd}
\tikzset{Rightarrow/.style={double equal sign distance,>={Implies},->},
  triple/.style={-,preaction={draw,Rightarrow}},
  quadruple/.style={preaction={draw,Rightarrow,shorten >=0pt},shorten >=1pt,-,double,double
    distance=0.2pt}}

\def\on{\operatorname}

\def\sf{\mathsf}
\def\CC{\mathbb{C}}

\def\C{\EuScript{C}}
\def\D{\EuScript{D}}

\def\BB{\mathbb{B}}

\def\DD{\mathbb{D}}
\def\bS{\textbf{S}}

\def\OO{\mathbb{O}}
\def\Fun{\on{Fun}}
\def\Cat{\on{Cat}}
\def\iCat{\EuScript{C}\!\on{at}}

\def\Hom{\on{Hom}}
\def\D{\EuScript{D}}
\def\Nerv{\on{N}}
\def\Nsc{\on{N}^{\on{sc}}}

\def\Nms{\on{N}^{\on{ms}}}

\def\id{\on{id}}
\def\Set{\on{Set}}
\def\Map{\on{Map}}
\def\scsSet{{\on{Set}_{\Delta}^{\mathbf{sc}}}}
\def\mbsSet{{\Set_\Delta^{\mathbf{mb}}}}
\def\St{\on{St}}
\def\ST{\mathbb{S}\!\on{t}}
\def\Un{\on{Un}}
\def\UN{\mathbb{U}\!\on{n}}
\def\sc{{\on{sc}}}

\SetEnumitemKey{axioms}{itemsep=0pt,align=left,itemindent=!,
                        listparindent=\parindent,parsep=\parskip,
                        before=,
                        label={\Roman*}}

\SetEnumitemKey{claims}{itemsep=0pt,align=left,itemindent=!,
                        listparindent=\parindent,parsep=\parskip,
                        before=,
                        label={\Roman*}}

\DeclarePairedDelimiterX\set[1]{\lbrace}{\rbrace}{\def\given{\:\delimsize\vert\allowbreak\:}#1}
\newlist{implications}{description}{1} 
\setlist[implications]{itemsep=0pt,leftmargin=\parindent}
\NewDocumentCommand\implication{o}
  {\IfValueTF{#1}
    {\auximplication#1\relax}
    {\item[\normalfont($\,\Rightarrow\,$)]}}
\NewDocumentCommand\auximplication{u-u\relax}
  {\item[\normalfont(#1)$\,\Rightarrow\,$(#2)]}

\makeatletter
\newcounter{diagram}[section]
\def\thediagram{\thesection.\arabic{diagram}}
\def\ftype@diagram{4}
\def\ext@diagram{diag}
\def\fnum@diagram{Diagram~\thediagram}
\def\fs@diagram{htbp!}
\ExplSyntaxOn
\NewDocumentEnvironment{diagram}{O{htbp!}m}
  {\@float{diagram}[#1]\centering}
  {
   
   \caption{}
   \label{#2}
   \end@float
  }

\newcounter{subdiagram}[diagram]
\def\thesubdiagram{\thediagram.\arabic{subdiagram}}
\NewEnviron{multidiagram}{\expandafter\domultidiagram\BODY\enddomultidiagram}
\NewDocumentCommand\domultidiagram{omu\enddomultidiagram}
 {
  \IfValueTF{#1}{\diagram[#1]}{\diagram}{}
   \refstepcounter{diagram}
   \centering
    \seq_clear:N \l_tmpb_seq
    \seq_set_split:Nnn \l_tmpa_seq { \next } { #3 }
    \seq_map_inline:Nn \l_tmpa_seq
     {
      \seq_put_right:Nn \l_tmpb_seq
       {
        \begin{tabular}[b]{@{}c@{}}
         ##1 \\[3ex]
         \refstepcounter{subdiagram}
         \label{#2\othercolon\the\value{subdiagram}}
         Diagram~\thesubdiagram 
        \end{tabular}
       }
     }
    \seq_use:Nn \l_tmpb_seq { \qquad }
   \let\label\@gobble
   \let\caption\@gobble
  \enddiagram
 }
\ExplSyntaxOff
\def\othercolon{:}
\makeatother

\declaretheoremstyle[bodyfont=\itshape,notefont=\bfseries]{abellanA}
\declaretheoremstyle[notefont=\bfseries]{abellanB}

\declaretheorem[style=abellanA,numberwithin=section,name={Theorem}]{theorem}

\declaretheorem[style=abellanA,numberlike=theorem,name={Lemma}]{lemma}
\declaretheorem[style=abellanB,numberlike=theorem,name={Definition}]{definition}
\declaretheorem[style=abellanB,numberlike=theorem,name={Remark}]{remark}
\declaretheorem[style=abellanB,numberlike=theorem,name={Construction}]{construction}
\declaretheorem[style=abellanA,numberlike=theorem,name={Proposition}]{proposition}
\declaretheorem[style=abellanB,numberlike=theorem,name={Example}]{example}
\declaretheorem[style=abellanB,numbered=no,name={Notation}]{notation}
\declaretheorem[style=abellanA,numberlike=theorem,name={Corollary}]{corollary}

\newtheorem*{thm*}{Theorem}
\newtheorem*{prop*}{Proposition}
\newtheorem*{cor*}{Corollary}

\docsvlist{theorem,lemma,definition,remark,proposition,example,corollary}

\let\leq\leqslant
\let\geq\geqslant
\let\epsilon\varepsilon
\let\isom\simeq

\newcommand*\numberset{\mathbb}
\newcommand*\tensor{\otimes}

\newcommand*\N{\numberset{N}}

\newcommand*\mathblank{\mathord{-}}

\DeclareMathOperator*\colim{colim}

\def\msSet{{\on{Set}_{\Delta}^+}}
\def\mssSet{{\on{Set}^{\mathbf{ms}}_{\Delta}}}

\def\Cat{\on{Cat}}

\usepackage[bbgreekl]{mathbbol}
\def\cchi{\mathbb{\bbchi}}
\def\PPhi{\mathbb{\bbphi}}

\let\emptyset\varnothing

\makeatletter
\newcommand{\fixed@sra}{$\vrule height 2\fontdimen22\textfont2 width 0pt\rightarrow$}
\newcommand{\shortarrowup}[1]{%
  \mathrel{\text{\rotatebox[origin=c]{65}{\fixed@sra}}}
}
\newcommand{\shortarrowdown}[1]{%
  \mathrel{\text{\rotatebox[origin=c]{250}{\fixed@sra}}}
}
\makeatother

\newcommand{\upslash}{\!\shortarrowup{1}}

\def\lra{\longrightarrow}
\def\lla{\longleftarrow}

\def\llra{\def\arraystretch{.1}\begin{array}{c} \lra \\ \lla \end{array}}

\tikzcdset{
  arrows={},
  nat/.style={Rightarrow},
  nat1/.style={nat,shift left=.8ex},
  nat2/.style={nat,shift right=.8ex,swap},
  map/.style={},
  map1/.style={map,shift left=.7ex},
  map2/.style={map,shift right=.7ex,swap},
  mapA/.style={map,shift left=.5ex},
  mapB/.style={map,shift left=.5ex},
  unique/.style={line cap=round,densely dotted},
  incl/.style={hookrightarrow},
  mono/.style={rightarrowtail},
  epi/.style={twoheadrightarrow},
  iso/.style={map,sloped,"{\tikz[x=.6ex,y=.3ex]{\draw[-](0,0)sin(1,1)cos(2,0)sin(3,-1)cos(4,0);}}"},
  line/.style={dash},
  line1/.style={line,shift left=.3ex},
  line2/.style={line,shift right=.3ex,swap},
}

\newcommand*\adj[4]
{\begin{tikzcd}[ampersand replacement=\&]
  #1 \arrow[r,mapA,"#3" ] \& #2 \arrow[l,mapB,"#4"]
 \end{tikzcd}}

\def\op{{\on{op}}}

\DeclareMathOperator\SSt{\mathbb{S}t}

\makeatletter
\newcommand*\dirlim{\mathop{\mathpalette\varlim@{\rightarrowfill@\scriptscriptstyle}}\nmlimits@}
\newcommand*\prolim{\mathop{\mathpalette\varlim@{\leftarrowfill@\scriptscriptstyle}}\nmlimits@}
\makeatother

\def\llra{\def\arraystretch{.1}\begin{array}{c} \lra \\ \lla \end{array}}

\newcommand{\nat}{\Rightarrow}

\tikzset{
  abellanarrows/.style={line cap=round,line join=round,line width=.4pt},
  abellanarrowlength/.store in=\abellanarrowlength,
}

\ExplSyntaxOn

\cs_generate_variant:Nn \tl_replace_all:Nnn { Nf }

\NewDocumentCommand \func { s O{} m }
 {
  \group_begin:
   \IfBooleanTF{#1}
    { \keys_set:nn { abellan / func } { aligned = true , #2 } }
    { \keys_set:nn { abellan / func } {#2} }
   \abellan_func:n {#3}
  \group_end:
 }
\NewDocumentCommand \arr { s o m }
 {
  \IfBooleanF{#1}
   { \bool_if:NT \l_abellan_aligned_bool { & } }
  \abellan_arr:n {#3}
 }
\NewDocumentCommand \addarr { o m m }
 {
  \keys_set:nn { abellan / func / addarrow } { name = {#2} , #3 }
  \tl_clear:N \l_abellan_arrname_tl
 }
\NewDocumentCommand \setupfunc { m } { \keys_set:nn { abellan / func } {#1} }

\bool_new:N \l_abellan_aligned_bool
\tl_new:N \l_abellan_func_tl
\tl_new:N \l_abellan_arrname_tl
\tl_new:N \l_abellan_arrmode_tl
\tl_new:N \g_abellan_func_substitutions_tl
\quark_new:N \q_abellan

\keys_define:nn { abellan / func }
 {
  aligned .bool_set:N = \l_abellan_aligned_bool ,
  aligned .initial:n = false ,
  mode .choices:nn = { base , tikz } { \tl_set_eq:NN \l_abellan_arrmode_tl \l_keys_choice_tl } ,
  mode .initial:n = base ,
  tikzarrowlength .code:n = \tikzset{abellanarrowlength={#1}} ,
  tikzarrowlength .initial:n = 2em ,
  tikzarrowstyle .code:n = \tikzset{abellanarrows/.append ~ style={#1}} ,
 }
\keys_define:nn { abellan / func / addarrow }
 {
  name .tl_set:N = \l_abellan_arrname_tl ,
  base .code:n =
    \abellan_addarrow:nnww { \l_abellan_arrname_tl } { base } #1 \q_abellan ,
  tikz .code:n =
    \abellan_addarrow:nnww { \l_abellan_arrname_tl } { tikz } #1 \q_abellan ,
  substitutions .code:n =
   \tl_map_inline:nn {#1}
    {
     \tl_gput_right:Nx \g_abellan_func_substitutions_tl
      {
       \exp_not:n { \tl_replace_all:Nnn \l_abellan_func_tl {##1} }
        { \arr{\l_abellan_arrname_tl} }
      }
    } ,
 }
\cs_new_protected:Npn \abellan_func:n #1
 {
  \resetdynamicto
  \tl_set:Nn \l_abellan_func_tl {#1}
  \tl_replace_all:Nfn \l_abellan_func_tl
   { \char_generate:nn { `: } { 12 } } { \colon }
  \bool_if:NTF \l_abellan_aligned_bool
   { \tl_replace_all:Nnn \l_abellan_func_tl { ; } { \\ } }
   { \tl_replace_all:Nnn \l_abellan_func_tl { ; } { ,\ } }
  \tl_use:N \g_abellan_func_substitutions_tl
  \bool_if:NT \l_abellan_aligned_bool { \! \begin {aligned} \scan_stop: }
  \tl_use:N \l_abellan_func_tl
  \bool_if:NT \l_abellan_aligned_bool { \\ \end {aligned} }
 }
\NewDocumentCommand \abellan_addarrow:nnww { m m O{} u\q_abellan }
 {
  \exp_args:Nc \NewDocumentCommand { abellan_arr_#1_#2:w } { #3 }
   {
    \use:c { abellan_arr_ \l_abellan_arrmode_tl :n } { #4 }
   }
 }
\cs_new_protected:Npn \abellan_arr:n #1
 {
  \use:c { abellan_arr_#1_\l_abellan_arrmode_tl :w }
 }
\cs_new_protected:Npn \abellan_arr_base:n #1
 {
  \mathrel{#1}
 }
\cs_new_protected:Npn \abellan_arr_tikz:n #1
 {
  \mathrel{\vcenter{\hbox{\tikz[abellanarrows]{#1}}}}
 }
\ExplSyntaxOff

\addarr{monomorphism}
 {
  substitutions = {>->}{ mono }{\rightarrowtail}{\mono} ,
  base = [o]{\IfValueTF{#1}{\rightarrowtail{#1}}{\mono}} ,
  tikz = [o]{\draw[>->,inner sep=0pt]
     (0,0)--({max(\abellanarrowlength,.5em\IfValueT{#1}{+width("$\scriptstyle#1$")})},0)
     \IfValueT{#1}{node[midway,above=.7ex]{\smash{$\scriptstyle#1$}}
     node[midway,below=.7ex]{}}
     ;} ,
  }
\addarr{mapsto}
 {
  substitutions = {\mapsto}{ mapsto } ,
  base = \mapsto ,
  tikz = {\draw[|->](0,0)--(\abellanarrowlength,0);} ,
 }
\addarr{rrarrows} 
 {
  substitutions = {\rightrightarrows}{ doublerr }{\doublerr} ,
  base = \rightrightarrows ,
  tikz = {\draw[>={To[width=.8ex,length=.4ex]},->](0,0)--(\abellanarrowlength,0);
          \draw[yshift=.8ex,>={To[width=.8ex,length=.4ex]},->](0,0)--(\abellanarrowlength,0);} ,
 }
\addarr{inclusion}
 {
  substitutions = {\hookrightarrow}{\incl}{ incl } ,
  base = [o]{\IfValueTF{#1}{\hookrightarrow{#1}}{\incl}} ,
  tikz = {\draw[{Hooks[right]}->](0,0)--(\abellanarrowlength,0);} ,
 }
\addarr{natural}
 {
 substitutions = {\Rightarrow}{=>}{ nat }{\nat} ,
  base = [o]{\IfValueTF{#1}{\xRightarrow{#1}}{\nat}} ,
  tikz = [o]{\draw[line cap=butt,double equal sign distance,-Implies,inner sep=0pt]
    (0,0)--({max(\abellanarrowlength,.3em\IfValueT{#1}{+width("$\scriptstyle#1$")})},0)
    \IfValueT{#1}{node[midway,above=.7ex]{\smash{$\scriptstyle#1$}}
    node[midway,below=.7ex]{}}
    ;} ,
 }
\addarr{epimorphism} 
 {
  substitutions = {\twoheadrightarrow}{->>}{ epi }{\epi},
   base = [o]{\IfValueTF{#1}{\xtwoheadrightarrow{#1}}{\epi}} ,
   tikz = [o]{\draw[->>,inner sep=0pt]
      (0,0)--({max(\abellanarrowlength,.5em\IfValueT{#1}{+width("$\scriptstyle#1$")})},0)
      \IfValueT{#1}{node[midway,above=.7ex]{\smash{$\scriptstyle#1$}}
      node[midway,below=.7ex]{}}
      ;} ,
   }
\addarr{to}
 {
  substitutions = {\to}{\rightarrow}{ to }{ funct }{ map }{->} ,
  base = [o]{\IfValueTF{#1}{\xrightarrow{#1}}{\to}} ,
  tikz = [o]{\draw[->,inner sep=0pt]
    (0,0)--({max(\abellanarrowlength,.5em\IfValueT{#1}{+width("$\scriptstyle#1$")})},0)
    \IfValueT{#1}{node[midway,above=.7ex]{\smash{$\scriptstyle#1$}}
    node[midway,below=.7ex]{}}
    ;} ,
 }
\newcommand*\resetdynamicto
  {\gdef\dynamicto{\arr*{to}\gdef\dynamicto{\arr*{mapsto}}}}
\resetdynamicto

\setupfunc{mode=tikz}

\ExplSyntaxOn
\NewDocumentCommand \printheader { m o m }
 {
  \par\noindent
  \begin{minipage}[t]{\textwidth}\noindent
  
  \begin{tabular}[t]{ll}
     & \keyval_parse:NNn \abellan_printname:n \abellan_printnamemail:nn { #3 } 
  \end{tabular}
  \vspace{.4cm}
  \end{minipage}
  \begin{center}\Large\bfseries
   #1 \IfValueT{#2}{\\[1ex] \large #2}
  \end{center}
  \vspace{.6cm}
 }
\tl_new:N \g_abellan_autores_tl
\tl_new:N \g_abellan_asignatura_tl
\cs_new_protected:Npn \abellan_printnamemail:nn #1 #2
 {
  #1 \\ & \quad\texttt{#2} \\ &
 }
\cs_new_protected:Npn \abellan_printname:n #1
 {
  #1 \\ &
 }
\ExplSyntaxOff

\makeatletter
\tikzcdset{
  open/.code     = {\tikzcdset{hook, circled};},
  open'/.code    = {\tikzcdset{hook', circled};},
  circled/.code  = {\tikzcdset{markwith = {\draw (0,0) circle (.375ex);}};},
  markwith/.code ={
    \pgfutil@ifundefined%
    {tikz@library@decorations.markings@loaded}%
    {\pgfutil@packageerror{tikz-cd}{You need to say %
        \string\usetikzlibrary{decorations.markings} to use arrows with markings}{}}{}%
    \pgfkeysalso{/tikz/postaction = {
        /tikz/decorate,
        /tikz/decoration={markings, mark = at position 0.5 with {#1}}}
    }
  },
}
\makeatother

\def\mbsSet{\on{Set}_{\Delta}^{\mathbf{mb}}}
\def\bS{\textbf{MB}}
\def\sS{\textbf{S}}
\def\Sst{\mathbb{S}\!\on{t}}

\def\scr{\EuScript}

\makeatletter
\newcommand{\myitem}[1]{%
  \item[#1]\protected@edef\@currentlabel{#1}%
}
\makeatother

\makeatletter
\DeclareSymbolFont{lettersA}{U}{txmia}{m}{it}

\DeclareRobustCommand*{\varmathbb}[1]{\gdef\F@ntPrefix{m@thbbch@r}%
	\@EachCharacter #1\@EndEachCharacter}

\long\def\DoLongFutureLet #1#2#3#4{%
	\def\@FutureLetDecide{#1#2\@FutureLetToken
		\def\@FutureLetNext{#3}\else
		\def\@FutureLetNext{#4}\fi\@FutureLetNext}
	\futurelet\@FutureLetToken\@FutureLetDecide}
\def\DoFutureLet #1#2#3#4{\DoLongFutureLet{#1}{#2}{#3}{#4}}
\def\@EachCharacter{\DoFutureLet{\ifx}{\@EndEachCharacter}%
	{\@EachCharacterDone}{\@PickUpTheCharacter}}
\def\m@keCharacter#1{\csname\F@ntPrefix#1\endcsname}
\def\@PickUpTheCharacter#1{\m@keCharacter{#1}\@EachCharacter}
\def\@EachCharacterDone \@EndEachCharacter{}

\DeclareMathSymbol{\m@thbbch@rA}{\mathord}{lettersA}{129}
\DeclareMathSymbol{\m@thbbch@rB}{\mathord}{lettersA}{130}
\DeclareMathSymbol{\m@thbbch@rC}{\mathord}{lettersA}{131}
\DeclareMathSymbol{\m@thbbch@rD}{\mathord}{lettersA}{132}
\DeclareMathSymbol{\m@thbbch@rE}{\mathord}{lettersA}{133}
\DeclareMathSymbol{\m@thbbch@rF}{\mathord}{lettersA}{134}
\DeclareMathSymbol{\m@thbbch@rG}{\mathord}{lettersA}{135}
\DeclareMathSymbol{\m@thbbch@rH}{\mathord}{lettersA}{136}
\DeclareMathSymbol{\m@thbbch@rI}{\mathord}{lettersA}{137}
\DeclareMathSymbol{\m@thbbch@rJ}{\mathord}{lettersA}{138}
\DeclareMathSymbol{\m@thbbch@rK}{\mathord}{lettersA}{139}
\DeclareMathSymbol{\m@thbbch@rL}{\mathord}{lettersA}{140}
\DeclareMathSymbol{\m@thbbch@rM}{\mathord}{lettersA}{141}
\DeclareMathSymbol{\m@thbbch@rN}{\mathord}{lettersA}{142}
\DeclareMathSymbol{\m@thbbch@rO}{\mathord}{lettersA}{143}
\DeclareMathSymbol{\m@thbbch@rP}{\mathord}{lettersA}{144}
\DeclareMathSymbol{\m@thbbch@rQ}{\mathord}{lettersA}{145}
\DeclareMathSymbol{\m@thbbch@rR}{\mathord}{lettersA}{146}
\DeclareMathSymbol{\m@thbbch@rS}{\mathord}{lettersA}{147}
\DeclareMathSymbol{\m@thbbch@rT}{\mathord}{lettersA}{148}
\DeclareMathSymbol{\m@thbbch@rU}{\mathord}{lettersA}{149}
\DeclareMathSymbol{\m@thbbch@rV}{\mathord}{lettersA}{150}
\DeclareMathSymbol{\m@thbbch@rW}{\mathord}{lettersA}{151}
\DeclareMathSymbol{\m@thbbch@rX}{\mathord}{lettersA}{152}
\DeclareMathSymbol{\m@thbbch@rY}{\mathord}{lettersA}{153}
\DeclareMathSymbol{\m@thbbch@rZ}{\mathord}{lettersA}{154}

\makeatother 
\def\bcat{\varmathbb}
\def\Fr{\on{Fr}}

\usepackage{glossaries}

\newglossaryentry{DuskNerve}{
	name={$\ensuremath{\N_2}$},
	description={The Duskin nerve $\N_2(\CC)\in \Set_\Delta$ of a 2-category $\CC$}
}

\newglossaryentry{ScNerve}{
	name={$\Nsc$},
	description={The scaled nerve $\Nsc(\CC)\in \scsSet$ of a 2-category $\CC$}
}

\newglossaryentry{GenCatNotation}{
	name={$C$, $\scr{C}$, $\mathbb{C}$, $\bcat{C}$},
	description={Generic notations for categories. A 1-category is $C$, an $(\infty,1)$-category is $\scr{C}$, a 2-category is $\CC$, and an $(\infty,2)$-category is $\bcat{C}$}
}

\newglossaryentry{MBann}{
	name={\textbf{MB}-anodyne},
	description={The marked-biscaled anodyne morphisms, used to construct the model structure on $\mbsSet$}
}

\newglossaryentry{MSann}{
	name={\textbf{MS}-anodyne},
	description={The marked-scaled anodyne morphisms, used to construct the model structure on $\mssSet$}
}

\newglossaryentry{MSsSet}{
	name={$\mssSet$},
	description={The category of marked-scaled simplicial sets}
}

\newglossaryentry{scsSet}{
	name={$\scsSet$},
	description={The category of scaled simplicial sets}
}

\newglossaryentry{MBsSet}{
	name={$\mbsSet$},
	description={The category of marked-biscaled simplicial sets}
}

\newglossaryentry{ST}{
	name={$\ST_\phi$},
	description={The scaled straightening functor}
}

\newglossaryentry{UN}{
	name={$\UN_\phi$},
	description={The scaled unstraightening functor}
}

\newglossaryentry{FX}{
	name={$\mathbb{F}(\bcat{X})$, $\mathbb{F}(p)$},
	description={The free 2-Cartesian fibration associated to an $(\infty,2)$-functor $p:\bcat{X}\to \bcat{D}$}
}

\newglossaryentry{BoxProd}{
	name={$\boxtimes$},
	description={The box product of $\msSet$-enriched functors}
}

\newglossaryentry{GrayProd}{
	name={$\otimes$},
	description={The Gray Product of scaled simplicial sets}
}

\newglossaryentry{MX}{
	name={$\mathcal{M}_X$},
	description={The `mapping simplex' associated to a 2-Cartesian fibration $X\to \Delta^n_\flat$}
}

\newglossaryentry{LX}{
	name={$\mathcal{L}_X$},
	description={Auxiliary form of the mapping simplex $\mathcal{M}_X$}
}

\newglossaryentry{BicatInfty}{
	name={$\varmathbb{B}\negthinspace\operatorname{icat}_\infty$},
	description={The $\infty$-bicategory of $\infty$-bicategories}
}

\newglossaryentry{CartINfty}{
	name={$2\varmathbb{C}\negthinspace\operatorname{art}(S)$},
	description={The $\infty$-bicategory of 2-Cartesian fibrations over a scaled simplicial set $S$}
}

\newglossaryentry{OI}{
	name={$\OO^I$},
	description={The 2-Categories which represent the $n$-simplices of the Duskin Nerve}
}

\newglossaryentry{RhoC}{
	name={$\cchi_{\CC}$},
	description={The relative 2-nerve}
}
\newglossaryentry{FrC}{name={$\Fr(\CC)$},description={The strict free 2-Cartesian fibration}}

\renewcommand{\glossarysection}[2][]{}
\makeglossaries

\title{2-Cartesian fibrations II: A Grothendieck construction for $\infty$-bicategories}
\author{Fernando Abellán \& Walker H. Stern}
\date{} 
\begin{document}
  \maketitle
  \begin{abstract}
  	In this work, we conclude our study of fibred $\infty$-bicategories by providing a Grothendieck construction in this setting. Given a scaled simplicial set $S$ (which need not be fibrant) we construct a 2-categorical version of Lurie's straightening-unstraightening adjunction, thereby furnishing an equivalence between the $\infty$-bicategory of 2-Cartesian fibrations over $S$ and the $\infty$-bicategory of contravariant functors $S^\op \to \bcat{B}\mathbf{\!}\on{icat}_\infty$ with values in the $\infty$-bicategory of $\infty$-bicategories. We provide a relative nerve construction in the case where the base is a 2-category, and use this to prove a comparison to existing bicategorical Grothendieck constructions.
  	\par\vskip\baselineskip\noindent
  	\par\vskip\baselineskip\noindent
  \end{abstract}
  
  \tableofcontents
  
    \section{Introduction}
  
  This is the second of a two-paper sequence devoted to $(\infty,2)$-categorical fibrations and the concomitant Grothendieck constructions, with an eye towards understanding $(\infty,2)$-categorical cofinality. In this paper, we provide a explicit and computationally tractable Grothendieck construction for $(\infty,2)$-categories fibred in $(\infty,2)$-categories. In a companion paper \cite{AGScofinality}, we leverage this technology to prove a complete characterization of \emph{marked cofinal} functors of $(\infty,2)$-categories, which we conjectured in \cite{AGSQuillen}.
  
  The Grothendieck construction first emerged as a tool to study descent in \cite{SGA1}, but it has since become an invaluable tool to study universal properties more generally. In its original form, it takes the form of an equivalence
  \[
  \on{Fib}(\scr{C})\simeq \Fun^{\on{ps}}(\C^\op,\Cat),
  \]
  for any small category $\C$, between the category of \emph{fibred categories over $\C$}, and the category of \emph{pseudo-functors} $\C^{\op} \to \Cat$. The underlying idea is that certain conditions on a a functor $p:\scr{D}\to \scr{C}$ mean that the fibres of $p$ vary \emph{(pseudo-)functorially} in $\scr{C}$. Indeed, the original definition of a fibred category, in \cite{GrothendieckDescent}, was what we today would call a pseudo-functor $F:\scr{C}^\op\to \Cat$. More precisely, an assignment of a category $F(x)\in \Cat$ for every $x\in \C$, a functor $F(f):F(y)\to F(x)$ for every morphisms $f:x\to y$ in $\C$, and natural isomorphisms $F(g)\circ F(f)\cong F(f\circ g)$ for every composable pair of morphisms, satisfying additional coherence conditions. 
  
  The Grothendieck construction reformulates the data of a pseudo-functor into a \emph{Cartesian fibration}. Given a  functor $P:\scr{F}\to \C$, an morphism $f:x\to y$ in $\scr{F}$ is called \emph{Cartesian} if, for every $g:z\to y$ in $\scr{F}$, and every commutative diagram 
  \[
  \begin{tikzcd}
  	& P(x)\arrow[dr,"P(f)"] & \\
  	P(z) \arrow[rr,"P(g)"']\arrow[ur,"h"] & & P(y)
  \end{tikzcd}
  \]
  in $\C$, there is a unique morphism $\tilde{h}:z\to x$ with $P(\tilde{h})=h$, such that $f\circ \tilde{h}=g$. The functor $P$ is said to be an \emph{Cartesian fibration} if, for every $f:c\to P(y)$ in $\scr{C}$, there is a Cartesian morphism $\tilde{f}:x\to y$ in $\scr{F}$ such that $P(\tilde{f})=f$. 
  
  The equivalence between pseudo-functors $F:\scr{C}^\op\to \Cat$ and Cartesian fibrations over $\C$ is then achieved by constructing a Cartesian fibration $P:\on{El}(F)\to \scr{C}$ as follows:
  \begin{itemize}
  	\item The objects of $\on{El}(F)$ consist of pairs $(c,x)$, where $c\in \C$, and $x\in F(c)$. 
  	\item A morphism $(f,\tilde{f}):(c,x)\to (d,y)$ consists of a morphism $f:c\to d$ in $\C$, together with a morphism $\tilde{f}:x\to F(f)(y)$ in $F(x)$.
  \end{itemize}
  The Cartesian morphisms of $\on{El}(F)$ are precisely those $(f,\tilde{f})$ such that $\tilde{f}$ is an isomorphism.

  \subsection{Higher-categorical Grothendieck constructions}
  
  More recent incarnations of the Grothendieck construction have focused on $\infty$-categorical variants. By their very nature, functors of $(\infty,1)$-categories generalize pseudo-functors of $(2,1)$-categories\footnote{There is a quibble to be made here about models, since in quasi-categorical terms, every functor of $(\infty,1)$-categories is assumed to \emph{strictly} preserve identities (degenerate 1-simplices).}, so that higher Grothendieck constructions now take the form of equivalences 
  \[
  \scr{C}\!\on{art}(\C)\simeq \Fun(\C^\op,\iCat_\infty)
  \]
  of $\infty$-categories. This equivalence was proven by Lurie in \cite{HTT}, using model-categorical techniques which we adapt in the present work.
  
  The basic form of these arguments is not hard to follow. Given an $\infty$-category $\scr{C}$, presented as a quasi-category, Lurie defines \emph{marked simplicial sets} over $\scr{C}$ to be pair $(X,M_X)$ consisting of a simplicial set $X\in \Set_\Delta$ and a subset $M_X\subset X_1$ of \emph{marked edges} containing all degenerate edges, equipped with a morphism $p:X\to \scr{C}$ of simplicial sets. Requiring maps to preserve these marked edges yields a category $(\Set_\Delta^+)_{/\scr{C}}$. Lurie then constructs a model structure on this category, the fibrant objects of which satisfy lifting properties akin to those defining 1-categorical Cartesian fibrations. In particular, the corresponding model structure on $\msSet\cong (\msSet)_{/\Delta^0}$ models $(\infty,1)$-categories. 
  
  With these model structures in place, one can consider the category $(\msSet)^{\mathfrak{C}[\scr{C}]^\op}$ of simplicially-enriched functors $\mathfrak{C}[\C]^\op \to \msSet$, and equip it with the projective model structure. The $(\infty,1)$-categorical Grothendieck construction then takes the form of a Quillen equivalence 
  \[
  \St_{\scr{C}}: (\msSet)_{/\C}  \llra (\msSet)^{\mathfrak{C}[\C]^\op}: \Un_{\C} 
  \]
  between these two model categories. 
  
  In the $\infty$-categorical context, Grothendieck constructions have become an indispensable  tool, as the added computational complexity of $\infty$-categorical constructions renders many ad-hoc constructions of functors  nearly impossible to work with. It is often far easier to work with the fibration associated to a functor of $\infty$-categories than with the functor itself. Examples of such applications include the study of monoidal $(\infty,1)$-categories in \cite{DAGX} and \cite{HA} and the approach to lax colimits presented in \cite{GHN}. The study of higher forms of cofinality, which is our intended application, is another case in which it is virtually essential to use the Grothendieck construction.
  
  \subsection{Scaled simplicial sets and non-fibrant base}
  
  Throughout this work, we model $(\infty,2)$-categories using \emph{scaled simplicial sets}. Introduced by Lurie in \cite{LurieGoodwillie}, this model is similar in perspective to the marked simplicial sets described above. More formally, a scaled simplicial set is a pair $(X,T_X)$ consisting of a simplicial set $X\in \Set_{\Delta}$, together with a subset $T_X\subset X_2$ containing all degenerate 2-simplices. We think of the simplices in $T_X$ --- typically referred to as \emph{thin triangles} or \emph{thin two simplices} --- as representing invertible 2-morphisms, and we think of all other two simplices as representing non-invertible 2-morphism. 
  
  Our use of this model for $(\infty,2)$-categories presents a number of advantages. Most immediately apparently, much intuition from quasi-categorical models of $(\infty,1)$-categories carries over to scaled simplicial sets, and such intuitions suffuse all of our proofs. More importantly though, our model-categorical approach to the Grothendieck construction means that our proof holds even when the base of our fibrations is a \emph{non-fibrant scaled simplicial set}. 
  
  The benefit of this generality may not be immediately apparent, but arises from the consideration of lax functors. A morphism of scaled simplicial sets $(X,T_X)\to (Y,T_Y)$ sends scaled 2-simplices to scaled 2-simplices, and thus sends invertible 2-morphisms to invertible 2-morphisms. In this sense, morphisms of scaled simplicial sets between fibrant objects generalize pseudofunctors of 2-categories. However, if one considers a fibrant scaled simplicial set $(X,T_X)$ (what we will later call an $\infty$-bicategory), and replaces the scaling $T_X$ with the scaling $M_X$ which consists in thin triangles where the edge $\Delta^{\set{0,1}}$ or the edge $\Delta^{\set{1,2}}$ is an equivalence in $X$ then a morphism
  \[
  \func{F:  (X,M_X)\to  (Y,T_Y)}
  \]
  now can be seen as a generalization of a \emph{normal lax functor}. The identities are still preserved and the thin triangles of $M_X$ represent 2-morphisms in $X$, which tells us that $F$ is “strictly“ functorial on mapping categories. However, in this situation composability only holds up to chosen coherent 2-morphisms. As a consequence, the Grothendieck construction in the case of non-fibrant base provides a potent tool to understand normal lax $(\infty,2)$-functors in terms of their associated fibrations.
  
  \subsection{Variances}

  The zoo of $(\infty,1)$-categorical Grothendieck constructions is complicated by the fact that such Grothendieck constructions come in \emph{two} variances. One can either consider the aforementioned \emph{Cartesian} fibrations of $(\infty,1)$-categories over $\scr{C}$, or consider \emph{coCartesian} fibrations over $\scr{C}$. The former correspond to $(\infty,1)$-functors 
  \[
  \func{
  	F:\C^\op\to \iCat_\infty
  }
  \]
  whereas the latter correspond to $(\infty,1)$-functors 
  \[
  \func{F:\C \to \iCat_\infty}
  \] 
  Additionally, if one treats the case of functors 
  \[
  \func{F:\C^\op\to \scr{S}\subset \iCat_\infty }
  \]
  valued in $\infty$-groupoids (spaces), one obtains more restrictive variants of Cartesian/coCartesian fibrations, called \emph{right fibrations} and \emph{left fibrations}, respectively, in \cite[Ch. 2]{HTT}. 
  
  Each variance can be obtained from the other by appropriate dualization proceedures, and so, in practice it is only necessary to prove one correspondence to obtain the other. In the world of $(\infty,2)$-categories, where there are \emph{four} possible variances, a similar principle applies, although the dualization procedures can become more complicated. As a result, we have focused on a single variance in our exploration of the $(\infty,2)$-categorical Grothendieck construction. Later in the introduction  we will give a more complete account of known Grothendieck constructions, as well as a table of relations between them. 
  
  \subsection{The $(\infty,2)$-categorical Grothendieck construction}
  
  The present paper provides a complete $(\infty,2)$-categorical Grothendieck construction. Loosely speaking, for every scaled simplicial set $S$, we provide an equivalence of $(\infty,2)$-bicategories (or simply $\infty$-bicategories as in \cite{LurieGoodwillie})
  \[
  2\bcat{C}\!\on{art}(S) \simeq \Fun(S^\op,\bcat{B}\!\on{icat}_{\infty}) 
  \]
  between 2-Cartesian fibrations\footnote{What we call 2-Cartesian fibrations are called \emph{outer 2-Cartesian fibrations} in \cite{GHL_Cart}. Because we focus on a single variance, we trim the terminology for ease of reading.} over $S$, and $(\infty,2)$-functors $S^\op\to \bcat{B}\!\on{icat}_{\infty}$ with values in $\infty$-bicategories. 
  
  To understand this construction on an intuitive level, it is helpful to first consider the strict 2-categorical variant of the construction, developed by Buckley in \cite{Buckley}. In this setting, we consider strict 2-functors $p:\CC\to \DD$. A 1-morphism $f:c\to \overline{c}$ in $\CC$ is called \emph{Cartesian} if, for every $a\in \CC$ there is a pullback square of categories 
  \[
  \begin{tikzcd}
  	\CC(a,c) \arrow[r,"f_\ast"]\arrow[d,"P"'] & \CC(a,\overline{c})\arrow[d,"P"] \\
  	\DD(P(a),P(c)) \arrow[r,"P(f)_\ast"'] & \DD(P(a),P(\overline{c}))
  \end{tikzcd}
  \]
  A 2-morphism $\alpha:f\Rightarrow g$ in $\CC(x,y)$ is called \emph{coCartesian} if it is a coCartesian 1-morphism for the map 
  \[
  \func{P:\CC(x,y)\to \DD(P(x),P(y)).}
  \]
  The functor $P$ is then called a 2-Cartesian fibration if it admits Cartesian lifts of all 1-morphisms, and coCartesian lifts of all 2-morphisms. 
  
  In our previous paper, \cite{AGS_CartI}, we provided a model structure which we claimed modeled the appropriate $(\infty,2)$-categorical variant of the above definitions. To keep track of the data of (1) invertible 2-morphisms, (2) Cartesian 1-morphisms, and (3) coCartesian 2-morphisms in the simplicial setting, we considered a 3-part decoration on simplicial sets. Given a simplicial set $X\in \Set_\Delta$, we define a \emph{marking and biscaling} on $X$ to consist of 
  \begin{itemize}
  	\item As in \cite{LurieGoodwillie}, invertible 2-morphisms are encoded as a collection $T_X\subset X_2$ of 2-simplices, which is required to contain degenerate simplices. The 2-simplices in $T_X$ are called \emph{thin} 2-simplices. 
  	\item As in \cite{HTT}, Cartesian 1-morphisms are encoded as a collection $M_X\subset X_1$ of 1-simplices, which is required to contain the degenerate 1-simplices. The 1-simplices in $M_X$ are called \emph{marked} 1-simplices.
  	\item The coCartesian 2-morphisms are encoded as a collection $C_X\subset X_2$. Since every invertible 2-morphism should be coCartesian, we require that $T_X\subset C_X$. We refer to the 2-simplices in $C_X$ as \emph{lean} 2-simplices.  
  \end{itemize}
  A tuple $(X,M_X,T_X\subset C_X)$ is referred to as a \emph{marked-biscaled simplicial set} (or $\bS$ \emph{simplicial set} for short). We denote the category of $\bS$ simplicial sets by $\mbsSet$. Summarizing the main results from our paper \cite{AGS_CartI}, we have 
  
  \begin{thm*}
  	Let $(S,T_S)$ be a scaled simplicial set. 
  	\begin{enumerate}
  		\item There is a left proper, combinatorial, simplicial model structure on $(\mbsSet)_{/(S,\sharp,T_S\subset \sharp)}$, called the \emph{2-Cartesian model structure}. 
  		\item If $S=\Delta^0$ is the terminal scaled simplicial set, the resulting model structure models $\infty$-bicategories. 
  		\item If $(S,T_S)$ is the scaled nerve of a strict 2-category $\DD$, every 2-Cartesion fibration of strict 2-categories $P:\CC\to \DD$ gives rise to a fibrant object of $(\mbsSet)_{/(S,\sharp,T_S\subset \sharp)}$.
  	\end{enumerate} 
  \end{thm*}
  
  In the second of these results, our decoration becomes highly redundant. In a fibrant object, the marked 1-morphisms correspond to equivalences, the thin 2-simplices correspond to invertible 2-morphisms, but the lean 2-simplices are identical to the thin 2-simplices. To simplify our later computations, we rectify this redundancy by also considering \emph{marked-scaled} simplicial sets, i.e., triples $(X,M_X,T_X)$ consisting of a simplicial set $X$, a collection of marked 1-simplices $M_X$, and a collection of thin 2-simplices $T_X$. The category of marked-scaled simplicial sets is denoted by $\mssSet$. The first result of this paper formalizes the fact that marked-scaled simplicial sets should also model $\infty$-bicategories. 
  
  \begin{thm*}
  	There is a left proper, combinatorial, $\Set_\Delta^+$-enriched model structure on $\mssSet$. Moreover, it is Quillen equivalent to the 2-Cartesian model structure on $\mbsSet$, and thus models $\infty$-bicategories.
  \end{thm*}
  This can be found in the body of the paper as \autoref{thm:markedscaledmodel} and \autoref{prop:markedscaledenriched}.

  The main construction of this paper yields a functor for each scaled simplicial set $S$ 
  \[
  \func{\ST_S: (\mbsSet)_{/S}\to \Fun(\mathfrak{C}^{\sc}[S]^\op,\mssSet)}
  \] 
  called the \emph{bicategorical straightening over $S$}. The functor itself is simply a more highly decorated version of previous straightening functors (e.g., that of \cite{HTT}), and is discussed in detail at the beginning of section \ref{sec:Grothendieck}. We then show that $\ST_S$ admits a right adjoint $\UN_S$ which we call the (bicategorical) \emph{unstraightening over $S$}. As already discussed, the category $(\mbsSet)_{/S}$ carries a model structure which models 2-Cartesian fibrations. If we equip the category of $\on{Set}_\Delta^+$-enriched functors $\mathfrak{C}[S]^\op \to \mssSet$  with the projective model structure, we obtain an enriched model category which models the $(\infty,1)$-category of $(\infty,2)$-functors $S^\op \to \bcat{B}\!\on{icat}_\infty$. The main technical result of this paper is that this adjunction is in fact a Quillen equivalence.

  \begin{thm*}[\autoref{thm:QE_general}]
  	Let $S$ be an scaled simplicial set. Then the bicategorical straightening-unstraightening adjunction defines a Quillen equivalence 
  	\[
  	\ST_S: (\mbsSet)_{/S} \llra \left(\mssSet\right)^{\mathfrak{C}[S]^\op}:\UN_S
  	\]
  	between the model structure on (outer) 2-Cartesian fibrations over $S$ and the projective model structure on $\msSet$-enriched functors $\mathfrak{C}[S]^\op \to \mssSet$ with values in marked-scaled simplicial sets.
  \end{thm*} 
  
  Observe that both model categories are in fact $\on{Set}^+_\Delta$-enriched categories. After performing elementary explicit verifications we prove that the functor $\UN_S$ is compatible with the (co)tensoring yielding an upgrade of the previous theorem to an intrinsinc bicategorical result.
  
  \begin{thm*}[\autoref{cor:intrinsincbicat}]
  	The bicategorical straightening is a left Quillen equivalence for any scaled simplicial set $S$. Moreover, the functor $\UN_S$ provides an equivalence of $(\infty,2)$-categories 
  	\[
  	2\bcat{C}\!\on{art}(S)\simeq \Fun(S^\op,\bcat{B}\!\on{icat}_\infty).
  	\]
  \end{thm*}
  
  The majority of this paper is thus devoted to this proof. We will recapitulate the major ideas of the proof, as well as the structure of the paper, in the penultimate section of the introduction.
  
  \subsection{A relative 2-nerve}
  
  Although it is desirable to have a bicategorical Grothendieck construction that works in the most the general context possible, 
  many practical applications make use of those $\infty$-bicategories which arise as scaled nerves of strict 2-categories. We provide a version of the Grothendieck construction better suited to this particular situation in the final section of this paper. In this context, we define an explicit version $\bbchi_\CC$ of the unstraightening functor over $\Nsc(\CC)$, which we call the \emph{relative 2-nerve}. 
  
  \begin{thm*}[\autoref{cor:relative2}]
  	Let $\mathbb{C}$ be an strict 2-category. Then there is a Quillen equivalence
  	\[
  	\PPhi_{\mathbb{C}}: (\mbsSet)_{/\Nsc(\mathbb{C})}  \llra \left(\mssSet\right)^{\mathbb{C}^{(\op,-)}}:\cchi_{\mathbb{C}}
  	\]
  	and an equivalence of left-derived functors $\on{L}\!\ST_{\mathbb{C}} \xRightarrow{\simeq}  \on{L}\!\PPhi_{\mathbb{C}}$.
  \end{thm*}
  
  As in the $(\infty,1)$-categorical setting (see Section 3.2.5 in \cite{HTT}) the benefits of a relative nerve construction are twofold: on the one hand, the relative 2-nerve is particularly computationally tractable and well-suited to explicit examples; on the other, the relative 2-nerve allows us to compare our $\infty$-bicategorical Grothendieck construction to preexisting strict Grothendieck constructions. We apply our relative nerve construction to obtain a comparison with the Grothendieck construction appearing in \cite{Buckley}. The strict 2-categorical Grothendieck construction of \cite{Buckley} takes the form of an equivalence 
  \[
  \func{
  	\mathbb{E}\!\on{l}: \Fun^{\on{ps}}(\CC^\op,\Cat_2)\to 2\on{Cart}
  }
  \]
  for a 2-category $\CC$. The final result of the paper shows that relative 2-nerve coincides with $\mathbb{E}\!\on{l}$ for every strict 2-functor with values in $2$-categories.
  
  \begin{thm*}[\autoref{thm:BuckleyComp}]
  	Let 
  	\[
  	\func{F:\CC^{(\op,-)}\to 2\!\Cat}
  	\]
  	be a 2-functor, and let $\tilde{F}$ denote the composite 
  	\[
  	\func{\CC^{(\op,-)}\to 2\!\Cat\to \Set_\Delta^{\mathbf{ms}}}
  	\]
  	Then there is an equivalence 
  	\[
  	\begin{tikzcd}
  		\cchi_\CC(\tilde{F}) \arrow[dr]\arrow[rr, rightarrow, "\simeq"] & &\arrow[dl] (\Nerv_2(\mathbb{E}\!\on{l}(F)),M,T\subset C)\\
  		& \Nsc(\CC)& 
  	\end{tikzcd}
  	\]
  	of 2-Cartesian fibrations over $\Nsc(\CC)$.  
  \end{thm*}
  
  \subsection{The zoo of Grothendieck constructions} 
  
  To aid the reader in connecting our work here to other results in the literature, we here provide a brief overview of existing Grothendieck constructions, as well as known connections among them. Where practicable, we will choose the version of the construction with the correct variance to agree with our construction.
  
  \begin{itemize}
  	\item The classical Grothendieck construction of \cite{SGA1} takes the form of a equivalence 
  	\[
  	\func{\on{El}: \Fun^{\on{ps}}(C^\op,\Cat) \to \on{Cart}(C)}
  	\]
  	for a 1-category $C$
  	\item The classical Grothendieck construction is often restricted to categories fibred in groupoids, in which case it takes the form of an equivalence 
  	\[
  	\func{\on{el}: \Fun^{\on{ps}}(C^\op,\on{Grpd}) \to \on{Rfib}(C)}
  	\]
  	for a 1-category $C$.
  	\item The strict 2-categorical Grothendieck construction of \cite{Buckley} takes the form of an equivalence 
  	\[
  	\func{
  		\mathbb{E}\!\on{l}: \Fun^{\on{ps}}(\CC^\op,\Cat_2)\to 2\on{Cart}
  	}
  	\]
  	for a 2-category $\CC$. 
  	\item The three Grothendieck-Lurie constructions: 
  	\begin{itemize}
  		\item For $(\infty,1)$-categories fibred in $\infty$-groupoids, the construction of \cite[Ch 2]{HTT} takes the form of a left Quillen equivalence
  		\[
  		\func{
  			\St^{}_S:(\Set_\Delta)_{/S}\to (\Set_\Delta)^{\mathfrak{C}[S]^\op} 
  		}
  		\]
  		for $S$ a simplicial set. Here $(\Set_\Delta)_{/S}$ is equipped with the model structure for right fibrations, and $(\Set_\Delta)^{\mathfrak{C}[S]^\op}$ is equipped with the projective model structure obtained from the Kan-Quillen model structure. 
  		\item For $(\infty,1)$-categories fibred in $(\infty,1)$-categories, the construction of \cite[Ch. 3]{HTT} takes the form of a left Quillen equivalence
  		\[
  		\func{
  			\St^+_S: (\msSet)_{/S} \to (\mssSet)^{\mathfrak{C}[S]^\op}  
  		}
  		\]
  		where $S$ is a simplicial set. Here $(\msSet)_{/S}$ carries the Cartesian model structure and $(\msSet)^{\mathfrak{C}[S]^\op} $ carries the projective model structure on $\Set_\Delta$-enriched functors. 
  		\item For $\infty$-bicategories fibred in $(\infty,1)$-categories, the construction of \cite{LurieGoodwillie} takes the form of a left Quillen equivalence 
  		\[
  		\func{
  			\St_S^{(2,1)}: (\msSet)_{/S}\to (\Set_\Delta^+)^{\mathfrak{C}^{\sc}[S]^{(\op,\op)}} 
  		}
  		\] 
  		where $S$ is a scaled simplicial set. Here $(\msSet)_{/S}$ carries the \emph{$\mathfrak{P}$-anodyne model structure} of \cite[Section 3.2]{LurieGoodwillie} and $(\Set_\Delta^+)^{\mathfrak{C}^{\sc}[S]^{(\op,\op)}}$ carries the projective model structure on $\msSet$-enriched functors.
  	\end{itemize}
  	\item A general $(\infty,n)$-categorical Grothendieck construction is given by Rasekh in \cite[Section 5]{RasekhDYoneda}, using $\Theta_n$-spaces. This Grothendieck construction takes the form of a zig-zag of Quillen equivalences, which, given an $(\infty,n+1)$-category $\C$ presented as a complete Segal object in $\Theta_n$-spaces, induces an equivalence
  	\[
  	\func{\on{E}_\Theta: \on{Cart}_{(\infty,n)}(\C) \to[\sim] \Fun_{(\infty,n+1)}(\C,\Cat_{(\infty,n)})}
  	\]
  	between $(\infty,n)$-Cartesian fibrations over $\scr{C}$, and $(\infty,n+1)$-functors $\scr{C}\to \Cat_{(\infty,n)}$. 
  	\item The \emph{comprehension construction}, defined by Riehl and Verity in \cite{RiehlVerityComp} works in an $\infty$-cosmos $\scr{K}$. In the $\infty$-cosmos of quasi-categories, it provides a functor 
  	\[
  	\func{\on{Comp}:\Fun(B,\on{qCat})\to \on{coCart}(\on{qCat})_{/B}}
  	\]
  	from functors into small quasi-categories to coCartesians fibrations over $B$. However, the authors defer the proof that this is an equivalence to a latter work. 
  \end{itemize}
  
  We summarize the known relations relations between these constructions in the following diagram. An arrow in the diagram represents a special case, e.g. $\on{El}\to \St_S^+$ means that $\on{El}$ is known to be equivalent to a special case of $\St_S^+$. A dashed arrow will represent a relation which we conjecture to hold. 
  
  In particular: 
  \begin{enumerate}
  	\item These relations are proven in \cite{HTT}. In particular, the comparison of the $(\infty,1)$-categorical and strict cases passes through the \emph{relative nerve} of section 3.2.5. 
  	\item In \cite[Remark 4.5.10]{LurieGoodwillie} the author claims that due to the formal differences in the construction of both straightening functors, no direct comparison seems possible. Instead, the author proves a comparison on fibrant object without showing naturality in \cite[Prop. 4.5.10]{LurieGoodwillie} using model-independent arguments.
  	
  	However, we attribute the difficulty of a comparison to the fact that both straightening functors have different variances. It follows from our construction (see point 4 on this list) of $\ST_S$ that $\St^+_S$  morally has an outer Cartesian variance. Since $\St^+_S$ models the Grothendieck construction for $\infty$-categories this construction is blind to the variance in 2-morphisms and it is seen as having simply a Cartesian variance. The construction of \cite[Prop. 4.5.9]{LurieGoodwillie} passes through two 1-morphism dualizations to obtain a Cartesian variant of $\St^{(2,1)}_S$. We thus believe a more complete comparison result with the $\St_S^{(2,1)}$ after taking the pertinent 2-morphism dual, as well.
  	\item We show this relation in \autoref{thm:BuckleyComp}, making use of a relative 2-nerve which we construct for that purpose in the final section. 
  	\item This relation is nearly immediate from the definitions. We give a proof in \autoref{prop:infty1comparison}. 
  	\item We believe that this relation will hold, however, the constructions $\on{St}_S^{(2,1)}$ and $\ST_S$ are related by a 2-morphism dualization. Given the difficulties inherent in realizing 2-morphism duals in scaled simplicial sets, we defer any attempt to prove this statement for the time being. 
  \end{enumerate}
  
  
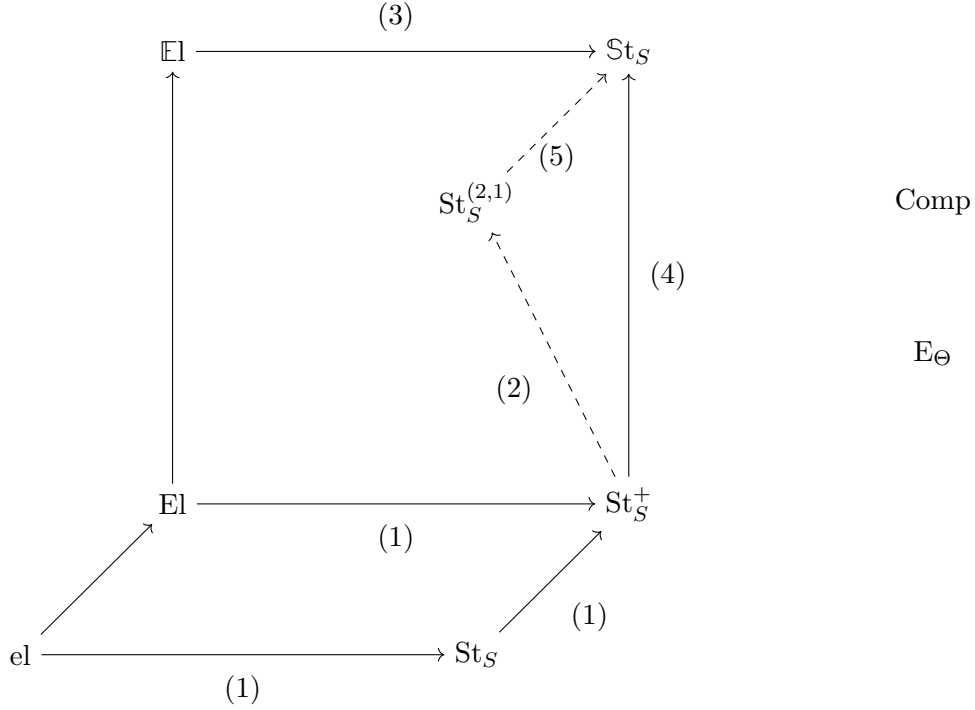
\begin{center}
  	\begin{tikzpicture}
  		\path (0,0) node (EL) {$\mathbb{E}\!\on{l}$}; 
  		\path (6,0) node (ST) {$\ST_S$};
  		\path (4,-2) node (ST21) {$\St_S^{(2,1)}$};
  		\path (0,-6) node (El) {$\on{El}$}; 
  		\path (6,-6) node (Stp) {$\St^+_S$};
  		\path (4,-8) node (Stgrpd) {$\St_S$}; 
  		\path (-2,-8) node (el) {$\on{el}$}; 
  		\path (10,-2) node (comp) {$\on{Comp}$};
  		\path (10,-4) node (ET) {$\on{E}_\Theta$};
  		\draw[->] (Stgrpd) to node[label=below right:(1)] {} (Stp);
  		\draw[->,dashed] (Stp) to node[label=below left:(2)] {} (ST21);
  		\draw[->] (el) to (El);
  		\draw[->] (El) to (EL);
  		\draw[->] (El) to node[label=below:(1)] {} (Stp);
  		\draw[->] (el) to node[label=below:(1)] {} (Stgrpd);
  		\draw[->] (EL) to node[label=above:(3)] {} (ST); 
  		\draw[->] (Stp) to node[label=right:(4)] {} (ST);
  		\draw[->,dashed] (ST21) to node[label=below:(5)] {} (ST);
  	\end{tikzpicture}
  \end{center}

  \subsection{Structure of the paper}
  
  In Section 2, we provide some background on the model structures and constructions the paper will use. In particular, we describe model structures on scaled, marked-scaled, and marked-biscaled simplicial sets and show that they define Quillen equivalent models for $\infty$-bicategories. We also recall the model structure for 2-Cartesian fibrations from our previous paper \cite{AGS_CartI}. 
  
  Section 3 contains the main results and technical arguments of the paper, in particular the $(\infty,2)$-categorical Grothendieck construction. We first construct the adjunction which will define the Grothendieck construction, and show that it is a Quillen adjunction. To show that our straightening functor $\ST_S$ is a Quillen equivalence, we perform the kind of dimensional induction used in \cite[Section 3.2]{HTT}: 
  \begin{itemize}
  	\item We prove that $\ST_{\Delta^0}$ is a left Quillen equivalence in Section \ref{subsec:STequivPt}. The proof of this fact is quite direct, and proceeds by constructing a natural equivalence from $\ST_{\Delta^0}$ to a more canonical left Quillen equivalence $\func{L:\mbsSet\to \mssSet}$ (which is defined in \autoref{subsec:MSsSet}).
  	\item We then prove that $\ST_{\Delta^n_\flat}$ is a left Quillen equivalence in \autoref{subsec:STequivSimp}. This is the most technically demanding step in the proof. We first show that, for any 2-Cartesian fibration $p:X\to \Delta^n_\flat$, there is a homotopy pushout diagram 
  	\[
  	\begin{tikzcd}
  		X\times (\Delta^{n-1})^\diamond \arrow[r]\arrow[d] & X_n \times (\Delta^n)^\diamond\arrow[d] \\
  		X\times_{\Delta^{n}} (\Delta^{n-1})^\diamond\arrow[r] & X
  	\end{tikzcd}
  	\]
  	in $(\mbsSet)_{/\Delta^n_\flat}$. This statement appears as \autoref{cor:hmtpy_pushout_mapsimp}, the proof of which occupies most of Section \ref{subsec:STequivSimp}. 
  	
  	Once the homotopy pushout is established, we can show that, for a 2-Cartesian fibration $X\to \Delta_\flat$ with $i$-fibre $X_i$, the canonical map $\ST_{\Delta^0}\to \ST_{\Delta^n_\flat}(X)(i)$ is an equivalence. We then use the fibrewise nature of equivalences in $(\mbsSet)_{/\Delta^n_\flat}$ together with the pointwise nature of equivalences in $(\mssSet)^{\mathfrak{C}^{\sc}[\Delta^n_\flat]^\op}$ to complete the proof.  	  
  	\item The general case follows from the special cases over simplices by an inductive argument nearly identical to that of \cite[Prop. 3.8.4]{LurieGoodwillie}. 
  \end{itemize} 
  
  The final section of the paper, Section 4, is devoted to a relative nerve construction, and a comparison with the strict 2-categorical constructions of \cite{Buckley}.

  \subsection*{Acknowledgements}
  
  F.A.G. would like to acknowledge the support of the VolkswagenStiftung through the Lichtenberg Professorship Programme while he conducted this research. W.H.S. wishes to acknowledge the support of the NSF Research Training Group at the
  University of Virginia (grant number DMS-1839968) during the preparation of this work. We are very grateful to Tobias Dyckerhoff for the careful revision of this draft and the many improvements suggested. F.A.G would also like to thank Tobias Dyckerhoff for suggesting studying 2-categorical notions of cofinality as one of the main topics of his PhD thesis and for all of the guidance offered during the process.

\newpage 

\section{Preliminaries}\label{sec:Preliminaries}

We here collect background information and preliminary results which will be necessary for our arguments in the rest of the paper. While we will review some 2-category theory, it is impracticable to recapitulate all of the needed background on $(\infty,2)$-categories, so we will limit ourselves to a brief discussion of scaled simplicial sets, and direct the reader to \cite{LurieGoodwillie} for more comprehensive background. 

\subsection{2-categories}

We will often use strict 2-categories as a tool to explore $\infty$-bicategories. In this section, we briefly collect some notations and constructions we will use in the sequel.

\begin{notation} 
	By a \emph{2-category}, we will always mean a strict 2-category. By a \emph{2-functor}, we will mean a strict 2-functor unless specified otherwise. We will denote strict 2-categories by blackboard bold letters, e.g. $\DD$. 
	
	A 2-category $\DD$ has three duals, determined by reversing the direction of 1-morphisms, 2-morphisms, or both, respectively. In this work, we will primarily make use of the 1-morphism dual. To better accord with the notation used in (scaled) simplicial sets, we will denote the 1-morphism dual simply by $\DD^\op$. Where needed, we denote the 2-morphism dual by $\DD^{(-,\op)}$, and the dual which reverses both 1- and 2-morphisms by $\DD^{(\op,\op)}$. We will denote the 1-category of 2-categories and strict 2-functors by $2\!\Cat$. 
\end{notation}


\begin{definition}\label{def:OI}
	Let $I$ be a linearly ordered finite set. We define a $2$-category $\OO^{I}$ \glsadd{OI} as
	follows
	\begin{itemize}
		\item the objects of $\OO^I$ are the elements of $I$,
		\item the category $\mathbb{O}^{I}(i,j)$ of morphisms between objects $i,j \in I$ is defined
		as the poset of finite sets $S \subseteq I$ such that $\min(S)=i$ and $\max(S)=j$
		ordered by inclusion,
		\item the composition functors are given, for $i,j,l\in I$, by
		\[
		\mathbb{O}^{I}(i,j) \times \mathbb{O}^{I}(j,l) \to \mathbb{O}^{I}(i,l), \quad (S,T) \mapsto S \cup T.
		\]
	\end{itemize}
	When $I=[n]$, we denote $\OO^I$ by $\OO^n$. Note that the $\OO^n$ form a cosimplicial object in $2\!\Cat$, which we denote by $\OO^\bullet$. 
\end{definition}

\begin{definition}\label{defn:upslice}
	Let $f:\CC \to \DD$ be a functor of 2-categories. Given an object $d \in \DD$ we define the lax slice $\CC_{d\upslash}$ 2-category as follows:
	\begin{itemize}[noitemsep]
		\item Objects are morphisms $u:d \to f(c)$ in $\DD$ with source $d$ such that $c \in \CC$.
		\item A 1-morphism from $u: d \to f(c)$ to $v:d \to f(c')$ is given by a 1-morphism $\alpha: c \to c'$ in $\CC$ and a 2-morphism $f(\alpha) \circ u \xRightarrow{} v$.
		\item A 2-morphism in $\CC_{d\upslash}$ is given by a 2-morphism $\epsilon: \alpha  \nat \beta$ such that the diagram below commutes
		\[
		\begin{tikzcd}
			f(\alpha) \circ u \arrow[dd, Rightarrow,"f(\epsilon)*u",swap] \arrow[rrd, Rightarrow] &  &   \\
			&  & v \\
			f(\beta) \circ u \arrow[rru, Rightarrow]                        &  &  
		\end{tikzcd}
		\]
		If the functor $f$ is the identity on the 2-category $\DD$ we will use the notation $\DD_{d\upslash}$.
	\end{itemize}
\end{definition}

\begin{example}
	Let $I$ be a linearly ordered finite set and denote its minimum by $i$. Unraveling \autoref{defn:upslice} we see that $\OO^{I}_{i \upslash}$ is the 2-category given by:
	\begin{itemize}[noitemsep]
		\item Objects are subsets $S \subseteq I$ such that $\min(S)=i$.
		\item A morphism $S \to T$ is given by a subset $U \subseteq I$ such that $\max(S)=\min(U)$ and $\max(U)=\max(T)$ and such that $S \cup U \subseteq T$.
		\item We have a 2-morphism $U \xRightarrow{} V$ precisely if $U \subseteq V$.
	\end{itemize}
\end{example}

\begin{remark}\label{rem:initialmapping}
	We observe that the non-empty mapping categories in $\OO^I_{i\upslash}$ are all contractible since each has an initial object.
\end{remark}

To represent 2-categories as simplicial sets, we need a nerve operation. We denote the category of simplicial sets by $\Set_\Delta$. 

\begin{definition}\label{def:Duskin}
	Given $\DD\in 2\!\Cat$, we define a simplicial set $\Nerv_2(\DD)\in \Set_\Delta$, the \emph{Duskin nerve} of $\DD$, by 
	\[
	\Nerv_2(\DD)_n:= 2\!\Cat(\OO^n,\DD).
	\]
\end{definition}

\subsection{Higher categories and decorated simplicial sets}

Throughout this paper, we will make extensive use of models for higher categories in terms of decorated simplicial sets. For models for $(\infty,1)$-categories, we will direct the reader to \cite[\S 2.2.5, \S 3.1]{HTT}, though we briefly discuss notation here. 

\begin{definition}
	We denote the category of simplicial sets by $\Set_\Delta$. A \emph{marked simplicial set} is defined to be a pair $(X,M_X)$ consisting of a simplicial set $X\in \Set_\Delta$, and a collection $M_X\subseteq X_1$ of 1-simplices in $X$ which contains the degenerate 1-simplices. We will denote the category of marked simplicial sets by $\msSet$. 
	
	We will typically view $\Set_\Delta$ as equipped with either the \emph{Kan-Quillen model structure} (see, e.g., \cite[Ch. 1]{GoerssJardine}) or the \emph{Joyal model structure} (see, e.g. \cite[\S 2.2.5]{HTT}). We will view $\msSet$ as equipped with the \emph{Cartesian model structure} of \cite[\S 3.1]{HTT}.
\end{definition}

The first model structure we will use to study $\infty$-bicategories is a model structure on enriched categories.

\begin{notation}
	We denote by $\Cat_\Delta^+$ the category of $\msSet$-enriched categories. 
\end{notation}

\begin{proposition}
	There is a left-proper, combinatorial model structure on $\Cat_\Delta^+$ such that 
	\begin{itemize}
		\item[W] The weak equivalences are those enriched functors which are essentially surjective on homotopy categories and induce equivalences on all mapping spaces.
		\item[C] The cofibrations the smallest weakly saturated class containing $\varnothing \to [0]_{\msSet}$, and each inclusion $[1]_A\to [1]_{B}$ where $A\to B$ is a generating cofibration for $\msSet$. 
	\end{itemize}
\end{proposition}

\begin{proof}
	This is a special case of \cite[A.3.2.4]{HTT}.
\end{proof}

\begin{remark}\label{rmk:2Cat_as_msSetCat}
	Notice that, given a strict 2-category $\DD$, we can take the nerves of the mapping categories to obtain a $\Set_\Delta$-enriched category. If we define a marking on each mapping category by declaring precisely the isomorphisms to be marked, we obtain a canonical element of $\Cat_\Delta^+$ associated to $\DD$. We will uniformly abuse notation by denoting this $\msSet$-enriched category by $\DD$ as well. 
\end{remark}

The basic idea for the main models for $\infty$-bicategories used in this paper is that simplicial sets may be used to model $\infty$-bicategories, provided we keep track of which 2-simplices are considered to represent invertible 2-morphisms. 

\begin{definition}
	A \emph{scaled simplicial set} \glsadd{scsSet} consists of a pair $(X,T_X)$, where $X\in \Set_\Delta$ is a simplicial set, and $T_X\subseteq X_2$ is a collection of 2-simplices --- called the \emph{thin} 2-simplices --- which contains all degenerate 2-simplices. We denote by $\scsSet$ the category of scaled simplicial sets.  
\end{definition}

\begin{notation}
	We will sometimes make use of subscripts to denote scalings. In particular, $(X,X_2)$ will be will be denoted by $X_\sharp$, and $(X,\on{deg}(X_2))$ will be denoted by $X_\flat$. Similarly, we will sometimes use superscripts to denote markings on a simplicial set: $(X,X_1)=X^\sharp$, and $(X,\on{deg}(X_1))=X^\flat$. 
\end{notation}

\begin{definition}\label{def:scanodyne}
	The set of \emph{generating scaled anodyne maps} \(\sS\) is the set of maps of scaled simplicial sets consisting of:
	\begin{enumerate}
		\item[(i)]\label{item:anodyne-inner} the inner horns inclusions
		\[
		\bigl(\Lambda^n_i,\{\Delta^{\{i-1,i,i+1\}}\}\bigr)\rightarrow \bigl(\Delta^n,\{\Delta^{\{i-1,i,i+1\}}\}\bigr)
		\quad , \quad n \geq 2 \quad , \quad 0 < i < n ;
		\]
		\item[(ii)]\label{i:saturation} the map 
		\[
		(\Delta^4,T)\rightarrow (\Delta^4,T\cup \{\Delta^{\{0,3,4\}}, \ \Delta^{\{0,1,4\}}\}),
		\]
		where we define
		\[
		T\overset{\text{def}}{=}\{\Delta^{\{0,2,4\}}, \ \Delta^{\{ 1,2,3\}}, \ \Delta^{\{0,1,3\}}, \ \Delta^{\{1,3,4\}}, \ \Delta^{\{0,1,2\}}\};
		\]
		\item[(iii)]\label{item:anodyne_outer} the set of maps
		\[
		\Bigl(\Lambda^n_0\coprod_{\Delta^{\{0,1\}}}\Delta^0,\{\Delta^{\{0,1,n\}}\}\Bigr)\rightarrow \Bigl(\Delta^n\coprod_{\Delta^{\{0,1\}}}\Delta^0,\{\Delta^{\{0,1,n\}}\}\Bigr)
		\quad , \quad n\geq 3.
		\]
	\end{enumerate}
	A general map of scaled simplicial set is said to be \emph{scaled anodyne} if it belongs to the weakly saturated closure of \(\sS\).
\end{definition}

\begin{definition}
	We say that a map of scaled simplicial sets $p:X \to S$ is a weak \sS-fibration if it has the right lifting property with respect to the class of scaled anodyne maps. We call $X\in \scsSet$ a \emph{$\infty$-bicategory} if the unique map $X\to \Delta^0_\sharp$ is a weak \sS-fibration. 
\end{definition}

\begin{definition}
	The composite
	\[
	\begin{tikzcd}
		\Delta \arrow[r,"\mathfrak{C}"] & {\Cat_\Delta} \arrow[r,"\flat"]& {\Cat_\Delta^+}
	\end{tikzcd}
	\]
	gives us a cosimplicial object in $\Cat_\Delta^+$. We can moreover send the thin 2-simplex $\Delta^2_\sharp$ to $\mathfrak{C}[\Delta^2]$ equipped with sharp-marked mapping spaces. The usual machinery of nerve and realization then gives us adjoint functors  
	\[
	\begin{tikzcd}
		\mathfrak{C}^{\sc}: &[-3em] \scsSet \arrow[r,shift left] & \Cat_\Delta^+ \arrow[l,shift left] &[-3em] :\Nsc 
	\end{tikzcd}
	\]
	which we will call the \emph{scaled nerve} and \emph{scaled rigidification}. \glsadd{ScNerve}
\end{definition}

\begin{remark}
	Given a 2-category $\DD$, we can define a scaling on $N_2(\DD)$ by declaring a triangle to be thin if and only if the corresponding 2-morphism is invertible. Viewing $\DD$ as an $\msSet$-enriched category as in \autoref{rmk:2Cat_as_msSetCat}, we see that this scaled simplicial set coincides with $\Nsc(\DD)$. We thus are justified in speaking of the \emph{scaled nerve} of a 2-category.
\end{remark}

\begin{theorem}
	There is a left proper, combinatorial model structure on $\scsSet$ with
	\begin{itemize}
		\item[W] The weak equivalences are the morphisms $f:A\to B$ such that $\mathfrak{C}^{\sc}[f]:\mathfrak{C}^{\sc}[A]\to \mathfrak{C}^{\sc}[B]$ is an equivalence in $\Cat_\Delta^+$. 
		\item[C] The cofibrations are the monomorphisms. 
	\end{itemize}
	Moreover, the fibrant objects in this model structure are the $\infty$-bicategories, and the adjunction 
	\[
	\begin{tikzcd}
		\mathfrak{C}^{\sc}: &[-3em] \scsSet \arrow[r,shift left] & \Cat_\Delta^+ \arrow[l,shift left] &[-3em] :\Nsc 
	\end{tikzcd}
	\]
	is a Quillen equivalence.
\end{theorem}
\begin{proof}
	This is \cite[Thm A.3.2.4]{LurieGoodwillie}. The characterization of fibrant objects is \cite[Thm 5.1]{GHL_Equivalence}.
\end{proof}

\begin{remark}[Key notational convention]
	When \glsadd{GenCatNotation} no confusion is likely to arise, we denote decorated simplicial sets by simple roman majescules: $X$, $Y$, $Z$, etc. For fibrant objects, representing (higher) categories of various kinds, we fix the following conventions:
	\begin{itemize}
		\item Strict 1-categories will be denoted by undecorated roman majescules, $B$, $C$, $D$, etc.
		\item $(\infty,1)$-categories as presented by Joyal fibrant simplicial sets of fibrant marked simplicial sets will be denoted by calligraphic majescules: $\scr{B}$, $\scr{C}$, $\scr{D}$, etc. 
		\item Strict 2-categories will be denoted by blackboard-bold majescules, $\BB$, $\CC$, $\DD$, etc. 
		\item $(\infty,2)$-categories, presented as fibrant scaled simplicial sets, whill be denoted by thickened blackboard-bold majescules: $\bcat{B}$, $\bcat{C}$, $\bcat{D}$, etc.
	\end{itemize}
\end{remark}

\subsection{Marked biscaled simplicial sets and 2-Cartesian fibrations.}\label{subsec:MSsSet}
In this section we collect useful definitions and results introduced in \cite{AGS_CartI} that will play a relevant role in this paper.
\begin{definition}
	A \emph{marked biscaled} simplicial set (mb simplicial set) is given by the following data
	\begin{itemize}
		\item A simplicial set $X$.
		\item A collection of edges  $E_X \in X_1$ containing all degenerate edges.
		\item A collection of triangles $T_X \in X_2$ containing all degenerate triangles. We will refer to the elements of this collection as \emph{thin triangles}.
		\item A collection of triangles $C_X \in X_2$ such that $T_X \subseteq C_X$. We will refer to the elements of this collection as \emph{lean triangles}.
	\end{itemize}
	We will denote such objects as triples $(X,E_X, T_X \subseteq C_X)$. A map $(X,E_X, T_X \subseteq C_X) \to (Y,E_Y,T_Y \subseteq C_Y)$ is given by a map of simplicial sets $f:X \to Y$ compatible with the collections of edges and triangles above. We denote by $\mbsSet$ the category of mb simplicial sets. \glsadd{MBsSet}
\end{definition}

\begin{notation}
	Let $(X,E_X, T_X \subseteq C_X)$ be a mb simplicial set. Suppose that the collection $E_X$ consist only of degenerate edges. Then we fix the notation $(X,E_X, T_X \subseteq C_X)=(X,\flat,T_X \subseteq E_X)$ and do similarly for the collection $T_X$. If $C_X$ consists only of degenerate triangles we fix the notation $(X,E_X, T_X \subseteq C_X)=(X,E_X, \flat)$. In an analogous fashion we wil use the symbol “$\sharp$“ to denote a collection containing all edges (resp. all triangles). Finally suppose that $T_X=C_X$ then we will employ the notation $(X,E_X,T_X)$.
\end{notation}

\begin{remark}
	We will often abuse notation when defining the collections $E_X$ (resp. $T_X$, resp. $C_X$) and just specified its non-degenerate edges (resp. triangles).
\end{remark}

\begin{definition}\label{def:mbsanodyne}
	The set of \emph{generating mb anodyne maps} \(\bS\) \glsadd{MBann} is the set of maps of mb simplicial sets consisting of:
	\begin{enumerate}
		\myitem{(A1)}\label{mb:innerhorn} The inner horn inclusions 
		\[
		\bigl(\Lambda^n_i,\flat,\{\Delta^{\{i-1,i,i+1\}}\}\bigr)\rightarrow \bigl(\Delta^n,\flat,\{\Delta^{\{i-1,i,i+1\}}\}\bigr)
		\quad , \quad n \geq 2 \quad , \quad 0 < i < n ;
		\]
		\myitem{(A2)}\label{mb:wonky4} The map 
		\[
		(\Delta^4,\flat,T)\rightarrow (\Delta^4,\flat,T\cup \{\Delta^{\{0,3,4\}}, \ \Delta^{\{0,1,4\}}\}),
		\]
		where we define
		\[
		T\overset{\text{def}}{=}\{\Delta^{\{0,2,4\}}, \ \Delta^{\{ 1,2,3\}}, \ \Delta^{\{0,1,3\}}, \ \Delta^{\{1,3,4\}}, \ \Delta^{\{0,1,2\}}\};
		\]
		\myitem{(A3)}\label{mb:leftdeglefthorn} The set of maps
		\[
		\Bigl(\Lambda^n_0\coprod_{\Delta^{\{0,1\}}}\Delta^0,\flat,\flat \subset\{\Delta^{\{0,1,n\}}\}\Bigr)\rightarrow \Bigl(\Delta^n\coprod_{\Delta^{\{0,1\}}}\Delta^0,\flat,\flat \subset\{\Delta^{\{0,1,n\}}\}\Bigr)
		\quad , \quad n\geq 2.
		\]
		These maps force left-degenerate lean-scaled triangles to represent coCartesian edges of the mapping category.
		\myitem{(A4)}\label{mb:2Cartesianmorphs} The set of maps
		\[
		\Bigl(\Lambda^n_n,\{\Delta^{\{n-1,n\}}\},\flat \subset \{ \Delta^{\{0,n-1,n\}} \}\Bigr) \to \Bigl(\Delta^n,\{\Delta^{\{n-1,n\}}\},\flat \subset \{ \Delta^{\{0,n-1,n\}} \}\Bigr) \quad , \quad n \geq 2.
		\]
		This forces the marked morphisms to be $p$-Cartesian with respect to the given thin and lean triangles. 
		\myitem{(A5)}\label{mb:2CartliftsExist} The inclusion of the terminal vertex
		\[
		\Bigl(\Delta^{0},\sharp,\sharp \Bigr) \rightarrow \Bigl(\Delta^1,\sharp,\sharp \Bigr).
		\]
		This requires $p$-Cartesian lifts of morphisms in the base to exist.
		\myitem{(S1)}\label{mb:composeacrossthin} The map
		\[
		\Bigl(\Delta^2,\{\Delta^{\{0,1\}}, \Delta^{\{1,2\}}\},\sharp \Bigr) \rightarrow \Bigl(\Delta^2,\sharp,\sharp \Bigr),
		\]
		requiring that $p$-Cartesian morphisms compose across thin triangles.
		\myitem{(S2)}\label{mb:coCartoverThin} The map
		\[
		\Bigl(\Delta^2,\flat,\flat \subset \sharp \Bigr) \rightarrow \Bigl( \Delta^2,\flat,\sharp\Bigr),
		\]
		which requires that lean triangles over thin triangles are, themselves, thin.
		\myitem{(S3)}\label{mb:innersaturation} The map
		\[
		\Bigl(\Delta^3,\flat,\{\Delta^{\{i-1,i,i+1\}}\}\subset U_i\Bigr) \rightarrow \Bigl(\Delta^3,\flat, \{\Delta^{\{i-1,i,i+1\}}\}\subset \sharp \Bigr) \quad, \quad 0<i<3
		\]
		where $U_i$ is the collection of all triangles except $i$-th face. This and the next two generators serve to establish composability and limited 2-out-of-3 properties for lean triangles.
		\myitem{(S4)}\label{mb:dualcocart2of3} The map
		\[
		\Bigl(\Delta^3 \coprod_{\Delta^{\{0,1\}}}\Delta^0,\flat,\flat \subset U_0\Bigr) \rightarrow \Bigl(\Delta^3 \coprod_{\Delta^{\{0,1\}}}\Delta^0,\flat, \flat \subset \sharp \Bigr) 
		\]
		where $U_0$ is the collection of all triangles except the $0$-th face.
		\myitem{(S5)}\label{mb:coCart2of3} The map
		\[
		\Bigl(\Delta^3,\{\Delta^{\{2,3\}}\},\flat \subset U_3\Bigr) \rightarrow \Bigl(\Delta^3,\{\Delta^{\{2,3\}}\}, \flat \subset \sharp \Bigr) 
		\]
		where $U_3$ is the collections of all triangles except the $3$-rd face.
		\myitem{(E)}\label{mb:equivalences} For every Kan complex $K$, the map
		\[
		\Bigl( K,\flat,\sharp  \Bigr) \rightarrow \Bigl(K,\sharp, \sharp\Bigr).
		\]
		Which requires that every equivalence is a marked morphism.
	\end{enumerate}
	A map of mb simplicial sets is said to be \bS-anodyne if it belongs to the weakly saturated closure of \bS.
\end{definition}

\begin{definition}
	Let $f:(X,E_X,T_X \subseteq C_X) \to (Y,E_Y,T_Y \subseteq C_Y)$ be a map of mb simplicial sets. We say that $f$ is a \bS-fibration if it has the right lifting property against the class of \bS-anodyne morphisms.
\end{definition}

\begin{definition}\label{def:mappingbicats}
	Given two mb simplicial sets $(K,E_K,T_K \subseteq C_K), (X,E_X,T_X \subseteq C_X)$ we define  another mb simplicial set denoted by $\on{Fun}^{\mathbf{mb}}(K,X)$ and characterized by the following universal property
	\[
	\on{Hom}_{\mbsSet}\Bigr(A,\on{Fun}^{\mathbf{mb}}(K,X) \Bigl)\isom \Hom_{\mbsSet}\Bigr(A \times K,X  \Bigl).
	\]
\end{definition}

\begin{proposition}\label{prop:bsfibfun}
	Let $f:(X,E_X,T_X \subseteq C_X) \to (Y,E_Y,T_Y \subseteq C_Y)$ be a \bS-fibration. Then for every $K \in \mbsSet$ the induced morphism $\on{Fun}^{\mathbf{mb}}(K,X) \to \on{Fun}^{\mathbf{mb}}(K,Y)$ is a \bS-fibration.
\end{proposition}

\begin{definition}
	Let $f: Y \to S$ be a morphism of mb simplicial another map $g:X \to Y$. We define a mb simplicial set $\on{Map}_Y(K,X)$ by means of the pullback square
	\[
	\begin{tikzcd}[ampersand replacement=\&]
		\on{Map}_Y(K,X) \arrow[r] \arrow[d] \& \on{Fun}^{\mathbf{mb}}(K,X) \arrow[d] \\
		\Delta^0 \arrow[r,"g"] \& \on{Fun}^{\mathbf{mb}}(K,Y)
	\end{tikzcd}
	\]
	If $f:X \to Y$ is a \bS-fibration then it follows from the previous proposition that $\on{Map}_Y(K,X)$ is an $\infty$-bicategory.
\end{definition}

Let $S \in \on{Set}^{\mathbf{sc}}_{\Delta}$ for the rest of the section we will denote $(\mbsSet)_{/S}$ the category of mb simplicial set over $(S,\sharp,T_S \subset \sharp)$.

\begin{definition}\label{def:fibrantobjects}
	We say that an object $\pi:X \to S$ in $(\mbsSet)_{/S}$ is an \emph{outer} 2-\emph{Cartesian} fibration if it is a $\bS$-fibration.
\end{definition}

\begin{remark}
	We will frequently abuse notation and refer to outer 2-Cartesian as \emph{2-Cartesian fibrations}.
\end{remark}

\begin{definition}\label{def:underlyingmapping}
	Let $\pi:X \to S$ be a morphism of mb simplicial sets. Given an object $K\to S$, we define $\on{Map}^{\on{th}}_{S}(K,X)$ to be the  mb sub-simplicial set consisting only of the thin triangles. Note that if $\pi$ is a 2-Cartesian fibration this is precisely the underlying $\infty$-category of $\on{Map}_S(K,X)$. 
	
	We similarly denote by $\on{Map}^{\isom}_S(K,X)$ the mb sub-simplicial set consisting of thin triangles and marked edges. As before, we note that if $\pi$ is a 2-Cartesian fibration, the simplicial set $\on{Map}^{\isom}_S(K,X)$ can be identified with the maximal Kan complex in $\on{Map}_S(K,X)$.
\end{definition}

\begin{definition}
	We define a functor $\func{I: \on{Set}^+_{\Delta} \to \mbsSet}$ mapping a marked simplicial set $(K,E_K)$ to the mb simplicial set $(K,E_K,\sharp)$. If $K$ is maximally marked we adopt the notation $I^{+}(K^{\sharp})=K^{\sharp}_{\sharp}$
\end{definition}

\begin{remark}
	Note that we can endow the $(\mbsSet)_{/S}$ with the structure of a $\on{Set}_{\Delta}^{+}$-enriched category by means of $\on{Map}^{\on{th}}_{S}(\mathblank,\mathblank)$. In addition given $K \in \on{Set}_{\Delta}^+$ and $\pi:X \to S$ we define  $K \tensor X:= I(K) \times X$ equipped with a map to $S$ given by first projecting to $X$ and then composing with $\pi$. This construction shows that $(\mbsSet)_{/S}$ is tensored over $\on{Set}^+_{\Delta}$. One can easily show that $(\mbsSet)_{/S}$ is also cotensored over $\on{Set}^+_{\Delta}$.
	
	In a similar way one can use $\on{Map}^{\isom}_{S}(\mathblank,\mathblank)$ to endow $(\mbsSet)_{/S}$ with the structure of a $\on{Set}_{\Delta}$-enriched category. In this case the cotensor is given by $K \tensor X= I(K^{\sharp}) \times X$.
\end{remark}

\begin{definition}\label{def:weakequiv1}
	Let $\func{L \to[h] K \to[p] S}$ be a morphism in $(\mbsSet)_{/S}$. We say that $h$ is a cofibration when it is a monomorphism of simplicial sets. We will call $h$ a weak equivalence if for every 2-Cartesian fibration $\pi:X \to S$ the induced morphism
	\[
	\func{h^{*}:\on{Map}_S(K,X) \to \on{Map}_S(L,X)}
	\]
	is a bicategorical equivalence.
\end{definition}

For the convenience of the reader, we here recall the main result of \cite{AGS_CartI}:

\begin{theorem}[\cite{AGS_CartI} Theorem 3.38]\label{thm:MBModelStructure}
	Let $S$ be a scaled simplicial set. Then there exists a left proper combinatorial simplicial model structure on $(\mbsSet )_{/S}$, which is characterized uniquely by the following properties:
	\begin{itemize}
		\item[C)] A morphism $f:X \to Y$ in $(\mbsSet )_{/S}$ is a cofibration if and only if $f$ induces a monomorphism betwee the underlying simplicial sets.
		\item[F)] An object $X \in (\mbsSet )_{/S}$ is fibrant if and only if $X$ is a 2-Cartesian fibration. 
	\end{itemize}
\end{theorem}

\subsubsection{\bS-anodyne morphisms and dull subsets}
Before proceeding, we here record two variants of the pivot point trick \cite[Lem. 1.10]{AGS_Twisted} which will be of use later. 

\begin{definition}
	Let $\mathbb{P}(n)$ be the power set of $[n]$. Given $\mathcal{A} \subset \mathbb{P}(n)$ and $X \in \mathbb{P}(n)$ we say that $X\subset \mathcal{A}$ is \emph{$\mathcal{A}$-basal} if it contains precisely one element from each $S\in \mathcal{A}$. We denote the set of $\mathcal{A}$-basal sets by $\on{Bas}(\mathcal{A})$.
\end{definition}

\begin{definition}
	Given subset $\mathcal{A}\subset \mathbb{P}(n)$ such that $\emptyset \notin \mathcal{A}$, and a marked-biscaled simplex $(\Delta^n)^\dagger$, we define a marked-biscaled simplicial subset 
	\[
	(\mathcal{S}^{\mathcal{A}})^\dagger =\bigcup_{S\in\mathcal{A}} \Delta^{[n]\setminus S}.
	\] 
\end{definition}

\begin{definition}\label{def:innerdull}
	We call a subset $\mathcal{A}\subset \mathbb{P}(n)$ \emph{inner-dull} if the following conditions are satisfied
	\begin{enumerate}
		\item $\mathcal{A}$ does not contain $\varnothing$.
		\item There exists $0<i<n$ such that $i \notin S$ for every $S \in \mathcal{A}$.
		\item For any $S,T \in \mathcal{A}$, $S \cap T=\varnothing$.
		\item For every $\mathcal{A}$-basal set $X \in \mathbb{P}(n)$ there exists $u,v \in X$ such that $u<i < v$.
	\end{enumerate}
	We call the element $i$ in the second condition the pivot point.
\end{definition}

\begin{definition}
	Given an inner-dull subset $\mathcal{A} \subset \mathbb{P}(n)$, we define $\scr{M}_{\mathcal{A}}$ to be the set of subsets $X \in \mathbb{P}(n)$ satisfying:
	\begin{itemize}
		\item[A1)] $X$ contains the pivot point $i \in X$.
		\item[A2)] The simplex $\sigma_X:\Delta^X \to (\Delta^n)^{\dagger}$ does not factor through $ (\mathcal{S}^{\mathcal{A}})^\dagger$.
	\end{itemize}
	We define $\scr{M}_{\mathcal{A}}^j=\set{X \in \scr{M}_{\mathcal{A}}\given |X|=j}$. Note that those elements $X \in \scr{M}_{\mathcal{A}}$ of minimal cardinality are of the form $X_0 \cup \set{i}$ for $X_0 \in \on{Bas}(\mathcal{A})$.
\end{definition}

\begin{definition}
	Let $\mathcal{A} \subset \mathbb{P}(n)$ be an inner-dull subset with pivot point $i$. Given an $\mathcal{A}$-basal subset $X$ we denote by $l^X < u^X$ the pair of consecutive elements such that $l^X <i < u^X$.
\end{definition}

\begin{lemma}[The pivot trick]\label{lem:innerpivot}
	Let $\mathcal{A} \subset \mathbb{P}(n)$ be an inner-dull subset and let $(\Delta^n)^{\dagger}$ be a marked biscaled simplex. Suppose that the following conditions hold:
	\begin{enumerate}
		\item Every marked edge (resp. thin triangle) which does not contain the pivot point $i$ factors through $(\mathcal{S}^{\mathcal{A}})^\dagger$.
		\item For every $X \in \on{Bas}(\mathcal{A})$ and every $l^X\leq r <i < s \leq u^X$ the triangle $\set{r,i,s}$ is thin. 
		\item Let $\sigma=\set{a<b<c}$ be a lean simplex not containing the pivot point $i$. Then either $\sigma$ factors through $(\mathcal{S}^{\mathcal{A}})^\dagger$ or we have $a < i <c$ and the simplex $\sigma \cup \set{i}$ is fully lean scaled.
	\end{enumerate}
	Then the inclusion
	\[
	\func{(\mathcal{S}^{\mathcal{A}})^\dagger \to (\Delta^n)^\dagger}
	\]
	is in the weakly saturated hull of morphisms of type \ref{mb:innerhorn} and \ref{mb:innersaturation}.
\end{lemma}
\begin{proof}
	Observe that since $\mathcal{A}$ is inner-dull it follows that every $\mathcal{A}$-basal set has the same cardinality which we denote $\epsilon$. For every $\epsilon \leq j \leq n$ we define
	\[
	Y_j= Y_{j-1}\cup \bigcup\limits_{X \in \scr{M}_{\mathcal{A}}^j}\sigma_X
	\]
	where $Y_{\epsilon-1}=(\mathcal{S}^{\mathcal{A}})^\dagger$ and we view $\sigma_X$ as having the inherited decorations. This yields a filtration
	\[
	(\mathcal{S}^{\mathcal{A}})^\dagger \to Y_\epsilon \to \cdots \to Y_{n-1} \to (\Lambda^n_i)^{\dagger} \to (\Delta^n)^\dagger
	\]
	We will show that each step of this filtration can be obtained as an iterated pushout along morphisms of type \ref{mb:innerhorn}. Let $X \in \mathcal{M}_{\mathcal{A}}^j$ for $\epsilon\leq j \leq n-1$ and consider the pullback diagram
	\[
	\begin{tikzcd}[ampersand replacement=\&]
		\Lambda^X_i \arrow[r] \arrow[d] \& \Delta^X \arrow[d,"\sigma_X"] \\
		Y_{j-1} \arrow[r] \& Y_j
	\end{tikzcd}
	\]
	We claim that the top horizontal morphism is in the weakly saturated hull of morphisms of type \ref{mb:innerhorn} and \ref{mb:innersaturation}. First we notice that the triangle $\set{i-1,i,i+1}$ is thin in $\Delta^X$ in virtue of our assumptions. Observe that if the dimension of $\Delta^X$ is bigger than $3$ then all the possible decorations factor through $ \Lambda^X_i$. We will therefore assume that the dimension is at most 3 otherwise the claim follows directly. Suppose that $\epsilon=2$ then we can have some $\Delta^X$ of dimension $2$. In this case our assumptions guarantee that the edge that does not have the vertex $i$ cannot be marked. If $\epsilon=2$ and the dimension of $\Delta^X$ is $3$ then it follows that the face that misses the vertex $i$ cannot be thin-scaled. If that face is not lean-scaled then the claim follows immediately. Otherwise our assumptions imply that $\Delta^X$ is fully lean scaled the the map $ \Lambda^X_i  \to \Delta^X$ is a composite of a morphism of type \ref{mb:innerhorn} and a morphism of type \ref{mb:innersaturation}.  The final case $\epsilon=3$ is similar and left as an exercise.
	
	We finish the proof by noting that $X,Y \in \scr{M}_{\mathcal{A}}^j$ it follows that $\sigma_X \cap \sigma_Y \in Y_{j-1}$ which implies that the order in which the add the simplices is irrelevant. We conclude that each step in the filtration belongs to the weakly saturated hull of morphisms of type \ref{mb:innerhorn} and \ref{mb:innersaturation}.
\end{proof}
We finish the discussion on dull subsets by giving a right-horn variant of the previous construction.
\begin{definition}
	We call a subset $\mathcal{A}\subset \mathbb{P}(n)$ \emph{right-dull} if the following conditions are satisfied
	\begin{enumerate}
		\item $\mathcal{A}$ does not contain $\varnothing$.
		\item For every $S\in \mathcal{A}$, $n\notin S$. 
		\item For any $S,T\in \mathcal{A}$, $S\cap T=\varnothing$. 
		\item For every $\mathcal{A}$-basal subset $X$ we have $u,v \in X$ such that $u<v<n$.
	\end{enumerate}
	In this case we call $n$ the pivot point.
\end{definition}

\begin{lemma}\label{lem:previousRightpivot}
	Let $\mathcal{A}\subset \mathbb{P}(n)$ be a right-dull subset. Let $(\Delta^n)^{\dagger}$ be a marked-biscaled simplex whose thin triangles are degenerate. Suppose that the following conditions holds
	\begin{itemize}
		\item For every $\mathcal{A}$-basal subset $X$ and for every $s,r\in [n]$ such that $s\leq \min(X) < \max(X)\leq r<n$, the triangle $\{s<r<n\}$ is lean, and the edge $r\to n$ is marked. 
		\item Let $e$ be a marked edge in $(\Delta^n)^{\dagger}$ not containing the vertex $n$. Then $e$ factors through $(\mathcal{S}^{\mathcal{A}})^\dagger $.
		\item Let $\sigma=\set{a<b<c}$ be a lean triangle in $(\Delta^n)^{\dagger}$ not containing the vertex $n$. Then either $\sigma$ factors through $(\mathcal{S}^{\mathcal{A}})^\dagger $ or $\sigma \cup \set{n}$ is fully lean-scaled and $c \to n$ is marked.
	\end{itemize}
	Then $(\mathcal{S}^{\mathcal{A}})^\dagger \to (\Delta^n)^\dagger$ is in the saturated hull of morphisms of type \ref{mb:2Cartesianmorphs}
\end{lemma}

\begin{proof}
	The argument is nearly identical to the proof of \autoref{lem:innerpivot}.
\end{proof}

\begin{lemma}\label{lem:Right_pivot}
	Let $\mathcal{A}\subset \mathbb{P}(n)$ be a right-dull subset. Let $(\Delta^n)^{\dagger}=(\Delta^n,E_n,T_n \subset C_n)$ be a marked-biscaled simplex such that $(\Delta^n)^\diamond:=(\Delta^n,E_n,\flat \subset C_n)$ satisfies the hypothesis of \autoref{lem:previousRightpivot}. Suppose that we are given a morphism 
	\[
	\func{(\Delta^n,E_n,T_n \subset C_r) \to (X,\sharp,T_X \subset \sharp)}
	\]
	Then the morphism $(\mathcal{S}^{\mathcal{A}})^\dagger \to (\Delta^n)^\dagger$ is an $\bS$-anodyne morphism over $(X,T_X)$.
\end{lemma}
\begin{proof}
	By \autoref{lem:previousRightpivot} we obtain a pushout diagram
	\[
	\begin{tikzcd}[ampersand replacement=\&]
		(\mathcal{S}^{\mathcal{A}})^\diamond \arrow[r] \arrow[d] \& (\Delta^n)^{\diamond} \arrow[d] \\
		(\mathcal{S}^{\mathcal{A}})^\dagger  \arrow[r] \& P 
	\end{tikzcd}
	\]
	where the top horizontal morphism is $\bS$-anodyne. Note $P$ only differs from $(\Delta^n)^\dagger$ in its thin-scaling. Moreover every lean triangle in $P$ whose image in $(\Delta^n)^{\dagger}$ is thin gets mapped to a thin triangle in $(X,T_X)$ so it can be scaled using a morphism of type \ref{mb:coCartoverThin}.
\end{proof}

\subsection{Marked scaled simplicial sets}

A special case of the model structure of \autoref{thm:MBModelStructure} of particular interest occurs when $S=\Delta^0$ is the terminal scaled simplicial set. Then, by \cite[Thm 3.39]{AGS_CartI}, the resulting model structure on $\mbsSet$ is Quillen equivalent to the model structure for $\infty$-bicategories on $\Set_\Delta^{\mathbf{sc}}$. In this case, the data of the two scalings becomes highly redundant --- for any fibrant object the two scalings coincide, and heuristically they no longer encode different information. 

We can avoid this redundancy by defining a further model structure which includes both markings and scalings, but avoids the redundancies created by a biscaling. The aim of this section is to define this model structure, and relate it to the \textbf{MB} model structure. 

\begin{definition}
	A \emph{marked-scaled simplicial set} consists of 
	\begin{itemize}
		\item A simplicial set $X$.
		\item A collection of edges $E_X\subseteq X_1$ containing all degenerate edges. We call the elements of $E_X$ \emph{marked edges}.
		\item A collection of triangles $T_X\subseteq X_2$ containing all degenerate triangles. We call the elements of $T_X$ \emph{thin triangles}.
	\end{itemize}
	
	We denote by $\mssSet$ the category of marked-scaled simplicial sets. \glsadd{MSsSet} We view this as a $\Set_\Delta^+$-enriched category by defining 
	\[
	\Hom_{\Set_\Delta^+}(X,\mssSet(Y,Z)):= \Hom_{\mssSet}(X_\sharp\times Y,Z)
	\]
	where $X_\sharp=(X,E_X,\sharp)$. 
\end{definition}

Before continuing with the construction of the model structure, we briefly digress to explore the relations between $\mbsSet$ and $\mssSet$. The primary component of our comparison will be the adjunction:
\[
\adj{\mssSet}{\mbsSet}{D}{R}
\]
where $D$ is given on objects by 
\[
\func{
	D: (X,E_X,T_X) \mapsto (X,E_X,T_X\subseteq T_X)
}
\]
and $R$ is given on objects by 
\[
\func{
	R :(Y,E_Y,T_Y\subseteq C_Y) \mapsto (Y,E_Y,T_Y)
}
\]
We will show that this adjunction becomes a Quillen equivalence once we have equipped $\mssSet$ with the appropriate model structure. 

This model structure itself is constructed exactly analogously to the model structure on $\mbsSet$. We begin with a set of generating anodyne morphisms:

\begin{definition}
	The set of \emph{generating $\mathbf{MS}$-anodyne maps} $\mathbf{MS}$ \glsadd{MSann} is the set of maps of marked-scaled simplicial sets consisting of:
	\begin{itemize}
		\myitem{(MS1)}\label{MS:inner} The inner horn inclusions 
		\[
		\bigl(\Lambda^n_i,\flat,\{\Delta^{\{i-1,i,i+1\}}\}\bigr)\rightarrow \bigl(\Delta^n,\flat,\{\Delta^{\{i-1,i,i+1\}}\}\bigr)
		\quad , \quad n \geq 2 \quad , \quad 0 < i < n ;
		\]
		\myitem{(MS2)}\label{MS:wonky4} The map 
		\[
		\func{
			(\Delta^4,\flat,T) \to (\Delta^4,\flat, T\cup \{\Delta^{\{0,3,4\}},\Delta^{\{0,1,4\}}\})
		}
		\]
		where $T$ is defined as in \autoref{def:mbsanodyne}, \ref{mb:wonky4}.
		\myitem{(MS3)}\label{MS:0horn} The set of maps 
		\[
		\Bigl(\Lambda^n_0\coprod_{\Delta^{\{0,1\}}}\Delta^0,\flat,\{\Delta^{\{0,1,n\}}\}\Bigr)\rightarrow \Bigl(\Delta^n\coprod_{\Delta^{\{0,1\}}}\Delta^0,\flat,\{\Delta^{\{0,1,n\}}\}\Bigr)
		\quad , \quad n\geq 2.
		\]
		\myitem{(MS4)}\label{MS:nhorn} The set of maps 
		\[
		\Bigl(\Lambda^n_n,\{\Delta^{\{n-1,n\}}\}, \{ \Delta^{\{0,n-1,n\}} \}\Bigr) \to \Bigl(\Delta^n,\{\Delta^{\{n-1,n\}}\}, \{ \Delta^{\{0,n-1,n\}} \}\Bigr) \quad , \quad n \geq 2.
		\]
		\myitem{(MS5)}\label{MS:Cartlifts} The inclusion of the terminal vertex 
		\[
		\func{
			\left(\Delta^0,\sharp,\sharp\right)\to \left(\Delta^1,\sharp,\sharp\right)
		}
		\] 
		\myitem{(MS6)}\label{MS:Compose} The map 
		\[
		\Bigl(\Delta^2,\{\Delta^{\{0,1\}}, \Delta^{\{1,2\}}\},\sharp \Bigr) \rightarrow \Bigl(\Delta^2,\sharp,\sharp \Bigr),
		\]
		\myitem{(MS7)}\label{MS:composedeg4} The map
		\[
		\Bigl(\Delta^3 \coprod_{\Delta^{\{0,1\}}}\Delta^0,\flat, U_0\Bigr) \rightarrow \Bigl(\Delta^3 \coprod_{\Delta^{\{0,1\}}}\Delta^0,\flat, \sharp \Bigr) 
		\]
		where $U_0$ is the collection of all triangles except the $0$-th face.
		\myitem{(MS8)}\label{MS:composemarked5} The map
		\[
		\Bigl(\Delta^3,\{\Delta^{\{2,3\}}\}, U_3\Bigr) \rightarrow \Bigl(\Delta^3,\{\Delta^{\{2,3\}}\}, \sharp \Bigr) 
		\]
		where $U_3$ is the collections of all triangles except the $3$-rd face.
		\myitem{(MSE)}\label{MS:kan} For every Kan complex $K$, the map
		\[
		\Bigl( K,\flat,\sharp  \Bigr) \rightarrow \Bigl(K,\sharp, \sharp\Bigr).
		\]
	\end{itemize}
	We will call a morphism in $\mssSet$ \emph{$\mathbf{MS}$-anodyne} if it lies in the saturated hull of $\mathbf{MS}$. 
\end{definition} 

We can immediately obtain two useful lemmata. 

\begin{lemma}\label{lem:Ui_MS-anodyne}
	The morphism
	\[
	\Bigl(\Delta^3,\flat,\{\Delta^{\{i-1,i,i+1\}}\}\subset U_i\Bigr) \rightarrow \Bigl(\Delta^3,\flat, \{\Delta^{\{i-1,i,i+1\}}\}\subset \sharp \Bigr) \quad, \quad 0<i<3,
	\]
	where $U_i$ is the collection of all triangles except $i$-th face, is $\mathbf{MS}$-anodyne.
\end{lemma}

\begin{proof}
	See \cite[Rmk 3.1.4]{LurieGoodwillie}.
\end{proof}

\begin{lemma}
	The morphism 
	\[
	\func{\theta: (\Delta^2,\{\Delta^{\{1,2\}},\Delta^{\{0,2\}}\},\sharp) \to (\Delta^2,\sharp,\sharp)}
	\]
	is $\mathbf{MS}$-anodyne.
\end{lemma}
\begin{proof}
	The proof follows exactly as in \cite[Lem. 3.7]{AGS_CartI}.
\end{proof}

Finally, in total analogy to the marked biscaled case, we can establish a pushout-product axiom, and thereby a model structure.

\begin{proposition}\label{prop:PPMS}
	Let $f:X\to Y$ be an $\mathbf{MS}$-anodyne morphism in $\mssSet$, and let $g:A\to B$ be a cofibration in $\mssSet$. The morphism
	\[
	\func{
		f\wedge g: X\times B\coprod_{X\times A} Y\times A \to Y\times B 
	}
	\]
	is $\mathbf{MS}$-anodyne.
\end{proposition}

\begin{proof}
	Every case is, mutatis mutandis, the same as the corresponding case in the proof of \cite[Prop. 3.10]{AGS_CartI}.
\end{proof}

As in the marked-biscaled case, we can immediately define several mapping spaces.

\begin{definition}
	Let $\overline{X}:=(X,E_X,T_X)$ be a fibrant marked-scaled simplicial set and $\overline{Y}:=(Y,E_Y,T_Y)$ any marked-scaled simplicial set. We can define a marked-scaled simplicial set $\Fun^{\mathbf{ms}}(\overline{Y},\overline{X})$ via the universal property
	\[
	\Hom_{\mssSet}(\overline{A},\Fun^{\mathbf{ms}}(\overline{Y},\overline{X}))\cong \Hom_{\mssSet}(\overline{A}\times \overline{Y},\overline{X}).
	\]
	It follows from the pushout-product that this is a fibrant marked-scaled simplicial set, and thus that the underlying scaled simplicial set is an $\infty$-bicategory. We denote this $\infty$-bicategory by $\Map_{\mathbf{ms}}(\overline{Y},\overline{X})$.
	
	We can similarly define 
	\begin{itemize}
		\item A marked simplicial set $\Map_{\mathbf{ms}}^{\on{th}}(\overline{Y},\overline{X})$ be the full subsimplicial set of $\Fun^{\mathbf{ms}}(\overline{Y},\overline{X})$ consisting of the thin triangles.
		\item A simplicial set $\Map_{\mathbf{ms}}^{\simeq}(\overline{Y},\overline{X})$, which consists of precisely the marked edges in $\Map_{\mathbf{ms}}^{\on{th}}(\overline{Y},\overline{X})$.
	\end{itemize} 
\end{definition}

Finally, we can establish the existence of the model structure:

\begin{theorem}\label{thm:markedscaledmodel}
	There is a left-proper combinatorial simplicial model category structure on $\mssSet$ uniquely characterized by the following properties:
	\begin{itemize}
		\item[C)] A morphism $f:X\to Y$ in $\mssSet$ is a cofibration if and only if it is a monomorphism on underlying simplicial sets.
		\item[F)] An object $X\in \mssSet$ is fibrant if and only if the unique map $X\to \Delta^0$ has the right lifting property with respect to the morphisms in $\mathbf{MS}$.  
	\end{itemize}
\end{theorem}

\begin{remark}
	It is not hard to see that we can tensor $\mssSet$ over $\msSet$ and $\Set_\Delta$ in a way compatible with the enrichments provided by $\Map_{\mathbf{ms}}^{\on{th}}(-,-)$ and $\Map_{\mathbf{ms}}^{\simeq}(-,-)$, respectively. The latter of these provides the simplicial structure in the preceding proposition.
	
	The weak equivalences  in the model structure are precisely those $f:\overline{A}\to \overline{B}$, which satisfy the equivalent conditions for any fibrant marked-scaled simplicial set $\overline{X}$:
	\begin{itemize}
		\item The induced map 
		\[
		\func{\Map_{\mathbf{ms}}(\overline{B},\overline{X})\to \Map_{\mathbf{ms}}(\overline{A},\overline{X})}
		\]
		is a bicategorical equivalence.
		\item The induced map 
		\[
		\func{\Map_{\mathbf{ms}}^{\on{th}}(\overline{B},\overline{X})\to \Map_{\mathbf{ms}}^{\on{th}}(\overline{A},\overline{X})}
		\]
		is a weak equivalence of marked simplicial sets.
		\item The induced map 
		\[
		\func{\Map_{\mathbf{ms}}^{\simeq}(\overline{B},\overline{X})\to \Map_{\mathbf{ms}}^{\simeq}(\overline{A},\overline{X})}
		\]
		is a weak equivalence of Kan complexes.
	\end{itemize}
\end{remark}

It is not hard to see that the adjunction $D\dashv R$ can be promoted to a simplicial adjunction. By construction, $L$ preserves cofibrations and $R$ preserves fibrant objects, and thus we see that 

\begin{lemma}
	The adjunction 
	\[
	\adj{\mssSet}{\mbsSet}{D}{R}
	\]
	is a simplicial Quillen adjunction. 
\end{lemma}

Further, we can define an adjunction 
\[
\adj{\Set_\Delta^{\mathbf{sc}}}{\mssSet}{(-)^\flat}{G}
\] 
where $G(X,E_X,T_X)=(X,T_X)$.

\begin{lemma}\label{lem:ms-sc_adjunction}
	The adjunction 
	\[
	\adj{\Set_\Delta^{\mathbf{sc}}}{\mssSet}{(-)^\flat}{G}
	\] 
	is a Quillen adjunction.
\end{lemma}

\begin{proof}
	It is immediate that $(-)^\flat$ preserves cofibrations.  Suppose that $f:(X,T_X)\to (Y,T_Y)$ is a weak equivalence. Let $(Z,E_Z,T_Z)$ be a fibrant object in $\mssSet$. It is easy to see that $G(Z,E_Z,T_Z)=(Z,T_Z)$ is a fibrant object in $\Set_\Delta^{\mathbf{sc}}$. We can then note that, by definition, there is an isomorphism of mapping scaled simplicial sets 
	\[
	\on{Map}_{\mathbf{sc}}((X,T_X),(Z,T_Z))\cong \on{Map}_{\mathbf{ms}}((X,\flat,T_X),(Z,E_Z,T_Z)).
	\] 
	Thus, since $f$ induces a bicategorical equivalence 
	\[
	\on{Map}_{\mathbf{sc}}((Y,T_Y),(Z,T_Z))\to \on{Map}_{\mathbf{sc}}((X,T_X),(Z,T_Z))
	\]
	we see that the map
	\[
	\on{Map}_{\mathbf{ms}}((Y,\flat,T_Y),(Z,E_Z,T_Z))\to \on{Map}_{\mathbf{ms}}((X,\flat,T_X),(Z,E_Z,T_Z))
	\]
	induced by $(f)^\flat$ is also an equivalence. We therefore see that $(f)^\flat$ is a weak equivalence in $\mssSet$, as desired.
\end{proof}

\begin{lemma}\label{lem:GpresWEs}
	The functor $G$ preserves weak equivalences. 
\end{lemma}

\begin{proof}
	If, for any $\infty$-bicategory $(Z,T_Z)$, there exists a set $E_Z$ of marked edges for $Z$ such that  $(Z,E_Z,T_Z)$ is a fibrant marked-scaled simplicial set, then this follows from the characterization in terms of mapping $\infty$-bicategories. 
	
	To see that this is the case, let $(Z,T_Z)$ be an $\infty$-bicategory. Then $Z^{\on{th}}$ is an $\infty$-category, and so we can define a marking $E_Z$ on $Z$ by declaring an edge to be marked if it lies in the maximal Kan complex in $Z^{\on{th}}$. From the definition, it is immediate that $(Z,E_Z,T_Z)$ has the extension property with respect to \ref{MS:inner}, \ref{MS:wonky4},\ref{MS:0horn},\ref{MS:Cartlifts},\ref{MS:Compose}, and \ref{MS:kan}. 
	
	It follows from \cite[Cor 4.20 ]{AGS_CartI} and \cite[Cor 4.23 ]{AGS_CartI} that  $Z\to \Delta^0$ is a 2-Cartesian fibration in which the strongly Cartesian edges are precisely the equivalences, and so we see that $(Z,E_Z,T_Z)$ has the extension property with respect to \ref{MS:nhorn}, \ref{MS:composedeg4}, and \ref{MS:composemarked5} as well.
\end{proof}

\begin{lemma}\label{lem:MS_max_Kan}
	Given a fibrant marked-scaled simplicial set $(Y,E_Y,T_Y)$, the full simplicial subset $Y^{\simeq}$ on the marked edges and scaled triangles is a Kan complex.
\end{lemma}

\begin{proof}
	It is immediate from the definitions that $(Y^{\on{th}},E_Y)$ is a fibrant marked simplicial set, and the lemma follows.
\end{proof}

We now can state and prove the main proposition of this section. 

\begin{theorem}\label{thm:QE_mb_ms_sc}
	The Quillen adjunctions 
	\[
	\adj{\mssSet}{\mbsSet}{D}{R}
	\]
	and
	\[
	\adj{\Set_\Delta^{\mathbf{sc}}}{\mssSet}{(-)^\flat}{G}
	\] 
	are Quillen equivalences.
\end{theorem}

\begin{proof}
	By \cite[Thm 3.39]{AGS_CartI}, the composite adjunction $D\circ (-)^\flat\dashv G\circ R$ is a Quillen equivalence. It thus suffices for us to check that the adjunction $(-)^{\flat}\dashv G$ is a Quillen equivalence. We will check explicitly that the derived adjunction unit and counit are equivalences. 
	
	First, let $(X,T_X)\in \Set_\Delta^{\mathbf{sc}}$. The derived adjunction unit on $(X,T_X)$ is the composite
	\[
	\func{
		(X,T_X) \to G(X,\flat,T_X) \to G((X,\flat,T_X)^{\on{fib}})
	}
	\]
	where the superscript $\on{fib}$ denotes fibrant replacement. The first of these maps is the identity (since $G(X,\flat,T_X)=(X,T_X)$) and the latter is the image under $G$ of an equivalence of marked-scaled simplicial sets. By \autoref{lem:GpresWEs}, this is an equivalence. 
	
	Now, let $(Y,E_Y,T_Y)\in \mssSet$ be a fibrant object. The derived adjunction counit on $(Y,E_Y,T_Y)$ is the composite
	\[
	\func{
		(G(Y,E_Y,T_Y)^{\on{cof}})^\flat\to G(Y,E_Y,T_Y)^\flat \to[\eta_Y] (Y,E_Y,T_Y)
	}
	\]
	Since every scaled simplicial set is cofibrant, the first map is an isomorphism, leaving us to check that the usual adjunction counit $\eta_Y$ is an equivalence. Note that $\eta_Y$ is simply the inclusion $(Y,\flat,T_Y)\to (Y,E_Y,T_Y)$.
	
	We have a pushout square 
	\[
	\begin{tikzcd}
		(Y^{\simeq},\flat,\sharp)\arrow[r,"\psi"]\arrow[d,hookrightarrow] & (Y^\simeq,\sharp, \sharp)\arrow[d,hookrightarrow] \\
		(Y,\flat,T_Y) \arrow[r,"\eta_Y"'] & (Y,E_Y,T_Y) 
	\end{tikzcd}
	\]
	and, by \autoref{lem:MS_max_Kan} the morphism $\psi$ is a morphism in $\mathbf{MS}$ of type \ref{MS:kan}. Thus, $\eta_Y$ is $\mathbf{MS}$-anodyne, and is a weak equivalence.
\end{proof}

\subsubsection{The $\Set_\Delta^+$-enrichment on $\Set_\Delta^{\mathbf{ms}}$}

We have already constructed a model structure on the category $\Set_\Delta^{\mathbf{ms}}$ of marked-scaled simplicial sets, and shown that it is a simplicial model category with respect to the mapping spaces $\Map^{\simeq}_{\mathbf{ms}}(-,-)$. However, we will need to consider $\Set_\Delta^+$-enriched functors in our analysis of the Grothendieck construction. Our aim in this section is therefore to show that our model structure can, additionally, be viewed as $\msSet$-enriched. The following lemma constitutes an easy first check in this direction.

\begin{lemma}
	The category $\Set_\Delta^{\mathbf{ms}}$ is powered and tensored over $\Set_\Delta^+$ via the maps 
	\[
	\func*{
		\Set_\Delta^+\times \Set_\Delta^{\mathbf{ms}}\to \Set_\Delta^{\mathbf{ms}};
		(K,X) \mapsto K_\sharp \times X 
	}
	\]
	and 
	\[
	\func*{{[-,-]}: \Set_\Delta^+\times \Set_\Delta^{\mathbf{ms}} \to \Set_\Delta^{\mathbf{ms}}; 
		(K,X) \mapsto \Fun^{\mathbf{ms}}(K_\sharp,X) }
	\]
	The tensoring and powering is compatible with the mapping spaces $\Map^{\on{th}}_{\mathbf{ms}}(-,-)$. 
\end{lemma}

Our aim throughout the rest of the section will be to show that the tensoring is a left Quillen bifunctor. We will follow the strategy of \cite{GHL_Gray}, showing first that the model structure on $\Set_\Delta^{\mathbf{ms}}$ is a Cisinski-Olschok model structure (as with $\Set_\Delta^{\mathbf{sc}}$ in \cite{GHL_Equivalence}), and then using testing pushout-products with the concomitant interval objects. 

We first show that the model structure on $\Set_\Delta^{\mathbf{ms}}$ is Cartesian-closed. This will follow immediately from \autoref{prop:PPMS} and the following 

\begin{lemma}
	Let $f:X\to Y$ and $g:A\to B$ be two weak equivalences in $\Set_\Delta^{\mathbf{mb}}$, then the product
	\[
	\func{f\times g: X\times A \to Y\times B}
	\]
	is a weak equivalence. 
\end{lemma}

\begin{proof}
	Precisely the same argument as in \cite[Lemma 4.2.6]{LurieGoodwillie} allows us to reduce to the case of the morphism 
	\[
	\func{Y\times A\to Y\times B}
	\]
	where $Y$, $A$, and $B$ are all fibrant objects. By the characterization of fibrant objects, this morphism is a weak equivalence if and only if the morphism on underlying scaled simplicial sets is an equivalence, which follows from loc. cit. 
\end{proof}

\begin{corollary}
	For any cofibrations $f:X\to Y$ and $g:A\to B$, the pushout-product
	\[
	\func{
		f\wedge g: Y\times A \coprod_{X\times A} X\times B \to Y\times B
	}
	\]
	is an equivalence if one of $f$ or $g$ is. 
\end{corollary}

\begin{proof}
	We can use the small object argument to factor $f$ as
	\[
	\func{X\to[h] Z\to[k] Y}
	\]
	where $h$ is \textbf{MS}-anodyne. Consequently, $k$ is a weak equivalence. We consider the diagram 
	\[
	\begin{tikzcd}
		Y\times A \coprod_{X\times A} X\times B\arrow[d]  & Z\times A \coprod_{X\times A} X\times B\arrow[d]\arrow[l] \\
		Y\times B &Z\times B\arrow[l]
	\end{tikzcd}
	\]
	It follows from the lemma that the bottom horizontal arrow is a weak equivalence, and the top horizontal arrow is the induced map on homotopy colimits by a natural weak equivalence. from \autoref{prop:PPMS}, it follows that the right-hand morphism is an equivalence, and the corollary follows. 
\end{proof}

\begin{corollary}
	The model structure on $\Set_\Delta^{\mathbf{ms}}$ is Cartesian-closed.
\end{corollary}

We now wish to show that the Cisinski-Olschok model structure on $\Set_\Delta^{\mathbf{ms}}$ with interval $\Delta^0\amalg \Delta^0 \to (\Delta^1)^\sharp_\sharp$  and generating anodyne maps the \textbf{MS}-anodyne maps is, in fact the model structure constructed in our previous section. We first note that, since one of the morphism $\Delta^0\to (\Delta^1)^\sharp_\sharp$ is \textbf{MS}-anodyne, it follows that both such morphisms are trivial cofibrations.

\begin{definition}
	We write $(\Set_\Delta^{\mathbf{ms}})_{\on{CO}}$ for the Cisinski-Olschok model structure on $\Set_\Delta^{\mathbf{ms}}$ with interval $\Delta^0\amalg\Delta^0\to (\Delta^1)^\sharp_\sharp$, and generating set of anodyne morphisms the set of \textbf{MB}-anodyne morphisms. 
	
	For ease, we will write $(\Set_\Delta^{\mathbf{ms}})_{\on{AH}}$ for the model structure previously defined
\end{definition}  

\begin{proposition}
	The two model structures $(\Set_\Delta^{\mathbf{ms}})_{\on{CO}}$ and $(\Set_\Delta^{\mathbf{ms}})_{\on{AH}}$ coincide. 
\end{proposition}

\begin{proof}
	It will suffice to show that the fibrant objects coincide. By construction, every fibrant object of $(\Set_\Delta^{\mathbf{ms}})_{\on{CO}}$ is a fibrant object of $(\Set_\Delta^{\mathbf{ms}})_{\on{AH}}$. However, since $(\Set_\Delta^{\mathbf{ms}})_{\on{AH}}$ is a Cartesian-closed model category, and the interval object for $(\Set_\Delta^{\mathbf{ms}})_{\on{CO}}$ is a cylinder in $(\Set_\Delta^{\mathbf{ms}})_{\on{AH}}$, every anodyne map in $(\Set_\Delta^{\mathbf{ms}})_{\on{CO}}$ is a trivial cofibration in $(\Set_\Delta^{\mathbf{ms}})_{\on{AH}}$. Thus every fibrant object of $(\Set_\Delta^{\mathbf{ms}})_{\on{AH}}$ is a fibrant object of $(\Set_\Delta^{\mathbf{ms}})_{\on{CO}}$. 
\end{proof}

As a consequence, we will now drop the unwieldy subscript notation for the model structure on $\Set_\Delta^{\mathbf{ms}}$. We can now prove the following.

\begin{proposition}\label{prop:markedscaledenriched}
	The model category $\Set_\Delta^{\mathbf{ms}}$ is a $\Set_\Delta^+$-enriched model category. 
\end{proposition}

\begin{proof}
	We need only show that the tensoring satisfies the pushout-product axiom, i.e., that for cofibrations $f:K\to S$ in $\msSet$ and $g:X\to Y$ in $\Set_\Delta^{\mathbf{ms}}$, the pushout-product $f\wedge g$ is a trivial cofibration that either $f$ or $g$ is. Since both model structures are Cisinski-Olschok model structures, it suffices to test generating monomorphisms against the two interval inclusions and against the generating anodyne morphisms. 
	
	It is immediate from \autoref{prop:PPMS} that if $f$ (resp. $g$) is marked (resp. \textbf{MS}) anodyne, then $f\wedge g$ is a trivial cofibration. It remains for us to test the cases when $f$ is $\{0\}\to (\Delta^1)^\sharp$ or $\{1\}\to (\Delta^1)^\sharp$, and the cases when $g$ is $\{0\}\to (\Delta^1)^\sharp_\sharp$ or $\{1\}\to (\Delta^1)^\sharp_\sharp$.
	
	However, since the morphisms $\{0\}\to (\Delta^1)^\sharp_\sharp$ or $\{1\}\to (\Delta^1)^\sharp_\sharp$ are trivial  cofibrations, and the model structure on $\Set_\Delta^{\mathbf{ms}}$ is Cartesian-closed, this follows immediately.
\end{proof}

\section{The bicategorical Grothendieck construction}\label{sec:Grothendieck}

Our first step towards an $\infty$-bicategorical Grothendieck construction is defining the functors which will realize the desired equivalence. These definitions will constitute an upgrade of the straightening and unstraightening constructions of \cite[Section 3.2]{HTT} to the more highly decorated setting of marked-biscaled simplicial sets and marked-scaled simplicial sets. These functors will define a Quillen equivalence of model categories between $(\mbsSet)_{/S}$ and a model category we now define.

\begin{definition}
	Let $\scr{C}$ be a $\on{Set}_{\Delta}^+$-category. We denote by $\left(\on{Set}^{\mathbf{ms}}_{\Delta}\right)^{\C}$ the category of $\on{Set}^{+}_{\Delta}$-enriched functors and natural transformations. We endow the category of enriched functors with the projective model structure (See, e.g., \cite[A.3.3.2]{HTT}).
\end{definition}

\begin{definition}
	Let $(Y,E_Y,T_Y)$ be a marked scaled simplicial set. We define a scaled simplicial set which we denote $(Y^{\triangleright},T_{Y^{\triangleright}})$ whose underlying simplicial set is given by $Y^{\triangleright}=Y \ast \Delta^0$ and whose non-degenerate scaled simplices are either those that factor through $Y$ or those of the form $f \ast \on{id}_{\Delta^0}$ where $f:\Delta^1 \to Y$ belongs to $E_Y$. 
\end{definition}

\begin{remark}[Important convention]
	Let $(X,M_X,T_X\subseteq C_X)$ be an $\bS$ simplicial set. By the \emph{underlying scaled simplicial set}, we will mean the scaled simplicial set $(X,T_X)$. 
\end{remark}

\begin{remark}[Notation for $\op$s]
	Given a simplicial set $X$ with \emph{any} decoration (marking, scaling, etc.), we will denote by $X^\op$ the opposite simplicial set with the same decoration. 
	
	Given an enriched category $\scr{C}$ (a $\Set_\Delta^+$-enriched category, a 2-category, etc.), we will denote $\scr{C}^\op$ the enriched category with the same objects and $\scr{C}^\op(x,y)=\scr{C}(y,x)$. In the specific case of a 2-category $\CC$, we will occasionally write $\CC^{(\op,-)}$ to denote $\CC^\op$. We will only rarely make use of the 2-morphism dual $\CC^{(-,\op)}$. 
\end{remark}

We now provide the underlying left Quillen functor of our bicategorical Grothendieck construction. 

\begin{construction}
	Fix a scaled simplicial set $S\in \scsSet$ and a functor of $\msSet$-enriched categories $\phi:\mathfrak{C}^{\sc}[S]\to \scr{C}$. Let $p:X\to S$ be an object of $(\mbsSet)_{/S}$. We define a scaled simplicial set $X_S$ via the pushout diagram 
	\[
	\begin{tikzcd}
		X \arrow[r,hookrightarrow]\arrow[d,"p"'] & X^\triangleright\arrow[d] \\
		S \arrow[r] & X_S 
	\end{tikzcd}
	\]
	We generically denote both the cone point of $X^\triangleright$ and its image in $X_S$ by $\ast$. We  then define a $\msSet$-enriched category 
	\[
	X_\phi:=\scr{C}\coprod_{\mathfrak{C}^{\sc}[S]} \mathfrak{C}^{\sc}[X_S].
	\] 
	Note that this is equivalently the pushout 
	\[
	\begin{tikzcd}
		{\mathfrak{C}^{\sc}[X]}\arrow[r]\arrow[d,"{\phi\circ \mathfrak{C}[p]}"'] & {\mathfrak{C}^{\sc}[X^\triangleright]}\arrow[d] \\
		\scr{C} \arrow[r] & X_\phi
	\end{tikzcd}
	\]
	of $\msSet$-enriched categories.
	
	Applying the enriched Yoneda embedding on the cone point $\ast$, this provides a $\msSet$-enriched functor 
	\[
	\func*{
		\St_\phi^+(X): \scr{C}^\op\to \msSet;
		s\mapsto X_\phi(s,\ast).  
	}
	\]
	We promote this functor to a $\msSet$-enriched functor 
	\[
	\func{
		\ST_\phi(X):\scr{C}^\op\to \mssSet 
	}
	\] 
	by equipping its values on objects with a scaling. 
	
	After a single, fairly ad-hoc definition, we are able to do this in a highly functorial way. The ad-hoc definition will be a promotion of $\mathfrak{C}^{\sc}[X^\triangleright]$ to a $\mssSet$-enriched category, such that the subcategory $\mathfrak{C}^{\sc}[X]\subset \mathfrak{C}^{\sc}[X^\triangleright]$ has all mapping spaces maximally scaled. We will denote the resulting $\mssSet$-enriched category $\mathfrak{C}^\sc[X^\triangleright]_\dagger$. More generally, we will denote scalings on the mapping spaces of a marked-simplicially enriched category $\scr{C}$ using subscripts, e.g. $\scr{C}_\sharp$ for maximally marked mapping spaces.

	We will define the scaling on $\mathfrak{C}^{\sc}[X^\triangleright]$ in three steps:
	\begin{enumerate}
		\item We define the scaling 
		\[
		\mathfrak{C}^\sc[X^\triangleright]_\dagger(s,t):= \mathfrak{C}^\sc[X^\triangleright](s,t)_\sharp 
		\]
		for $s,t\in X$. 
		\item We define an auxiliary scaling $P_{X^\triangleright}^s$ on each marked simplicial set $\mathfrak{C}^{\sc}[X^\triangleright](s,\ast)$. iven a map $\func{\sigma:\Delta^n\to X}$, we can pass to the associated $n+1$-simplex $\func{\sigma\star \id_0:\Delta^{n+1}\to X^\triangleright}$ and obtain a map of simplicial sets
		\[
		\func{\mathfrak{C}^{\sc}[\Delta^{n+1}](0,n+1)\to \mathfrak{C}^{\sc}[X^{\triangleright}](\sigma(0),\ast).}
		\]
		Each 2-simplex in $\mathfrak{C}^{\sc}[\Delta^{n+1}](0,n+1)$ is of the form 
		\[
		S_0\cup \{n+1\}\subset S_1\cup \{n+1\}\subset S_2 \cup \{n+1\}
		\]
		where $S_i\subseteq [n]$ contains $0$. We declare the image of such a 2-simplex to be scaled in $\mathfrak{C}^\sc[X^\triangleright](s,\ast)$ precisely when either
		\begin{itemize}
			\item $\max(S_i)=\max(S_j)$ for some $i,j\in \{0,1,2\}$; or
			\item the simplex $\sigma$ is lean in $X$ (i.e., lies in $C_X$) and the 2-simplex is $03\to 013\to 0123$. 
		\end{itemize}
		The auxiliary scaling $P^s_{X^\triangleright}$ then consists of all such 2-simplices.
		\item We extend the scaling $P^s_{X^\triangleright}$ by functoriality. That is, we declare a 2-simplex $\sigma:\Delta^2\to \mathfrak{C}[X^\triangleright](s,\ast)$ to be scaled if there is a $t$ in $X$ and a 2-simplex 
		\[
		\func{\theta=(\theta_1,\theta_2): \Delta^2 \to  \mathfrak{C}^\sc[X^\triangleright](s,t) \times \mathfrak{C}^\sc[X^\triangleright](t,\ast)}
		\] 
		such that $\theta_2\circ \theta_1=\sigma$, where $\theta_2\in P^s_{X^\triangleright}$. We would like to stress to the reader that this also adds scaled 2-simplices in the case where $\theta_2$ is \emph{degenerate}. 
	\end{enumerate}
	
	We can then define a $\mssSet$-enriched variant of $X_\phi$ to be the pushout of $\mssSet$-enriched categories 
	\[
	\begin{tikzcd}
		{\mathfrak{C}^{\sc}[X]_\sharp}\arrow[r]\arrow[d,"{\phi\circ \mathfrak{C}[p]}"'] & {\mathfrak{C}^{\sc}[X^\triangleright]_\dagger}\arrow[d] \\
		\scr{C}_\sharp \arrow[r] & \overline{X_\phi}
	\end{tikzcd}
	\]
	Unwinding the definitions, we see that a 2-simplex $\sigma:\Delta^2\to \overline{X_\phi}(s,\ast)$ is scaled if and only if it satisfies the following condition:
	\begin{itemize}
		\item There is a $t\in X_\phi$ and a 2-simplex 
		\[
		\func{\theta=(\theta_1,\theta_2): \Delta^2 \to  \overline{X_\phi}(s,t) \times \overline{X_\phi}(t,\ast)}
		\]
		such that (1) $\sigma=\theta_2\circ \theta_1$, and, (2) $\theta_2$ is either in the image of an element of $P^s_{X^\triangleright}$ or is degenerate. 
	\end{itemize}
	The \emph{bicategorical straightening of $X$} is then the restriction of the  $\mssSet$-enriched Yoneda embedding: \glsadd{ST}
	\[
	\func*{
		\ST_\phi(X): \scr{C}^\op \to \mssSet; 
		s \mapsto \overline{X_\phi}(s,\ast). 
	}
	\]
	A priori, this is an $\mssSet$-enriched functor. However, since we required the mapping spaces in $\scr{C}$ to be maximally scaled, this formula in fact defines an $\msSet$-enriched functor. This construction then yields a functor 
	\[
	\func{\ST_\phi: (\mbsSet)_{/S}\to (\mssSet)^{\scr{C}^\op}}
	\]
	which we call the \emph{(bicategorical) straightening functor}. 
\end{construction}

\begin{notation}
	We will denote by $\ST_S(X)$ the special case in which $\phi:\mathfrak{C}^{\sc}[S]\to \mathfrak{C}^{\sc}[S]$ is the identity. 
\end{notation}

\begin{remark}
	In line with the philosophy of \cite{Verity}, there should be a model for $(\infty,3)$-categories on the category of simplicial sets with decorations on 1-, 2-, \emph{and} 3-simplices. The ad-hoc construction of the $\mssSet$-enriched category $\overline{X_\phi}$ above seems likely to fit into some --- as-yet-undefined --- $(\infty,3)$-categorical version of the rigidification functor, which turns decorated 3-simplices in scaled 2-simplices in the corresponding mapping space. 
\end{remark}

\begin{remark}
	Given a 2-Cartesian fibration $p:X\to S$, we note that if \emph{every} triangle in $X$ is lean, the map $\ST_S(X)(i)\to \ST_S(X)(i)_\sharp$ is an equivalence of marked-scaled simplicial sets. More generally, we obtain a diagram 
	\[
	\begin{tikzcd}
		(\mbsSet)_{/S} \arrow[r,"\ST_S"] & (\Set_\Delta^{\mathbf{ms}})^{\mathfrak{C}^{\sc}[S]^\op} \\
		(\Set_\Delta^{\mathbf{ms}})_{/S} \arrow[r,"\St^+"'] \arrow[u,"(-)_{T\subset\sharp}"] & (\Set_\Delta^{+})^{\mathfrak{C}^{\sc}[S]^\op}\arrow[u,"(-)_\sharp"'] 
	\end{tikzcd}
	\]
	which commutes up to natural weak equivalence. 
	
	While we will not formalize this statement here, there should be a model structure on $(\mssSet)_{/S}$ modeling $\infty$-bicategories fibred in $(\infty,1)$-categories, such that $\St^+$ becomes a left Quillen equivalence to the projective model structure. The diagram above would then represent the restriction of our straightening-unstraightening equivalence to this special case.
\end{remark}

\subsection{First properties}

Before proceeding to the technical nitty-gritty of the Quillen equivalences, we establish some basic properties of the straightening functor. 

\begin{proposition}\label{prop:formalproperties}
	Let $S\in \scsSet$ and let $\phi:\mathfrak{C}^{\sc}[S]\to \scr{C}$ be a $\msSet$-enriched functor. Then the following hold
	\begin{enumerate}
		\item The straightening functor $\ST_\phi$ preserves colimits.
		\item (Base change for scaled functors) Given a morphism of scaled simplicial sets $f: T\to S$ there a diagram 
		\[
		\begin{tikzcd}
			(\mbsSet)_{/S}\arrow[dr,"\ST_\phi"] \ & \\
			& (\mssSet)^{\scr{C}^\op} \\
			(\mbsSet)_{/T}\arrow[ur,"{\ST_{\phi\circ \mathfrak{C}[f]}}"']\arrow[uu,"f\circ-"] &
		\end{tikzcd}
		\]
		which commutes up to natural isomorphism of functors. 
		\item (Base change for $\msSet$-functors) Given a $\msSet$-enriched functor $\psi:\C\to \D$ there is a diagram 
		\[
		\begin{tikzcd}
			& (\mssSet)^{\scr{D}^\op} \\
			(\mbsSet)_{/S}\arrow[ur,"\ST_{\psi\circ \phi}"]\arrow[dr,"\ST_{\phi}"'] & \\
			& (\mssSet)^{\C^\op} \arrow[uu,"\psi_!"']
		\end{tikzcd}
		\]
		which commutes up to natural isomorphism of functors. 
	\end{enumerate}
\end{proposition}

\begin{proof}
	All three statements hold on the level of $\St_\phi^+$, and so the proof amounts to checking scalings. We prove (1), and leave the other two statements to the reader. 
	
	It is follows from the definition that 
	\[
	\func{\on{St}^+_\phi:(\Set_\Delta^{\mathbf{mb}})_{/S}\to (\msSet)^{\C^\op}}
	\]
	preserves colimits. Since colimits in functor categories are computed pointwise, it will thus suffice to show that, given a diagram 
	\[
	\func{D:I\to (\Set_\Delta^{\mathbf{mb}})_{/S}}
	\]
	the scalings on $\colim_I \SSt_\phi(D(i))$ and $\SSt_\phi(\colim_I D(i))$ coincide. Indeed, applying the universal property, it will suffice to show that the map 
	\[
	\func{
		\SSt_\phi(\colim_I D(i)) \to \colim_I \SSt_\phi(D(i))
	}
	\]
	which is the identity on underlying marked simplicial sets preserves the scalings. 
	
	Fix $s\in \C$, we will first show that the map 
	\[
	\func{f_s:
		(\SSt_\phi(\colim_I D(i)))(s) \to (\colim_I \SSt_\phi(D(i)))(s)
	}
	\]
	preserves the scalings inherited from $P_{X^\triangleright}^s$. To this end, suppose given a scaled simplex $\sigma$ in $P_{(\SSt_\phi(\colim_I D(i)))_S}^s$ which does not come from a lean simplex in the colimit.  Tracing through the definition, we note that there must be a simplex $\eta:\Delta^n\to \colim_I(D(i))$ and a simplex $\mu:=\{S_0\cup \{n+1\}\to S_1\cup\{n+1\}\to S_2\cup \{n+1\}\}$ in $\mathfrak{C}^{\on{sc}}[\Delta^{n+1}](0,n+1)$ with $\max(S_i)=\max(S_j)$ for some $i,j=0,1,2$ such that $\sigma$ is the image of $\mu$ under the canonical map 
	\[
	g_1:\func{\mathfrak{C}^{\on{sc}}[\Delta^{n+1}](0,n+1)\to (\SSt_\phi(\colim_I D(i)))(s)}
	\]
	is not scaled.
	
	By the construction of colimits in simplicial sets, this means that there is an $k\in I$ and a simplex $\hat{\eta}:\Delta^n\to D(k)$ such that $\eta$ factors through the canonical map $D(k)\to \colim_I D(i)$ as $\hat{\eta}$. We can then note that $\hat{\eta}$ will yield a map 
	\[
	\func{
		g_2:\mathfrak{C}^{\on{sc}}[\Delta^{n+1}](0,n+1)\to \SSt_\phi(D(k))(s)\to (\colim_I\SSt_\phi(D(i)))(s)
	}
	\]
	such that the diagram 
	\[
	\begin{tikzcd}
		& (\SSt_\phi(\colim_I D(i)))(s)\arrow[dd,"f_s"] \\
		\mathfrak{C}^{\on{sc}}[\Delta^{n+1}](0,n+1)\arrow[ur,"g_1"]\arrow[dr,"g_2"']& \\
		& (\colim_I \SSt_\phi(D(i)))(s)
	\end{tikzcd}
	\]
	commutes. We thus see that $g_2(\mu)=f_s(g_1(\mu))=f_s(\sigma)$ is scaled, as desired. The same argument holds, \emph{mutatis mutandis}, for $\sigma \in P_{(\SSt_\phi(\colim_I D(i)))_S}^s$ coming from a lean 2-simplex in the colimit.
	
	We can now easily check that the full scalings $T_{X_S}^s$ are preserved by $f_s$ by simply noting that the diagram 
	\[
	\begin{tikzcd}
		\C(s^\prime,s)\times (\on{St}^+_S(\colim_I D(i)))(s)\arrow[d,"f_s"']\arrow[rr,"\circ"] & & (\on{St}^+_S(\colim_I D(i)))(s^\prime)\arrow[d,"f_{s^\prime}"] \\ 
		\C(s^\prime,s)\times (\colim_I \SSt_S(D(i)))(s)\arrow[rr,"\circ"'] & & (\colim_I \SSt_S(D(i)))(s^\prime)
	\end{tikzcd}
	\]
	commutes.
	
	To show (2) and (3), we again note that the statements are immediate if we replace $\SSt_\phi$ with $\on{St}_\phi^+$ (cf. \cite[Rmk 3.5.16]{LurieGoodwillie} and \cite[Prop 3.2.1.4]{HTT}). A similar check to the above assures us that the scalings coincide.  
\end{proof}

\begin{remark}\label{rmk:basechange}
	Note that, in the case where we consider $\phi$ to be the identity on $\mathfrak{C}^{\sc}[S]$ and are given a morphism $f:T\to S$, combining (2) and (3) in \autoref{prop:formalproperties} yields a diagram 
	\[
	\begin{tikzcd}
		(\mbsSet)_{/S}\arrow[r,"\ST_S"] & (\mssSet)^{\mathfrak{C}^{\sc}[S]^\op} \\
		(\mbsSet)_{/T}\arrow[r,"\ST_T"']\arrow[u,"f\circ-"] & (\mssSet)^{\mathfrak{C}^{\sc}[T]^\op}\arrow[u,"{\mathfrak{C}^{\sc}[f]_!}"']
	\end{tikzcd}
	\]
	which commutes up to natural isomorphism. 
\end{remark}

\begin{corollary}
	Let $S$ be an scaled simplicial set and $\phi:\mathfrak{C}^{\sc}[S]\to \scr{C}$ a $\msSet$-enriched functor. Then the straightening functor $\ST_\phi$ has a right adjoint 
	\[
	\func{\UN_\phi:\left(\mssSet\right)^{\C^\op} \to (\mbsSet)_{/S}}
	\]
	which we call the (bicategorical) unstraightening functor. \glsadd{UN}
\end{corollary}

\begin{proof}
	This follows from the first part in \autoref{prop:formalproperties} using the adjoint functor theorem.
\end{proof}

Let $\Delta^n_{\flat}$ denote the minimally scaled $n$-simplex and consider $(\Delta^n)^{\flat}_{\flat}=(\Delta^n,\flat,\flat)$ as an object of $(\mbsSet)_{/\Delta^n_{\flat}}$ via the identity map. To ease the notation we will denote the straightening of this object as $\ST_{\Delta^n_\flat}\!(\Delta^n)$. 

\begin{definition}\label{def:Lni}
	Let $n \geq 0$ and $0\leq s \leq n$. We denote by $L^n(s)$ the poset of subsets $S \subseteq [n]$ such that $\min(S)=s$ ordered by inclusion. Let $\sigma:S_0 \subseteq S_1 \subseteq S_2$ be a 2-simplex in the (nerve of) $L^n(s)$ and denote $s_i=\max(S_i)$ for $i=0,1,2$. We say that $\sigma$ is \emph{thin} if there exists a pair of indices $i,j$ such that $s_i=s_j$. We endow $L^n(s)$ with a scaling given by thin simplices and with the minimal marking. The resulting marked scaled simplicial set will be denoted by $\mathcal{L}_{\flat}^n(s)$
\end{definition}

\begin{lemma}
	Let $n\geq 0$ and $0 \leq s \leq n$. Then there is an isomorphism
	\[
	\func{\ST_{\Delta^n_\flat}\!(\Delta^n)(s) \to[\simeq] \mathcal{L}_{\flat}^n(s)}
	\]
	of marked scaled simplicial sets
\end{lemma}
\begin{proof}
	Immediate from unraveling the definitions.
\end{proof} 

\begin{definition}\label{def:saturation}
	Let $n\geq 0$ and consider a $\bS$ simplicial set $\Delta^n_{T}:=(\Delta^n,\flat,\flat \subseteq T)$ for some scaling $T$. Given $0\leq s \leq n$ we define a new scaling on $L^n(s)$ (see \autoref{def:Lni}) by declaring a 2-simplex $S_0 \subseteq S_1 \subseteq S_2$ if and only if the simplex defined by $\max(S_0)\leq \max(S_1)\leq \max(S_2)$ is lean in $\Delta^n_{T}$. We denote the resulting scaled simplicial set by $\mathcal{L}^n_{T}(s)$.
\end{definition}

\begin{lemma}\label{lem:saturation}
	Let $\Delta^n_{T}:=(\Delta^n,\flat,\flat \subseteq T)$ and denote by $\ST_{\Delta^n_\flat}\! (\Delta^n_T)$ the straightening of the map $\Delta^n_{T} \to \Delta^n_\flat$. Then for every $0\leq s \leq n$ the canonical map
	\[
	\func{\ST_{\Delta^n_\flat}\!( \Delta^n_T)(s) \to \mathcal{L}^n_{T}(s)}
	\]
	is $\mathbf{MS}$-anodyne.
\end{lemma}
\begin{proof}
	The existence of the morphism is clear from the definitions. Suppose that we are given a thin 2-simplex $\sigma:S_0 \subseteq S_1 \subseteq S_2$ in $\mathcal{L}^n_{T}(i)$. As before, we adopt the convention that $s_i:=\max(S_i)$. We will show that $\sigma$ can be scaled by taking pushouts along $\mathbf{MS}$-anodyne morphisms. First let us consider the 3-simplex 
	\[
	\theta: S_0 \subseteq S_0 \cup \set{s_1} \subseteq S_0 \cup \set{s_1,s_2} \subseteq S_2
	\]
	We immediately observe that all of its faces are scaled in $\Sst_{\Delta^n_\flat}\!( \Delta^n)_T(s)$ except the face missing 2. It follows we can scale the remaining face using a pushout along a $\mathbf{MS}$-anodyne map of the type described in  \autoref{lem:Ui_MS-anodyne}. Now we consider another 3-simplex
	\[
	\rho: S_0 \subseteq S_0 \cup \set{s_1} \subseteq S_1 \subseteq S_2
	\]
	Again we observe that all of its faces are scaled except possibly the face missing 1 which is precisely $\sigma$. The conclusion easily follows from \autoref{lem:Ui_MS-anodyne}
\end{proof}

Let $S_{\sharp}$ be a scaled simplicial set and assume every triangle is thin. Denote by $S$ its underlying simplicial set and let $(\on{Set}_{\Delta}^{+})_{/S}$ denote the category of marked simplicial sets over $S$. We define a functor 
\[
\func{\iota:(\on{Set}_{\Delta}^{+})_{/S} \to (\mbsSet)_{/S}; (X,E_X) \mapsto (X,E_X, \sharp)}
\]
We view the $\infty$-categorical straightening functor $\on{St}_S$ (see 3.2.1 in \cite{HTT}) as a functor with values $\left(\on{Set}^{\mathbf{ms}}_{\Delta}\right)^{\mathfrak{C}^{\on{sc}}[S]^{\on{op}}}$ by maximally scaling the values of $\on{St}_S\!X(s)$.

\begin{proposition}\label{prop:infty1comparison}
	There exists a natural transformation
	\[
	\func{\epsilon:\Sst_{S}\circ \iota \nat \on{St}_S}
	\]
	which is objectwise a weak equivalence of marked scaled simplicial sets.
\end{proposition}
\begin{proof}
	The existence of the natural transformation is automatic since both functors only differ on the scaling. It is clear that both functors preserve colimits and that they satisfy base change to respect to morphisms of simplicial sets $S \to T$. In addition, it is routine to verify that both functors respect cofibrations. An standard argument then shows that it suffices to check that the natural transformation is an equivalence (1) when $S=(\Delta^n)^{\flat}$ with $n\geq 0$ and $X \to S$ is the identity morphism, and (2) on $(\Delta^1)^{\sharp} \to \Delta^1$ when $S=\Delta^1$. This is a direct consequence of \autoref{lem:saturation}.
\end{proof}

We conclude this section with a first step towards showing that the bicategorical straightening is left Quillen.

\begin{proposition}\label{prop:St_pres_cof}
	Let $S$ be a scaled simplicial set and let $\phi:\mathfrak{C}^{\sc}[S]\to \mathcal{C}$ be a $\msSet$-enriched functor. Then the straightening functor
	\[
	\func{\mathbb{S}\!\on{t}_{\phi}\!:(\mbsSet)_{/S} \to \left(\on{Set}^{\mathbf{ms}}_{\Delta}\right)^{\mathfrak{C}[\mathcal{C}]^\op}}
	\]
	preserves cofibrations.
\end{proposition}

\begin{proof}
	The generators of the class of cofibrations of marked biscaled simplicial sets are given by
	\begin{itemize}
		\myitem{(C1)}\label{cof:bndry} $\Bigr(\partial \Delta^n,\flat,\flat\Bigl) \rightarrow \Bigr( \Delta^n,\flat,\flat\Bigl)$.
		\myitem{(C2)}\label{cof:marked1} $\Bigr(\Delta^1,\flat,\flat \Bigl) \rightarrow \Bigr(\Delta^1,\sharp,\flat \Bigl) $.
		\myitem{(C3)}\label{cof:coCart} $\Bigr(\Delta^2,\flat,\flat\Bigl) \rightarrow \Bigr(\Delta^2,\flat,\flat \subset \sharp)$.
		\myitem{(C4)}\label{cof:thin} $\Bigl(\Delta^2,\flat,\flat \subset \sharp \Bigr) \rightarrow \Bigl( \Delta^2,\flat,\sharp\Bigr)$.
	\end{itemize}
	Note that \ref{cof:thin} and \ref{mb:coCartoverThin} are the same morphism. Therefore using standard arguments it will suffice to check our clainm on those generators.
	
	Let $i:A \to B$ be a cofibration. As stated above it will suffice to check in the case where $i$ is one of the generating cofibrations. Furthermore we can use \autoref{prop:formalproperties} to reduce to the case where $S$ is the underlying scaled simplicial set of $B$, and $\phi$ is $\id:\mathfrak{C}^{\sc}[S]\to \mathfrak{C}^{\sc}[S]$.  The result follows from a straightforward computation.
\end{proof}

\subsection{Products and tensoring}

Before we can proceed to proving that the straightening-unstraightening adjunction is a Quillen equivalence (indeed, before we can prove the straightening is left Quillen), we need to establish the relation of the straightening to the $\msSet$-tensoring. We will prove this as a corollary of a more general result --- on products of \textbf{MB} simplicial sets --- which will be of use to us in the sequel. 

Let $A,B \in \on{Set}_\Delta^{\mathbf{sc}}$ and consider a pair of objects $X_A \in (\on{Set}^{\mathbf{mb}}_\Delta)_{/A}$ , $X_B \in (\on{Set}^{\mathbf{mb}}_\Delta)_{/B}$ giving rise to $X \in (\on{Set}^{\mathbf{mb}}_\Delta)_{/A\times B}$  then we can form a pushout diagram
\[
\begin{tikzcd}[ampersand replacement=\&]
	\mathfrak{C}^{\on{sc}}[X] \arrow[r] \arrow[d,"\phi"] \& \mathfrak{C}^{\on{sc}}[X^{\vartriangleright}] \arrow[d] \\
	\mathfrak{C}^{\on{sc}}[A] \times \mathfrak{C}^{\on{sc}}[B] \arrow[r] \& X_{A,B}
\end{tikzcd}
\]
where the left-most vertical morphism is the composite $ \mathfrak{C}^{\on{sc}}[X] \to  \mathfrak{C}^{\on{sc}}[A\times B] \to \mathfrak{C}^{\on{sc}}[A] \times \mathfrak{C}^{\on{sc}}[B]$. Let 
\[
\SSt_A X_A \boxtimes \SSt_B X_B: \mathfrak{C}^{\on{sc}}[A]^{\op} \times \mathfrak{C}^{\on{sc}}[B]^{\op} \to \on{Set}^{\mathbf{ms}}_\Delta 
\] 
be the pointwise product of $\SSt_A X_A$ and $\SSt_B X_B $ and observe there is a canonical natural transformation
\[
\epsilon_X:\SSt_{\phi}X \xRightarrow{}  \SSt_A X_A \boxtimes \SSt_B X_B
\]
We will prove the following theorem:
\begin{theorem}\label{thm:product}
	The map $\epsilon_X:\SSt_{\phi}X \xRightarrow{}  \SSt_A (X_A) \boxtimes \SSt_B (X_B)$ is a pointwise weak equivalence.
\end{theorem}
Before proceeding with the proof of the theorem we need to do some preliminary work. First we will do a careful study of the case where $A=(\Delta^n,\flat)$ and $B=(\Delta^k,\flat)$, $X_A=(\Delta^n,\flat,\flat)$ and $X_B=(\Delta^k,\flat,\flat)$. We will assume that the maps $X_A \to A$ and $X_B \to B$ are the identity on the underlying scaled simplicial sets. In this particular situation we will denote $\SSt_{\phi}X(i,j):=\mathbb{P}^{n,k}_{(i,j)}$ and $\SSt_{\Delta^n}\Delta^n(i) \times \SSt_{\Delta^k}\Delta^k(j):=\mathbb{S}^{n,k}_{(i,j)}$.

\begin{definition}
	Let $n,k \geq 0$ and let $i \in [n], j \in [k]$. We define marked scaled simplicial set $\mathbb{E}^{n,k}_{(i,j)}$ whose underlying simplicial set is given by $\mathfrak{C}[(\Delta^n \times \Delta^k)^{\triangleright}]((i,j),\ast)$. To define the marking and the scaling we construct a morphism
	\[
	\func{\xi^{n,k}_{(i,j)}:\mathbb{E}^{n,k}_{(i,j)} \to \mathbb{S}^{n,k}_{(i,j)}}
	\]
	an equip $\mathbb{E}^{n,k}_{(i,j)}$ with the induced marking and scaling. Recall that objects of $\mathbb{E}^{n,k}_{(i,j)}$ are given by a chain or sequence of inequalities $(a_0,b_0)<(a_1,b_1)< \cdots (a_\ell,b_\ell)$ where $a_i \in [n]$ and $b_i \in [n]$ for $i=0,\dots,\ell$ and with the property that $(a_0,b_0)=(i,j)$. We will use the notation $C=\set{(a_i,b_i)}_{i=0}^{\ell}$. A morphism between chains $C_1 \to C_2$ is simply given by an inclusion $C_1 \subset C_2$ which we call a refinement of the chain $C_1$. Then we define $\xi(C)=(S_a,S_b)$ where $S_a=\set{a_0,a_1,\dots,a_\ell}$ and similarly for $S_b$.
\end{definition}

\begin{remark}
	It is immediate to see that the map $\xi^{n,k}_{(i,j)}$ constructed before factors as
	\[
	\func{\mathbb{E}^{n,k}_{(i,j)} \to \mathbb{P}^{n,k}_{(i,j)} \to \mathbb{S}^{n,k}_{(i,j)}}
	\]
	where the second morphism is the component of the natural transformation $\epsilon_X$ at the object $(i,j)$ and the first morphism is a canonical collapse map. We will denote the first morphism by $\pi^{n,k}_{(i,j)}$ and the second morphism by $\epsilon^{n,k}_{(i,j)}$.
\end{remark}

\begin{definition}
	Let $C \in \mathbb{E}^{n,k}_{(i,j)}$ be a chain denoted by $C=\set{(a_i,b_i)}_{i=0}^{\ell}$. We set $|C|=\ell$ and we call it the \emph{length} of the chain.
\end{definition}

\begin{definition}
	Let $C \in \mathbb{E}^{n,k}_{(0,0)}$. We define $\mathcal{E}_{C}$ to be the full subposet (with the induced marking and scaling) of $\mathbb{E}^{n,k}_{(0,0)}$ consisting of those chains $K$ contained in $C$. 
\end{definition}

\begin{definition}
	Let $C \in \mathbb{E}^{n,k}_{(0,0)}$ be a chain. We say that $K \in \mathcal{E}_{C}$ is a \emph{rigid chain} if there is no marked morphism in $\mathcal{E}_{C}$ with source $K$. We denote the by $\mathcal{E}_{C}^{r}$ the full subposet of $\mathcal{E}_{C}$ on rigid chains.
\end{definition}

\begin{lemma}\label{lem:boxiso}
	Let $C \in \mathbb{E}^{n,k}_{(0,0)}$ be a chain and denote by $\mathcal{U}_C$ the image of the morphism $\mathcal{E}_{C} \to \mathbb{S}^{n,k}_{(0,0)}$. Then $\xi^{n,k}_{(0,0)}$ induces an isomorphism of marked scaled simplicial sets
	\[
	\begin{tikzcd}
		\xi^r_C:\mathcal{E}^r_{C} \arrow[r,"\cong"] & \mathcal{U}_C
	\end{tikzcd}
	\]
\end{lemma}
\begin{proof}
	The map $\xi^r_C$ is clearly surjective on vertices. Moreover, given a morphism $U\to K$ in $\mathcal{E}_C$, choose a marked morphism $U\to U^r$ to a rigid chain in $\mathcal{E}_C$. Then for every $(a,b)\in U^r\setminus U$, the object $K\cup \{(a,b)\}$ will lie in $\mathcal{E}_C$ over the same element of $\mathcal{U}_C$ as $K$. We thus obtain a morphism $U^r\to \hat{K}$ lying over the original morphism in $\mathcal{U}_C$, showing that $\xi^r_C$ is surjective on morphisms, and thus on higher simplices.
	
	Moreover $\xi^r_C$ detects and preserves marked edges and thin simplices. It will therefore suffice to show that $\xi_C^r$ is injective. Let $K_i \in \mathcal{E}^r_C$ for $i=1,2$ such that $\xi^r_C(K_1)=\xi^r_C(K_2)$. Let us denote $K_i=\set{(a^i_j,b^i_j)}_{j=0}^{\ell_i}$ for $i=1,2$ . Without loss of generality let us assume that we have some $(a^1_s,b^1_s)$ such that this pair is not an element in $K_2$. However, note that since $K_i \subset C$ for $i=1,2$ then there exists a map $K_2 \to \hat{K}_2$ where $\hat{K}_2$ is obtained from $K_2$ by appending the element $(a^1_s,b^1_s)$. By construction it follows that $\xi^r_C(K_2)=\xi^r_C(\hat{K}_2)$ since $K_2$ is rigid it follows that $\hat{K}_2=K_2$  and therefore $K_1=K_2$.
\end{proof}

\begin{lemma}\label{lem:boxcover}
	Let $C \in \mathbb{E}^{n,k}_{(0,0)}$. Then the induced morphism
	\[
	\func{\xi_C:\mathcal{E}_C \to[\simeq] \mathcal{U}_C}
	\]
	is an equivalence of marked scaled simplicial sets.
\end{lemma}
\begin{proof}
	Let $\iota:\mathcal{E}_C^r \to \mathcal{E}_C$ denote the obvious inclusion. Using \autoref{lem:boxiso} we can construct a map $s_C=\iota \circ (\xi^r_C)^{-1}$. It is clear that $\xi_C \circ s_C=\on{id}$. Given $K$, observe that by construction $s_C\circ \xi_C(K)$ is rigid. Let $K \to K^r$ be a marked edge where $K^r$ is rigid. Since the restriction of $\xi_C$ to rigid objects is injective it follows that $s_C \circ \xi_C(K)= K^r$. This yields a marked homotopy from $s_C \circ \xi_C$ to the identity and the result follows.
\end{proof}

\begin{lemma}
	Let $C_i \in \mathbb{E}^{n,k}_{(0,0)}$ for $i=1,2$. Then there exists a chain $K$ such that the intersection $\mathcal{E}_{C_1} \cap \mathcal{E}_{C_2}=\mathcal{E}_K$.
\end{lemma}
\begin{proof}
	Immediate.
\end{proof}

\begin{proposition}\label{prop:EtoSequiv}
	Let $n,k$ two non-negative integers and consider $i \in [n]$ and $j \in [k]$. Then the morphism
	\[
	\func{\xi^{n,k}_{(i,j)}:\mathbb{E}^{n,k}_{(i,j)} \to \mathbb{S}^{n,k}_{(i,j)}}   
	\]  
	is an equivalence of marked scaled simplicial sets.
\end{proposition}
\begin{proof}
	First let us observe that the map $\xi^{n,k}_{(i,j)}$ is an isomorphism if either $n$ or $k$ is equal to $0$. Using an inductive argument it will suffice to show that the map $\xi^{n,k}_{(0,0)}$ is an equivalence. Note that we can cover $\mathbb{E}^{n,k}_{(i,j)}$ with the subsimplicial sets $\mathcal{E}_C$ where $C$ is a chain of maximal length. Since $\xi^{n,k}_{(0,0)}$ is surjective is covered by the subsimplicial sets $\mathcal{U}_C$. Applying \cite[Lemma 3.2.13]{AGSRelNerve}, we express $\mathbb{E}^{n,k}_{(0,0)}$ and $\mathbb{S}^{n,k}_{(0,0)}$ as the colimit over the same diagram of two homotopy cofibrant diagrams. We can now identify $\xi^{n,k}_{(0,0)}$ as the map induced by the natural transformation whose components are $\xi_C$. Therefore using \autoref{lem:boxcover} it follows that $\xi^{n,k}_{(0,0)}$ is a weak equivalence.
\end{proof}

\begin{definition}
	Let $\mathbb{O}^{n,k}=\mathfrak{C}[\Delta^n \times \Delta^k]$. We define a marking on $\mathbb{O}^{n,k}((i,j),(a,b))$ by declaring an edge marked if an only if its image in $\mathbb{O}^n(i,a)\times \mathbb{O}^k(j,b)$ is degenerate. If $a,b=n,k$ we set the notation $\mathbb{O}^{n,k}((i,j),(a,b))=\mathbb{O}^{n,k}_{(i,j)}$.
\end{definition}

\begin{lemma}\label{lem:productOO}
	The canonical morphism $p:\mathbb{O}^{n,k}((i,j),(a,b)) \to \mathbb{O}^n(i,a) \times \mathbb{O}^k(j,b)$ is a weak equivalence of marked simplicial sets.
\end{lemma}
\begin{proof}
	The argument here is virtually identical to that given in \autoref{lem:boxiso}, \autoref{lem:boxcover}, and \autoref{prop:EtoSequiv}.
\end{proof}

\begin{definition}\label{defn:eqrelnOnk}
	Let $\sigma:K_0 \subset K_1 \cdots \subset K_\ell$ be a simplex in $\mathbb{O}^{n,k}((i,j),(a,b))$ such that $K_0 \neq (i,j)$. Given $(x,y) \in K_0$ then it follows that $\sigma$ is in the image of the map
	\[
	\func{\gamma_{x,y}:\mathbb{O}^{n,k}((i,j),(x,y))\times \mathbb{O}^{n,k}((x,y),(a,b)) \to \mathbb{O}^{n,k}((i,j),(a,b)).}
	\]
	Given a pair of simplices $\sigma_1, \sigma_2$ as above, let $(A_i,B_i)$ denote the preimages of $\gamma_i$ for $i=1,2$ under $\gamma_{x,y}$. We define $\mathcal{O}^{n,k}((i,j),(a,b))$ as a quotient of $\mathbb{O}^{n,k}((i,j),(a,b))$ by identifying those simplices $\sigma_1,\sigma_2$ as above such that their corresponding $A_i$'s get identified in $\mathbb{O}^n(i,x)\times \mathbb{O}^k(j,y)$ and $B_1=B_2$.
\end{definition}

\begin{remark}\label{rem:alphabeta}
	Observe that the previous definition yields a factorization
	\[
	\func{\mathbb{O}^{n,k}((i,j),(a,b))  \to[\alpha] \mathcal{O}^{n,k}((i,j),(a,b)) \to[\beta] \mathbb{O}^n(i,a) \times \mathbb{O}^k(j,b)}
	\]
\end{remark}

\begin{lemma}\label{lem:alphabeta}
	The morphisms in \autoref{rem:alphabeta} are equivalences of marked simplicial sets. 
\end{lemma}
\begin{proof}
	By \autoref{lem:productOO}, it suffices to show that $\alpha$ is an equivalence. 
	For $(i,j)<(x,y)$, let the \emph{distance} from $(x,y)$ to $(i,j)$ be the maximal length of a chain in $\mathbb{O}^{n,k}((i,j),(x,y))$, using the convention that we count neither $(i,j)$ nor $(x,y)$ towards this length. 
	
	It is clear that if the distance from $(a,b)$ to $(i,j)$ is $0$, $\alpha$ is an isomorphism. We then proceed by induction. Suppose that that statement is true for all $(i,j)$ and $(a,b)$ with distance less than $r$, and let $(i,j)$ and $(a,b)$ be distance $r$ apart. We define a sequence of marked simplicial sets by setting 
	\[
	X_{0}=\mathbb{O}^{n,k}((i,j),(a,b))
	\]
	and then defining 
	\[
	\begin{tikzcd}
		\coprod\limits_{d((i,j),(x,y))=\ell} \mathbb{O}^{n,k}((i,j),(x,y))\times \mathbb{O}^{n,k}((x,y),(a,b))\arrow[d] \arrow[r]& X_{\ell}\arrow[d]\\
		\coprod\limits_{d((i,j),(x,y))=\ell}\mathbb{O}^{n}(i,x) \times \mathbb{O}^k(j,y) \times \mathbb{O}^{n,k}((x,y),(a,b))\arrow[r] & X_{\ell+1}
	\end{tikzcd}
	\] 
	We then note two facts:
	\begin{itemize}
		\item For $(i,j)<(x,y)$ distance $0$ apart, the canonical map 
		\[
		\mathbb{O}^{n,k}((i,j),(x,y)) \times \mathbb{O}^{n,k}((x,y),(a,b))\to X_0 
		\]
		descends through an isomorphism to a map
		$\mathcal{O}^{n,k}((i,j),(x,y)) \times \mathbb{O}^{n,k}((x,y),(a,b)) \to X_0$
		\item For $(i,j)<(x,y)$ distance $\ell$ apart, the canonical map 
		\[
		\mathbb{O}^{n,k}((i,j),(x,y))\times \mathbb{O}^{n,k}((x,y),(a,b))\to X_\ell  
		\]
		descends to a map
		$\mathcal{O}^{n,k}((i,j),(x,y))\times \mathbb{O}^{n,k}((x,y),(a,b))\to X_\ell$, since we have already quotiented out by the relations involving intermediate elements of lesser distance. 
	\end{itemize}  
	We can thus replace the pushout above with the pushout
	\[
	\begin{tikzcd}
		\coprod\limits_{d((i,j),(x,y))=\ell} \mathcal{O}^{n,k}((i,j),(x,y))\times \mathbb{O}^{n,k}((x,y),(a,b))\arrow[d] \arrow[r,hookrightarrow]& X_{\ell}\arrow[d]\\
		\coprod\limits_{d((i,j),(x,y))=\ell}\mathbb{O}^{n}(i,x) \times \mathbb{O}^k(j,y) \times \mathbb{O}^{n,k}((x,y),(a,b))\arrow[r] & X_{\ell+1}
	\end{tikzcd}
	\] 
	where the upper horizontal map is now a cofibration. This means that, by our inductive hypothesis and \autoref{lem:productOO}, $X_\ell\to X_{\ell+1}$ is an pushout of an equivalence along a cofibration, and thus an equivalence of marked simplicial sets.
	
	Since any intermediate element $(i,j)<(x,y)<(a,b)$ must have distance from $(i,j)$ strictly less than $r$, we see that $X_r=\mathcal{O}^{n,k}((i,j),(a,b))$. Thus, the composite map 
	\[
	\func{\alpha: \mathbb{O}^{n,k}((i,j),(a,b))=X_0\to X_r=\mathcal{O}^{n,k}((i,j),(a,b))}
	\] 
	is an equivalence, as desired.
\end{proof}

\begin{proposition}
	Let $n,k\geq 0$ then the morphism $\pi^{n,k}_{(i,j)}: \mathbb{E}^{n,k}_{(i,j)} \to \mathbb{P}^{n,k}_{(i,j)}$ is a weak equivalence of marked scaled simplicial sets.
\end{proposition}
\begin{proof}
	Since $\pi^{n,k}_{(i,j)}$ is an isomorphism whenever either $n$ or $k$ is equal to $0$ it follows by an inductive argument that it will suffice to show that $\pi^{n,k}_{(0,0)}$ is an equivalence. We define a sequence of marked scaled simplicial sets beggining with 
	\[
	Y_0=\mathbb{E}^{n,k}_{(0,0).}
	\]
	Then we define
	\[
	\begin{tikzcd}[ampersand replacement=\&]
		\coprod\limits_{d((0,0),(x,y))=\ell} \mathcal{O}^{n,k}((0,0),(x,y)) \times \mathbb{E}^{n,k}_{(x,y)} \arrow[d] \arrow[r] \& Y_{\ell} \arrow[d] \\
		\coprod\limits_{d((i,j),(x,y))=\ell} \mathbb{O}^{n}(0,x)\times \mathbb{O}^k (0,y)\times \mathbb{E}^{n,k}_{(x,y)} \arrow[r] \& Y_{\ell+1}
	\end{tikzcd}
	\]
	and observe that the top horizontal morphism is a cofibration. Additionally one sees that the left-most vertical morphism is an equivalence due to \autoref{lem:alphabeta}. It follows by construction that $Y_{n+k}=\mathbb{P}^{n,k}_{(0,0)}$ and since each $Y_{\ell} \to Y_{\ell+1}$ is a weak equivalence the result now follows.
\end{proof}

\begin{corollary}\label{prop:productkey}
	Let $n,k$ two non-negative integers and consider $i \in [n]$ and $j \in [k]$. Then the morphism
	\[
	\func{\epsilon^{n,k}_{(i,j)}:\mathbb{P}^{n,k}_{(i,j)} \to \mathbb{S}^{n,k}_{(i,j)}}   
	\]  
	is an equivalence of marked scaled simplicial sets.
\end{corollary}

\begin{proof}[Proof of \autoref{thm:product}]
	As in \cite[3.2.1.13]{HTT}, it suffices to check this in the special case when $X_A\to A$ and $X_B\to B$ are identity morphisms on underlying simplicial sets, and both $A$ and $B$ are one of the following cases 
	\begin{itemize}
		\item The scaled 2-simplex $\Delta^2_\sharp$. 
		\item The unscaled $n$-simplex $\Delta^n_\flat$. 
	\end{itemize}
	In the case where $A=\Delta^n_\flat$ and $B=\Delta^k_\flat$, the morphism 
	\[
	\func{
		\epsilon_X: \ST_{\phi}(\Delta^n\times \Delta^k)(i,j) \to \left(\ST_{\Delta^n}(\Delta^n)\boxtimes \ST_{\Delta^k}(\Delta^k)\right)(i,j)
	}
	\] 
	is precisely the morphism 
	\[
	\func{\epsilon^{n,k}_{(i,j)}:\mathbb{P}^{n,k}_{(i,j)} \to \mathbb{S}^{n,k}_{(i,j)}}   
	\]  
	and thus is an equivalence of marked-scaled simplicial sets by \autoref{prop:productkey}. Each other case is a pushout of some $\epsilon_{(i,j)}^{n,k}$ by a cofibration, and thus is also an equivalence. 
\end{proof}

\subsection{Straightening and anodyne morphisms} \label{subsec:STMBann}

This section serves as a stepping-stone to see that the bicategorical straightening is a left Quillen functor. in particular, we will show that $\ST_S$ preserves \textbf{MB}-anodyne morphisms for any $S\in\scsSet$. 

\begin{definition}\label{def:lnhorns}
	Consider $\Lambda^n_i$ for $0 \leq i \leq n$. For every $0\leq s \leq n$ we define $\Lambda\mathcal{L}^n_i (s)$ to be the scaled subsimplicial set of $\mathcal{L}^n_\flat(s)$ consisting of those simplices $\sigma: S_0 \subseteq S_1 \subseteq \cdots \subseteq S_n$ satisfying at least one of the following conditions:
	\begin{itemize}
		\item There exists $k\in[n]$ with $k\neq i$ such that, for every $j\in [n]$, $k\notin S_j$.
		\item There exists some $0<j\leq n$ such that $j \in S_0$ and there exists $0\leq \ell  < j$ such that $\ell \neq i$.
	\end{itemize}
	Given $\Delta^n_{T}$ as in \autoref{def:saturation} we define $(\Lambda \mathcal{L}^n_i)_{T}(s)$ using the inherited scaling from $\mathcal{L}^n_{T}(s)$.
\end{definition}

\begin{definition}
	Given a \textbf{MB} simplicial set of the form $\Delta^n_T:= (\Delta^n,\flat,\flat\subset T)$ for some $T$, we denote by $(\Lambda^n_i)_T$ the horn with the induced marking and biscaling. We write $\ST_{\Delta^n_\flat}\!(\Lambda^n_i)_T$ for the functor associated to the object $(\Lambda^n_i)_{T} \to \Delta^n_\flat$.
\end{definition}

\begin{remark}
	In some specific instances we will have $\Delta^n_T:= (\Delta^n,\flat,\flat\subset T)$ where $T=\Delta^{\set{i,j,k}}$ a chosen 2-simplex in $\Delta^n$. In that situation we will chose a subscript notation $\Delta^n_\dagger=(\Delta^n,\flat,\flat \subset \Delta^{\set{i,j,k}})$. This convention will also applied to previously defined constructions like for example $(\Lambda^n_i)_\dagger$ or $\ST_{\Delta^n}\!(\Delta^n)_\dagger$.
\end{remark}

\begin{lemma}
	Let $\Delta^n_{T}=(\Delta^n,\flat,\flat \subseteq T)$. Then for every $0 \leq s \leq n$ the canonical morphism
	\[
	\func{\SSt_{\Delta^n_\flat}\!(\Lambda^n_i)_T(s) \to[\simeq] (\Lambda \mathcal{L}^n_i)_{T}(s)}
	\]
	is $\mathbf{MS}$-anodyne.
\end{lemma}
\begin{proof}
	It is clear that for every $0 < s \leq n$ we can pick the morphism to be an isomorphims on the underlying simplicial sets. We further note that the proof of \autoref{lem:saturation} still holds in this setting. Consequently, the claim follows.
\end{proof}

\begin{definition}\label{def:pathorder}
	Let $n\geq 0$ and $0 \leq s \leq n$. We say that a (non-degenerate) simplex $\sigma$ in $\mathcal{L}^n(s)$ is a \emph{path} if it is of maximal dimension. Let $\mathcal{P}^n_s$ be the set of such paths. We will define an total order on $\mathcal{P}^n_s$ as follows:
	
	Given a path $\sigma:S_0 \subset S_1 \subset S_2 \subset \cdots S_\ell$ one sees that $S_{i+1}\setminus S_{i}=\set{a_{i+1}}$ consists precisely in one element. Therefore we can identify $\sigma$ with a list of elements 
	\[
	S_\sigma=\{ a_{i}\}_{i=1}^{\ell}.
	\]
	Note that by the maximality of $\sigma$, $S_0=\{s\}$. 
	
	Suppose we are given two such lists $S_\sigma=\{ a_{i}\}_{i=1}^{\ell}$ and $S_\theta=\{ b_{i}\}_{i=1}^{\ell}$. We declare $\sigma < \theta$ if for the first index $j$ for which $a_{j} \neq b_{j}$ then we have $a_{j} < b_{j} $.
\end{definition}

\begin{lemma}\label{lem:anodyneA3}
	Let $\Delta^n_{\diamond}=(\Delta^n,\flat,\flat \subset \Delta^{\set{0,1,n}})$ and consider the induced morphism $(\Lambda \mathcal{L}^n_0)_{\diamond}(0) \to \mathcal{L}^n_{\diamond}(0)$. Collapsing the morphism $0 \to 01$ to a degenerate edge  on both sides yields a map of scaled simplicial sets
	\[
	\func{\Lambda\mathcal{R}^n_0 \to \mathcal{R}^n}
	\]
	which is scaled anodyne.
\end{lemma}

\begin{proof}
	We use the order from \autoref{def:pathorder} to add simplices to $\Lambda\mathcal{R}^n_0$. We will add simplices in \emph{reverse order}, i.e. for any path $\sigma$, we denote by $X^{\geq \sigma}$ the scaled simplicial subset of $\mathcal{R}^n$ obtained by adding to $\Lambda\mathcal{R}^n_0$ all paths $\theta$ such that $\theta\geq \sigma$. 
	
	The procedure yields a filtration
	\[
	\func{\Lambda \mathcal{R}^n_0=X^{\geq \sigma_0} \to X^{\geq \sigma_1} \to X^{\geq \sigma_2} \to \cdots \to \mathcal{R}^n }
	\]
	where we have labeled our paths $\sigma_i$ so that $\sigma_i>\sigma_{i+1}$. The proof proceeds by showing that $X^{\geq \sigma_{i-1}}\to X^{\geq \sigma_{i}}$ is scaled anodyne for any $i$. 
	
	The proof proceeds by cases. We fix the notation that $S_{\sigma_i}=\{a_k\}_{k=1}^n$. 
	\begin{enumerate}
		\item Suppose that $a_1\neq 1$. We prove this case by showing that the top horizontal map in the pullback diagram 
		\[
		\begin{tikzcd}[ampersand replacement=\&]
			A_{\sigma_i} \arrow[r,"\alpha_i"] \arrow[d] \& \Delta^n \arrow[d,"\sigma_i"] \\
			X^{\geq \sigma_{i-1}} \arrow[r] \& X^{\geq \sigma_{i}}
		\end{tikzcd}
		\]
		is itself scaled anodyne.

		We see that $A_{\sigma_i}$ is the union of the following faces of $\sigma_i$: 
		\begin{itemize}
			\item The face $d_0(\sigma_i)$, since we will have $0\leq 1<a_1$ in each $S_k$.
			\item The face $d_n(\sigma_i)$, since this face will always be missing $a_{n}\neq 0$. 
			\item The face $d_j(\sigma_i)$ for every $j$ such that $a_{j+1}>a_j$. This is because this will, equivalently, be the $j$\textsuperscript{th} face of the (greater) simplex with vertex list
			\[
			\{a_1,\ldots, a_{j-1},a_{j+1},a_j,a_{j+2},\ldots, a_n\}.
			\]  
		\end{itemize}
		To see that the inclusion of $A_{\sigma_i}\to \Delta^n$ is scaled anodyne, we first note that, by necessity, there is at least one face of $\Delta^n$ not contained in $A_{\sigma_i}$. We then choose $t\in [n]$ such that $d_t(\Delta^n)$ is not contained in $A_{\sigma_i}$. 
		
		If we let $j\in [n]$ be the an element such that $j< t$ and $d_j(\Delta^n)\subset A_{\sigma_i}$. Similarly, let $\ell\in [n]$ be the smallest element such that $\ell>t$ and $d_\ell(\Delta^n)\subset A_{\sigma_i}$. A similar argument to \cite[Lemma 1.10]{AGS_Twisted} (or \autoref{lem:innerpivot}) shows that it will suffice to see that the simplex $\Delta^{\{j,t,\ell\}}$ is scaled for every such $j$, $t$, and $\ell$. It is easy to see that $\max(S_\ell)=\max(S_t)$. Consequently, we see that $\alpha_i$ is scaled anodyne, as desired.
		\item Now suppose that $a_1=1$, and $S_{\sigma_i}\neq \{1,2,\ldots,n-1,n\}$. We now must instead consider the pullback diagram
		\[
		\begin{tikzcd}[ampersand replacement=\&]
			B_{\sigma_i} \arrow[r,"\beta_i"] \arrow[d] \& \Delta^n \coprod_{\Delta^{\set{0,1}}}\Delta^0 \arrow[d] \\
			X^{\geq\sigma_{i-1}} \arrow[r] \& X^{\geq \sigma_i}
		\end{tikzcd}
		\]
		as above, we see that $B_{\sigma_i}$ consists of the faces 
		\begin{itemize}
			\item $d_n(\sigma_i)$
			\item $d_j(\sigma_i)$ for each $j$ such that $a_j<a_{j+1}$. 
		\end{itemize}
		Since $S_{\sigma_i}\neq \{1,2,\ldots,n-1,n\}$, there exists some $1<t<n$ such that $B_{\sigma_i}$ does not contain the $t$\textsuperscript{th} face of $\Delta^n\coprod_{\Delta^{\{0,1\}}}\Delta^0$.  
		
		We can then consider the pullback diagram 
		\[
		\begin{tikzcd}[ampersand replacement=\&]
			C_{\sigma_i} \arrow[r,"\gamma_i"] \arrow[d] \& \Delta^{n-1} \arrow[d,"d_0"] \\
			B_{\sigma_i} \arrow[r] \& \Delta^n \coprod_{\Delta^{\set{0,1}}}\Delta^0
		\end{tikzcd}   
		\]
		an apply precisely the argument from the first case to $t\in[n-1]$ described above to find that $\gamma_i$ is scaled anodyne. This means that, via a pushout, we may assume that $B_{\sigma_i}$ contains the $0$\textsuperscript{th} face of  $\Delta^n\coprod_{\Delta^{\{0,1\}}}\Delta^0$. We can then repeat essentially the same argument, and thereby see that $\beta_i$ is scaled anodyne. 
		\item If $S_{\sigma_i}=\{1,2,\ldots,n-1,n\}$, then the map $\func{X^{\geq \sigma_{i-1}}\to X^{\geq \sigma_i}=\mathcal{R}^n}$ is an inclusion 
		\[
		\func{\Lambda^n_0\coprod_{\Delta^{\{0,1\}}}\Delta^0 \to \Delta^n\coprod_{\Delta^{\{0,1\}}}\Delta^0}
		\]
		where $\Delta^{\{0,1,n\}}$ is scaled. \qedhere
	\end{enumerate}
\end{proof}

\begin{lemma}\label{lem:anodyneA4}
	Let $\Delta^n_{\dagger}=(\Delta^n,\flat,\flat \subset \Delta^{\set{0,n-1,n}})$ and consider the induced morphism $(\Lambda \mathcal{L}^n_n)_{\dagger}(0) \to \mathcal{L}^n_{\dagger}(0)$. Denote by $\mathcal{T}^n$ (resp. $\Lambda \mathcal{T}^n_n$) the marked scaled simplicial set obtained from $\mathcal{L}^n_n(0)$ (resp. $(\Lambda \mathcal{L}^n_n)_{\dagger}(0)$) by marking the edges associated to the edge $\func{(n-1) \to n}$ in $\Delta^n$.  Then the associated map
	\[
	\func{\Lambda \mathcal{T}^n_n \to \mathcal{T}^n}
	\]
	is $\mathbf{MS}$-anodyne.
\end{lemma}

\begin{proof}
	The argument is nearly identical to the proof of \autoref{lem:anodyneA3}. Using the same order as in that proof, we produce a filtration 
	\[
	\func{\Lambda \mathcal{T}^n_n=X^{\geq \sigma_{0}} \to X^{\geq \sigma_{1}}\to \cdots \to \mathcal{T}^n }
	\]
	and show each step is scaled anodyne. 
	
	As before, we set $S_{\sigma_i}=\{a_j\}_{j=1^n}$, and consider the pullback diagram
	\[
	\begin{tikzcd}[ampersand replacement=\&]
		A_{\sigma_{i}} \arrow[r,"\alpha_i"] \arrow[d] \& \Delta^n \arrow[d,"\sigma_i"] \\
		X^{\geq \sigma_{i-1}} \arrow[r] \& X^{\geq \sigma_i}
	\end{tikzcd}
	\]
	
	The case distinction now rests on whether or not $d_n\sigma_i$ factors through $A_{\sigma_i}$. The case when it does is formally identical to case (1) from \autoref{lem:anodyneA3}. 
	
	If $d_n(\sigma_i)$ does not factor through $A_{\sigma_i}$, then $a_n=n$. There are again two cases, based on whether $S_{\sigma_i}=\{1,2\ldots,n-1,n\}$. The case $S_{\sigma_i}\neq\{1,2\ldots,n-1,n\}$ is identical to the corresponding case in \autoref{lem:anodyneA3}. In the case $S_{\sigma_i}=\{1,2\ldots,n-1,n\}$ we find that $A_{\sigma_i}=\Lambda^n_n$, the last edge is marked, and $\Delta^{\{0,n-1,n\}}$ is scaled. This is a morphism of type \ref{MS:nhorn}, and thus is $\mathbf{MS}$-anodyne.
\end{proof}

\begin{lemma}\label{lem:anodyneA1}
	Let $\Delta^n_{\ast_i}=(\Delta^n,\flat,\flat \subset \Delta^{\set{i-1,i,i+1}})$ and consider the induced morphism $(\Lambda \mathcal{L}^n_i)_{\ast_i}(0) \to \mathcal{L}_{\ast_{i}}(0)$. Let $\mathcal{S}^n$ (resp $\Lambda \mathcal{S}^n_i$) denote the marked scaled simplicial set obtained by marking the edges of the form $S \to S'$ such that $i, i+1 \in S$ but $i \notin S$ and such that $S'=S \cup \set{i}$. Then the induced morphism
	\[
	\func{\Lambda \mathcal{S}^n_i \to \mathcal{S}^n_i }
	\]
	is $\mathbf{MS}$-anodyne.
\end{lemma}
\begin{proof}
	Let $S_\tau=\{1,2,\ldots,i-1,i+1,\ldots,n-1,n,i\}$ and denote by $\hat{\tau}$ the smallest maximal simplex  such that $\hat{\tau}>\tau$. We define a filtration as in \autoref{lem:anodyneA3} and \autoref{lem:anodyneA4} up until the stage $X^{\geq \hat{\tau}}$, yielding 
	\[
	\func{\Lambda\mathcal{S}^n_i \to X^{\geq \sigma_1} \to \cdots \to X^{\geq \hat{\tau}} \to \mathcal{S}^n_i.}
	\]
	We will first prove that that each step of this factorization is $\mathbf{MS}$-anodyne, making a distinction into 2 cases. 
	
	We consider the map $X^{\geq \sigma_{k-1}}\to X^{\sigma_k}$, and set $S_{\sigma_k}:=\{a_j\}_{j=1}^n$. We again form the pullback 
	\[
	\begin{tikzcd}[ampersand replacement=\&]
		A_{\sigma_{k}} \arrow[r,"\alpha_k"] \arrow[d] \& \Delta^n \arrow[d,"\sigma_k"] \\
		X^{\geq \sigma_{k-1}} \arrow[r] \& X^{\geq \sigma_k}
	\end{tikzcd}
	\]
	We then have two cases.
	\begin{enumerate}
		\item If $S_{\sigma_k}$ has as its last entry anything other than $i$, then $A_{\sigma_k}$ consists of 
		\begin{itemize}
			\item The face $d_0(\sigma_k)$. 
			\item The face $d_n(\sigma_k)$. 
			\item The face $d_j(\sigma_k)$ for each $j$ such that $a_{j+1}> j_j$. 
		\end{itemize}
		The argument is then nearly identical to case (1) from \autoref{lem:anodyneA3}. 
		\item If the last entry of $S_{\sigma_k}$ is $i$, then $A_{\sigma_k}$ consists of 
		\begin{itemize}
			\item The face $d_0(\sigma_k)$.
			\item The face $d_j(\sigma_k)$ for each $j$ such that $a_{j+1}>a_j$. 
		\end{itemize}
		The remainder of the argument is nearly identical to case (2) of \autoref{lem:anodyneA3}. 
	\end{enumerate}
	It now remains only for us to show that $X^{\geq \hat{\tau}}\to \scr{S}^n_i$ is $\mathbf{MS}$-anodyne. For ease of notation, we set $Z:=X^{\geq \hat{\tau}}$.
	
	We now need to add the remaining simplices. Write $\Sigma^\leq$ for the set of maximal simplices which are less than or equal to $\tau$. Given $\theta\in \Sigma^{\leq}$, we write $S_\theta=\{b_j\}_{j=1}^n$ for the ordered vertex sequence, as usual. We further denote by $\widehat{S_\theta}$ the vertex sequence obtained by removing $i$. We will call a simplex $\theta\in \Sigma^{\leq}$ \emph{disordered} if $\widehat{S_\theta}>\widehat{S_{\tau}}$. If $\widehat{S_\theta}=\widehat{S_{\tau}}$, we will call $\theta$ \emph{calm}.\footnote{If $\theta\in \Sigma^{\leq}$ is calm, then the entries of $\widehat{S_\theta}$ are in the linear order induced by the order on the integers. If $\theta$ is disordered, they are not.}
	
	Our first order of business is to add the disordered simplices in $\Sigma^{\leq}$. For each disordered $\theta$, we define $Z^{\geq \theta}$ to be obtained from $Z$ by adding all the disordered simplices $\sigma$ for which $\sigma\geq\theta$. Applying the order induced on disordered simplices, we again obtain a filtration 
	\[
	\func{Z\to Z^{\geq \sigma_1} \to Z^{\geq \sigma_2}\to \cdots \to Z^{\geq \gamma}}
	\]  
	where $\gamma$ is the minimal disordered simplex under the order $<$.
	
	As before, we form a pullback diagram 
	\[
	\begin{tikzcd}[ampersand replacement=\&]
		B_{\sigma_k} \arrow[r,"\beta_k"]\arrow[d] \&\Delta^n \arrow[d,"\sigma_k"]\\
		Z^{\geq \sigma_{k-1}}\arrow[r] \& Z^{\geq \sigma_k}
	\end{tikzcd}
	\]
	and show that $\beta_k$ is $\mathbf{MS}$-anodyne. Note that $B_{\sigma_k}$ consists precisely of 
	\begin{itemize}
		\item The face $d_0(\sigma_k)$. 
		\item The face $d_n(\sigma_k)$ (since the final entry of $S_{\sigma_k}$ cannot be $i$).
		\item The face $d_j(\sigma_k)$ for each $j$ such that $a_{j+1}>a_j$. 
	\end{itemize} 
	The argument that $\beta_k$ is anodyne is, by now, routine.
	
	We now turn to adding the calm simplices. Notice that $\tau$ is the maximal calm simplex. We now set $Y:=Z^{\geq \gamma}$, and define a filtration 
	\[
	\func{Y\to Y^{\leq \sigma_1} \to Y^{\leq \sigma_2} \to  \cdots \to Y^{\leq\tau}=\scr{S}^n_i}
	\] 
	By defining $Y^{\leq \theta}$ to be the union of $Y$ with all of the calm simplices less than or equal to $\theta$. 
	
	For every calm $\sigma_k$ other than $\tau$ itself, we obtain a pullback diagram 
	\[
	\begin{tikzcd}[ampersand replacement=\&]
		\Lambda^n_\ell \arrow[r,"{\eta_k}"]\arrow[d] \& \Delta^n\arrow[d,"\sigma_k"] \\
		Y^{\leq \sigma_{k-1}}\arrow[r] \& Y^{\leq \sigma_k }
	\end{tikzcd}
	\]
	where $\Lambda^n_\ell$ is an inner horn. If $S_{\sigma_k}=\{1,2,3,\ldots, n-1,n\}$, then this is a $\Lambda^n_i$, and the scaling on $\Delta^n_{\ast_i}$ shows us that the simplex $\Delta^{\{i-1,i,i+1\}}\subset \sigma_k$ is scaled. On the other hand, if $S_{\sigma_k}\neq \{1,2,3,\ldots, n-1,n\}$, the simplex $\Delta^{\{\ell-1,\ell,\ell+1\}}\subset \Lambda^n_\ell$ is already scaled in $\scr{L}^n_\flat(0)\subseteq \scr{L}^n_{\ast_i}(0)$. The morphism $\eta_k$ is thus a scaled anodyne map, and the pushout is therefore $\mathbf{MS}$-anodyne.
	
	We are left only to add $\tau$. However, in this case, we obtain a pullback diagram
	\[
	\begin{tikzcd}[ampersand replacement=\&]
		\Lambda^n_n\arrow[d]\arrow[r,"\mu"] \& \Delta^n \arrow[d,"\tau"]\\
		Y^{<\tau}\arrow[r] \& \scr{S}^n_i
	\end{tikzcd}
	\]
	where the 2-simplex $\Delta^{\{0,n-1,n\}}$ is scaled and the edge $\Delta^{\{n-1,n\}}$ is marked. The result then follows from a pushout of type \ref{MS:nhorn}.
\end{proof}

\begin{proposition}
	Let $S$ be a scaled simplicial set and let $\mathfrak{C}^{\sc}[S]\to \mathcal{C}$ be a functor of $\msSet$-enriched categories. Consider an \bS-anodyne morphism $i:A \to B$ in $(\mbsSet)_{/S}$. Then for every $s \in S$ then induced map
	\[
	\func{\Sst_{\phi}\! A(s) \to \Sst_{\phi}\!B(s) }
	\]
	is a trivial cofibration of marked scaled simplicial sets.
\end{proposition}
\begin{proof}
	As in the proof of \autoref{prop:St_pres_cof}, we can assume that $S=B$, that $\phi$ is $\id:\mathfrak{C}^{\sc}[S]\to \mathfrak{C}^{\sc}[S]$, and that $i$ is one of the generators in \autoref{def:mbsanodyne}. We proceed to verify each case.
	\begin{itemize}[noitemsep]
		\myitem{A1)} It is immediate that $\SSt_B A(s)\to \SSt_B B(s)$ is an isomorphism when $s\neq 0$. \autoref{lem:anodyneA1} shows that the map is $\mathbf{MS}$-anodyne when $s=0$. 
		\myitem{A2)} Note that the morphism $\Sst_{B}\!A(s) \to \Sst_B\!B(s)$ is an isomorphism for $s \neq 0$. If $s=0$ the map is an isomorphism on the underlying marked simplicial sets. Let $\hat{T}=T \cup \Delta^{\set{0,1,4}}\cup \Delta^{\set{0,3,4}}$ (see \autoref{def:mbsanodyne}) and let $\mathcal{L}^4_{T}(0)$ and $\mathcal{L}^4_{\hat{T}}$ be the simplicial sets defined in \autoref{def:saturation} equipped with the marking given by the thin simplices in the base. We obtain a commutative diagram
		\[
		\begin{tikzcd}[ampersand replacement=\&]
			\Sst_{B}\! A(0) \arrow[r] \arrow[d,"\simeq"] \& \Sst_{B}\!B(0) \arrow[d,"\simeq"] \\
			\mathcal{L}^4_{T}(0) \arrow[r] \& \mathcal{L}^4_{\hat{T}}(0)
		\end{tikzcd} 
		\] 
		where the vertical morphisms are equivalences due to \autoref{lem:saturation}. We will show that the bottom morphism is  an equivalence. Observe that once we manage to scale the simplices $0\to 01 \to 014$ and $0 \to 03 \to 034$ then rest of the scaling follows using the argument given in \autoref{lem:saturation}. We start by considering the 4-simplex given by
		\[
		0 \to 01 \to 012 \to 0123 \to 01234
		\]
		The only faces that are not scaled are precisely $\set{0,01,01234}$ and $\set{0,0123,01234}$. Consequently we can scale them using a pushout of type \ref{MS:wonky4}. Now we consider a 3-simplex
		\[
		0 \to 01 \to 014 \to 01234
		\]
		where all of its faces are now scaled except possibly the 3rd face. We further note that we can factor the last morphism as $014 \to 0134 \to 01234$ where both morphisms are marked. Therefore we can assume without loss of generality that the map $014 \to 01234$ is also marked. This allows us to scale the 3rd face using a pushout along a map of type \ref{MS:composemarked5}. Inspecting the 3-simplex
		\[
		0 \to 03 \to 0123 \to 01234 
		\]
		we see that we can add to the scaling $\set{0,03,01234}$. Finally let us consider
		\[
		0 \to 03 \to 034 \to 01234.
		\]
		As we did before we factor the last map as a composite of marked morphisms $034 \to 0134 \to 01234$. The claim follows by a totally analogous argument as before.\\
		\myitem{A3)} Let $\ast$ denote the vertex to which $0$ and $1$ get identified. Then it follows that the induced map $\Sst_{B}\!A(s) \to \Sst_B\!B(s)$ is an isomorphism for $s \neq \ast$. \autoref{lem:anodyneA3} shows that the map is $\mathbf{MS}$-anodyne when $s=0$.
		\myitem{A4)} It is immediate that $\Sst_{B}\!A(s) \to \Sst_B\!B(s)$ is an isomorphism for $s \neq 0$. \autoref{lem:anodyneA4} shows that the map is $\mathbf{MS}$-anodyne when $s=0$. 
		\myitem{S2)} The induced map is an isomorphism for every object of $\Delta^2$.\\
		\myitem{S3)} The map is an isomorphism for every $s \in \Delta^3$ such that $s \neq 0$. We will prove the case $i=1$ leaving the case $i=2$ as an exercise to the reader. Let $\mathcal{L}^3_{U_1}(0)$ and $\mathcal{L}^3_{\sharp}(0)$ be as in \autoref{def:saturation} and equip them with the marking induced by the thin simplex $\Delta^{\set{0,1,2}}$. We obtain a commutative diagram
		\[
		\begin{tikzcd}[ampersand replacement=\&]
			\Sst_{B}\! A(0) \arrow[r] \arrow[d,"\simeq"] \& \Sst_{B}\!B(0) \arrow[d,"\simeq"] \\
			\mathcal{L}^3_{U_1}(0) \arrow[r] \& \mathcal{L}^3_{\sharp}(0)
		\end{tikzcd} 
		\] 
		where the vertical morphisms are equivalences due to \autoref{lem:saturation}. Therefore it will enough to show that the bottom morphism is an anodyne map of marked scaled simplicial sets. Consider the simplex $0 \to 01 \to 012 \to 0123$ and observe that all of its faces are scaled except the face missing $1$. Therefore we can scale the 1-face using a pushout along an anodyne morphism as described in \autoref{lem:Ui_MS-anodyne}. Now we consider $0 \to 02 \to 012 \to 0123$ and we observe that we can scale the face missing 2 by another pushout. Finally we look at $0 \to 02 \to 023 \to 0123$ and we note that the last edge must be marked and that all of the faces are scaled except the face missing the vertex $3$. Thefore another pushout along a morphism of type \ref{MS:composemarked5} yields the result.\\
		
		\myitem{S4) \& S5)} The proof is very similar to the previous case and left to the reader.\\

		\myitem{A5,S1 \& E)} Since these maps are always maximally thin scaled we can use \autoref{prop:infty1comparison} and apply the pertinent proofs in Proposition 3.2.1.11 in \cite{HTT}.
	\end{itemize}
\end{proof}

\subsection{Straightening over the point} \label{subsec:STequivPt}

In this section, we will prove two important results. We will show that the the bicategorical straightening functor is left Quillen over any scaled simplicial set, and we will show that the straightening is an equivalence over the point. We fix the notation $\SSt_{\Delta^0}=\SSt_*$.
\begin{definition}
	We define a an adjunction
	\[
	L: \on{Set}_{\Delta}^{\mathbf{mb}} \llra \on{Set}_{\Delta}^{\mathbf{ms}}:R
	\]
	where $L(X,E_X,T_X \subseteq C_X):=(X,E_X,C_X)$ and $R(Y,E_Y,T_Y)=(Y,E_Y,T_Y \subseteq T_Y)$. We note that $L\circ R=\on{id}$ and that the unit natural transformation $(X,E_X,T_X \subseteq C_X) \to (X,E_X,C_X \subseteq C_X)$ is \bS-anodyne. It is easy to see that $L$ preserves cofibrations and trivial cofibrations. In particular, we see that $L \dashv R$ is a Quillen equivalence.
\end{definition}

Our goal in this section is to construct a natural transformation $\SSt_{*} \xRightarrow{} L$ which is a levelwise weak equivalence. By general abstract nonsense, it will suffice to construct morphisms $ \alpha_X: \Sst_*(X) \to L(X)$ whenever $X$ is one of the following generators
\begin{itemize}
	\item $\Delta^n_{\flat}:=(\Delta^n,\flat,\flat)$, for $n\geq 0$,
	\item $\Delta^2_{\dagger}:=(\Delta^2,\flat,\flat \subset \Delta^2)$,
	\item $\Delta^2_{\sharp}:=(\Delta^2,\flat,\Delta^2)$,
	\item $(\Delta^1)^\sharp :=(\Delta^1,\Delta^1,\flat)$,
\end{itemize}
and to prove that that the maps $\alpha_X$ are natural with respect to morphisms among generators. The next step is to give a precise description of the straightening functor applied to those generators.

\begin{definition}\label{def:Qn}
	Let $n\geq 0$ and define a simplicial set
	\[
	Q^n:=\bigsqcup\limits_{0 \leq i \leq n}\mathbb{O}^{n+1}(i,n+1)_{\big/ \sim}
	\]
	where the relation $\sim$ identifies simplices $n$-simplices $\sigma_1 \in \mathbb{O}^{n+1}(i,n+1)$ and $\sigma_2 \in \mathbb{O}^{n+1}(j,n+1)$ with $i \leq j$ whenever $\sigma_1$ is in the image of the map
	\[
	\begin{tikzcd}
		\mathbb{O}^{n+1}(i,j)\times \Delta^n \arrow[r,"\on{id}\times \sigma_2"] &  \mathbb{O}^{n+1}(i,j) \times \mathbb{O}^{n+1}(j,n+1) \arrow[r] & \mathbb{O}^{n+1}(i,n+1)
	\end{tikzcd}
	\]
	We further observe that the morphisms
	\[
	\func{\mathbb{O}^{n+1}(i,n+1) \to \Delta^n; S \mapsto \max\left(S \setminus \set{n+1}\right)}
	\]
	assemble into a map $\alpha_n:Q^n \to \Delta^n$. We wish now to upgrade $Q^n$ to a scaled simplcial set. We do so by declaring a triangle $\sigma:\Delta^2 \to Q^n$ thin if and only if its image under $p$ is degenerate in $\Delta^n$. We denote this collection of thin triangles by $T_{Q^n}$.
\end{definition}

\begin{remark}
	Given an order preserving morphism $f:[n] \to [m]$ then it is straightforward to check that we can produce a commutative diagram
	\[
	\begin{tikzcd}
		Q^n \arrow[r,"Q(f)"] \arrow[d,"\alpha_n"] & Q^m \arrow[d,"\alpha_m"] \\
		\Delta^n \arrow[r,"f"] & \Delta^m
	\end{tikzcd}
	\]
\end{remark}

\begin{lemma}\label{lem:StQ}
	We have the following isomorphisms of marked scaled simplicial sets
	\begin{itemize}
		\item $\Sst_*(\Delta^n_{\flat}) \isom (Q^n,\flat,T_{Q^n})$.
		\item $\SSt_*(\Delta^2_{\dagger})=\SSt_*(\Delta^2_{\sharp}) \isom (Q^2,\flat,\sharp)$.
		\item $\SSt_*((\Delta^1)^{\sharp})=(Q^1,\sharp,\flat)$.
	\end{itemize}
\end{lemma}

\begin{lemma}\label{lem:alphan}
	The morphism
	\[
	\func{\alpha_n: Q^n \to \Delta^n_{\flat}}
	\]
	is a weak equivalence of marked scaled simplicial sets.
\end{lemma}
\begin{proof}
	We construct a section $s:\Delta^n_{\flat} \to Q^n$ by sending $i \in [n]$ to the set $[0,i] \cup \set{n+1}$ and note that $\alpha_n \circ s=\on{id}_{\Delta^n}$. To finish the proof we will construct a marked homotopy between $\on{id}_{Q^n}$ and $s \circ \alpha_n$.
	
	Let $\sigma:\Delta^k \to Q^n$ and pick a representative $S_0 \subseteq S_1 \subseteq \cdots \subseteq S_k$ with $S_j \in \mathbb{O}^{n+1}(i,n+1)$ for $0 \leq j \leq k$. To ease the notation we will omit the element $n+1$ from the subsets $S_j$. Let us denote $s_j=\max(S_j)$ and observe that we can produce a diagram $H(\mathblank,\sigma):\Delta^1 \times \Delta^n \to Q^n$
	\[
	\begin{tikzcd}
		S_0 \arrow[r] \arrow[d] & S_1 \arrow[d] \arrow[r] & \cdots \arrow[d] \arrow[r] & S_{k-1} \arrow[d] \arrow[r] & S_k \arrow[d] \\
		{[i,s_0]} \arrow[r]     & {[i,s_1]} \arrow[r]     & \cdots \arrow[r]           & {[i,s_{k-1}]} \arrow[r]     & {[i,s_k]}    
	\end{tikzcd}
	\]
	It is straightforward to check that if $\sigma \sim \theta$ then $H(\mathblank,\sigma) = H(\mathblank,\theta)$. We have constructed now a natural transformation $\Delta^1 \times Q^n \to Q^n$. It is immediate to see that $H(0,\mathblank)=\on{id}_{Q^n}$. In addition the fact that the bottom row in the diagram is equivalent to 
	\[
	\begin{tikzcd}
		{[i,s_0]} \arrow[r]     & {[i,s_1]} \arrow[r]     & \cdots \arrow[r]           & {[i,s_{k-1}]} \arrow[r]     & {[i,s_k]}    
	\end{tikzcd}
	\]
	ensures that $H(1,\mathblank)=s \circ \alpha_n$. We conclude the proof by noting that the morphism $S_0 \to [i,s_0]$ gets collapsed to a degenerate edge and thus the homotopy is marked.
\end{proof}

\begin{proposition}\label{prop:STptToL}
	There exists a natural transformation $\alpha:\SSt_* \xRightarrow{} L$ which is a levelwise weak equivalence. 
\end{proposition}

\begin{proof}
	Using \autoref{lem:StQ}  is immediate to verify that the maps (together with decorated variants) $\alpha_n: Q^n \to \Delta^n$ assemble into a natural transformation $\alpha: \SSt_* \xRightarrow L$. To check that $\alpha$ is a levelwise equivalence, we note that due to the fact that both $\SSt_*$ and $L$ are left adjoints which preserve cofibrations it will suffice to check on generators. We proceed case by case
	\begin{itemize}
		\item $\alpha_{n}:(Q^n,\flat,T_{Q^n}) \to (\Delta^n,\flat,\flat)$ is an equivalence due to \autoref{lem:alphan}.
		\item $\alpha_2^{\sharp}:(Q^2,\flat,\sharp) \to (\Delta^2,\flat,\sharp)$ is an equivalence since we can repeat the proof above with maximally scaled simplicial sets.
		\item $\alpha_1^{\sharp}:(Q^1,\sharp,\flat) \to (\Delta^1,\sharp,\flat)$ is an isomorphism. \qedhere
	\end{itemize}
\end{proof}

\begin{theorem}\label{thm:STLQ}
	Let $S$ be an scaled simplicial set, then the bicategorical straightening functor
	\[
	\func{\mathbb{S}\!\on{t}_S\!:(\mbsSet)_{/S} \to \left(\on{Set}^{\mathbf{ms}}_{\Delta}\right)^{\mathfrak{C}^{\on{sc}}[S]^\op}}
	\]
	is a left Quillen functor. 
\end{theorem}

\begin{proof}
	Given a weak equivalence $\func{f:X\to Y}$ in  $(\mbsSet)_{/S}$, we can apply fibrant replacement to obtain a commutative diagram 
	\[
	\begin{tikzcd}
		X \arrow[r]\arrow[d,"f"'] &\widetilde{X}\arrow[d] \\
		Y\arrow[r] & \widetilde{Y}
	\end{tikzcd}
	\]
	where the horizontal morphisms are \textbf{MB}-anodyne, and there vertical morphisms are weak equivalences.
	
	Since $\ST_S$ preserves \textbf{MB}-andoyne morphisms, we may thus assume without loss of generality that $X$ and $Y$ are fibrant objects. By \cite[Lemma 3.29]{AGS_CartI}, $f$ then has a homotopy inverse $g$. Let 
	\[
	\func{H:(\Delta^1)^\sharp_\sharp \times X\to X}
	\] 
	be a marked homotopy between $g\circ f$ and $\id_X$ over $S$. Then $\ST_S(H)$ factors as 
	\[
	\begin{tikzcd}
		\ST_S(X\times (\Delta^1)^\sharp_\sharp)\arrow[r,"\epsilon"] & \ST_S(X)\boxtimes\ST_\ast((\Delta^1)^\sharp_\sharp) \arrow[r,"\alpha"] &\ST_S(X)\boxtimes(\Delta^1)^\sharp_\sharp\arrow[r] &  \ST_{\Delta^0_\flat}(K^\sharp)
	\end{tikzcd}
	\]
	Where $\epsilon$ is an equivalence by \autoref{thm:product}, $\alpha$ is an equivalence by \autoref{prop:STptToL}, and the final map is an equivalence since $(\Delta^1)^\sharp\to \Delta^0$ is an equivalence of marked simplicial sets. Since $\ST_S(g\circ f)$ and $\ST_S(\id_X)=\id_{\ST_S(X)}$ are both sections of $\ST_S(H)$, they are thus equivalent in the homotopy category. An identical argument shows that $\ST_S(f\circ g)\simeq \id_{\ST_S(Y)}$, completing the proof. 
\end{proof}

\begin{corollary}\label{prop:SToverPt}
	In particular the adjunction
	\[
	\ST_*: \left(\on{Set}_{\Delta}^{\mathbf{mb}}\right)_{/\Delta^0} \llra \on{Set}_{\Delta}^{\mathbf{ms}}: \UN_*
	\]
	is a Quillen equivalence.
\end{corollary}

\begin{proof}
	By \autoref{prop:STptToL}, $\St_\ast$ is naturally equivalent to a left Quillen equivalence. The corollary follows immediately.
\end{proof}

\subsection{Straightening over a simplex}\label{subsec:STequivSimp}

As in \cite[Ch. 2]{LurieGoodwillie}, the proof that our Grothendieck construction is a Quillen equivalence over a general scaled simplicial set will be bootstrapped from a proof over the $n$-simplices $(\Delta^n)_\flat$. In analogy to the method in op. cit., we will prove this case by constructing a \emph{mapping simplex} for each 2-Cartesian fibration $X\to \Delta^n_\flat$ --- a tractible $\bS$ simplicial set $\mathcal{M}_X\to \Delta^n_\flat$ which is equivalent to $X$ over $\Delta^n_\flat$.\footnote{In contrast to the approach in \cite[Ch. 2]{LurieGoodwillie}, we will not construct this `mapping simplex' from an enriched functor, but rather as a pushout of $\bS$ simplicial sets over $\Delta^n_\flat$.} The majority of this section is given over to showing that we can decompose a 2-Cartesian fibration $X\to \Delta^n_\flat$ as a homotopy pushout of the restriction of $X$ to $\Delta^{n-1}$, which enables the inductive step of our proof. 

\begin{remark}
	The term ``mapping simplex'' used above is potentially misleading. in \cite{HTT} and \cite{LurieGoodwillie}, a mapping simplex is a fibration over $\Delta^n$ explicitly constructed from a functor $\scr{F}:[n]\to \Set_\Delta^+$ or a $\scr{F}:\mathfrak{C}[\Delta^n]\to \Set_\Delta^+$. Our construction makes use of no such functor, and thus is not a true mapping simplex in this sense. The abuse of the term mapping simplex in the above exposition should be seen as suggestive of the role this construction fills in our proof --- one roughly analogous to the role of the mapping simplex in the proof of the $(\infty,1)$-categorical Grothendieck construction in \cite[\S 3.2]{HTT}.  
\end{remark}

\begin{definition}
	We define a marked biscaled simplicial set $(\Delta^n)^{\diamond}:=(\Delta^n,E^n_{\diamond},\flat \subset \sharp)$ where $E^n_{\diamond}$ is the collection of all edges containing the vertex $n$. It is not hard to verify that the inclusion of the terminal vertex $\Delta^{\set{n}} \to (\Delta^n)^{\diamond}$ is \bS-anodyne. 
\end{definition}

For the rest of this section, we fix be a 2-Cartesian fibration $p:X \to \Delta^n$ over the minimally scaled $n$-simplex. We consider the commutative diagram
\[
\begin{tikzcd}[ampersand replacement=\&]
	X_n \times \Delta^{\set{n}} \arrow[r] \arrow[d] \& X \arrow[d,"p"] \\
	X_n \times (\Delta^n)^{\diamond} \arrow[r] \arrow[ur,dotted,"\alpha"] \& \Delta^n
\end{tikzcd}
\]
where $X_n$ denotes the fibre over the vertex $n$ and the dotted arrow exists due to the fact that the left vertical morphism is \bS-anodyne. 

Consider the inclusion morphism $\iota:\Delta^{[0,n-1]} \to \Delta^n$ and equip $\Delta^{[0,n-1]}$ with the structure of an $\bS$ simplicial set by declaring and edge (resp. triangle) marked (resp. thin, resp. lean) if its image in $\Delta^n$ is marked (resp. thin, resp. lean) in $(\Delta^n)^{\diamond}$. Notice that this amounts to equipping $\Delta^{n-1}$ with the minimal marking and thin-scaling, and the maximal lean-scaling. 

We denote the restriction of $X$ to $\Delta^{n-1}$ by $X|_{\Delta^{n-1}}:=X\times_{\Delta^{n}}\Delta^{n-1}$, and denote the restriction of $\alpha$ to $X_n\times (\Delta^{n-1})^\diamond$ by $\alpha^\prime$. We construct an $\bS$ simplicial set $\mathcal{M}_X$ \glsadd{MX} over $\Delta^n$ by means of the pushout square
\[
\begin{tikzcd}[ampersand replacement=\&]
	X_n \times (\Delta^{n-1})^{\diamond} \arrow[r] \arrow[d,"\alpha'"] \& X_n \times (\Delta^n)^{\diamond} \arrow[d] \\
	X|_{\Delta^{n-1}} \arrow[r] \& \mathcal{M}_X
\end{tikzcd}
\]
Note that, since the top horizontal map is a cofibration, this is a homotopy pushout square in $(\mbsSet)_{/\Delta^n_\flat}$. The morphism $\alpha$ and the inclusion $X|_{\Delta^{n-1}}\to X$
yield a cone over this diagram,  and thus a canonical morphism $\omega:\mathcal{M}_X \to X$ over $\Delta^n$. The key technical element in this section will be to show that $\omega$ is a weak equivalence in the 2-Cartesian model structure.

\begin{definition}
	Let $\sigma:\Delta^k \to X$. Given $I\subset [k]$, we define 
	\[
	F_I(\sigma)=\set{\theta: \Delta^I \to \mathcal{M}_X\given \omega(\theta)=d_I(\sigma)} \cup \set{\ast}.
	\]
	Given $J \subset I \subseteq [k]$ and $\theta \in F_I(\sigma)$ such that $\theta \neq \ast$ we denote by $d_{J,I}(\theta)$ the image of $\theta$ in $\mathcal{M}_X$ under the degeneracy operator induced by the inclusion $J \subset I$.
\end{definition}

\begin{definition}
	We define a \textbf{MB} simplicial set $\mathcal{L}_X$ \glsadd{LX} whose simplices $\sigma:\Delta^k \to \mathcal{L}_X$ are given by:
	\begin{itemize}
		\item A simplex $\hat{\sigma}:\Delta^k \to X$.
		\item For every non-empty subset $I\subseteq [k]$  an element $\theta_I \in F_I(\sigma)$. If $\theta_I=*$ we use the empty set notation $\theta_I=\emptyset$.
	\end{itemize}
	We impose to this data the following compatibility conditions
	\begin{itemize}
		\item[H1)] Given $J \subset I \subseteq [k]$ and $\theta_I \in F_I(\sigma)$ such that $\theta_I \neq \emptyset$ it follows that $d_{J,I}(\theta_I)=\theta_J$. 
		\item[H2)] Given $I \subseteq [k]$ with $i_m=\max(I)$, then if $p \circ \hat{\sigma}(i_m) \neq n$ it follows that $\theta_I \neq \emptyset$.
		\item[H3)] Given $I \subseteq [k]$ such that for every $i \in I$ we have $p \circ \hat{\sigma}(i)=n $, then it follows that $\theta_I \neq \emptyset$. 
	\end{itemize}
	Notice that by construction there is a canonical projection map $v:\mathcal{L}_X \to X$. We equip $\mathcal{L}_X$ with the marking and biscaling induced by $v$. 
	
	Given a simplex $\sigma:\Delta^k\to \mathcal{L}_X$, we refer to the collection $\{\theta_I\}_{I\subset [k]}$ as the \emph{restriction data of $\sigma$}.
\end{definition}

\begin{lemma}
	The projection map $v:\mathcal{L}_X \to X$ is a trivial fibration of  \textbf{MB} simplicial sets.
\end{lemma}
\begin{proof}
	Since $v$ by definition detects all possible decorations, it will suffice to show that $v$ is a trivial fibration on the underlying simplicial sets. Note that $v$ is a bijection on $0$-simplices. Given $k \geq 1$ we consider a lifting problem
	\[
	\begin{tikzcd}[ampersand replacement=\&]
		\partial \Delta^k \arrow[r] \arrow[d] \& \mathcal{L}_X \arrow[d] \\
		\Delta^k \arrow[r,"\hat{\sigma}"] \arrow[ur,dotted]\& X
	\end{tikzcd}
	\]
	To produce the dotted arrow we use the bottom horizontal morphism as our choice for simplex in $X$. If $p \circ \hat{\sigma}(k) \neq n$ or $p \circ \hat{\sigma}$ is constant on $n$, we set $\theta_{[k]}$ to be the unique preimage of $\hat{\sigma}$ in $\mathcal{M}_X$. If $p \circ \hat{\sigma}(k)=n$ and it is not constant on $n$, we set $\theta_{[k]}=\emptyset$. The rest of the $\theta_{I}$ are always chosen according to top horizontal morphism. The compatibilities are clearly satisfied.
\end{proof}

We construct a morphism $u:\mathcal{M}_X \to \mathcal{L}_X$ that sends a simplex $\theta:\Delta^k \to \mathcal{M}_X$ to the simplex $\omega(\theta)$ in $X$. For every $I \subseteq [k]$ we set $\theta_I=d_I(\theta)$. It is clear that $u$ is a cofibration. It is not hard to see that $u$ induces a bijection on the restriction to $\Delta^{[0,n-1]}$ and on the fibre over $n$.

\begin{remark}
	Let $\pi=p \circ v: \mathcal{L}_X \to \Delta^n$. Given $\sigma:\Delta^k \to \mathcal{L}_X$ we fix the notation $\overline{\sigma}=\pi \circ \sigma$.
\end{remark}

\begin{definition}\label{def:posetZ}
	Let $\sigma:\Delta^k \to \mathcal{L}_X$ be a simplex such that $\overline{\sigma}(k)=n$. Let $\kappa_{\overline{\sigma}}$ be the first element in $[k]$ such that $\overline{\sigma}(\kappa_{\overline{\sigma}})=n$. We define a full subposet $Z_{\sigma} \subset [k]\times [n]$ consisting of 
	\begin{itemize}
		\item Those vertices of the form $(x,\overline{\sigma}(x))$ with $x < \kappa_{\overline{\sigma}}$.
		\item Those vertices of the form $(x,y)$ with $x \geq \kappa_{\overline{\sigma}}$ and $y \geq \overline{\sigma}(0)$. 
	\end{itemize}
	
	We denote by $\mathcal{Z}_\sigma$ the nerve $\Nerv(Z_\sigma)$. Note that the projection $[k]\times [n]\to [n]$ yields a canonical map $\mathcal{Z}_{\sigma} \to \Delta^n$. We endow $\mathcal{Z}_{\sigma}$ with the structure of an $\bS$ simplicial set by declaring an edge $(x_1,y_1) \to (x_2,y_2)$ marked if $x_1 =x_2 \geq \kappa_{\overline{\sigma}}$ and $y_2=n$. A triangle is declared to be lean if the associated 2-simplex in $\Delta^k$ is degenerate. Finally we say that a triangle in $\mathcal{Z}_{\sigma}$ is thin if it is already lean and its image in $\Delta^n$ is degenerate.
	
	We call those non-degenerate simplices $\rho:\Delta^\ell\to \mathcal{Z}_{\sigma}$ which are not contained in any other non-degenerate simplex \emph{essential}. 
\end{definition}

\begin{figure}
	\begin{center}
		\begin{tikzpicture}
			\path (0,0) node (a00){} (2,-2) node (a11){} (4,-2) node (a21){} (6,-4) node (a32) {};
			\foreach \lab in {a00,a11,a21,a32}{
				\draw[fill=black] (\lab) circle (0.05);
			};
			\foreach \x in {4,5}{
				\foreach \y in {0,1,2,3,4}{	
					\path (2*\x,-2*\y) node (a\x\y) {};
					\draw[fill=black] (2*\x,-2*\y) circle (0.05);		
				};
			};
			\foreach \x in {4,5}{
				\foreach \y/\z in {0/1,1/2,2/3,3/4}{	
					\draw[->] (a\x\y) to (a\x\z);	
				};
			};
			\foreach \y in {0,1,2,3,4}{	
				\draw[->] (a4\y) to (a5\y);	
			};
			\draw[->] (a00) to (a40);
			\draw[->] (a00) to (a11);
			\draw[->] (a11) to (a21);
			\draw[->] (a21) to (a41);
			\draw[->] (a21) to (a32);
			\draw[->] (a32) to (a42);
			\draw[->] (a32) to (a44);
		\end{tikzpicture}
	\end{center}
	\caption{The poset $Z_\sigma$ corresponding to the map $[5]\to [4]$ given by the sequence of values $0,1,1,2,4,4$.}
\end{figure}
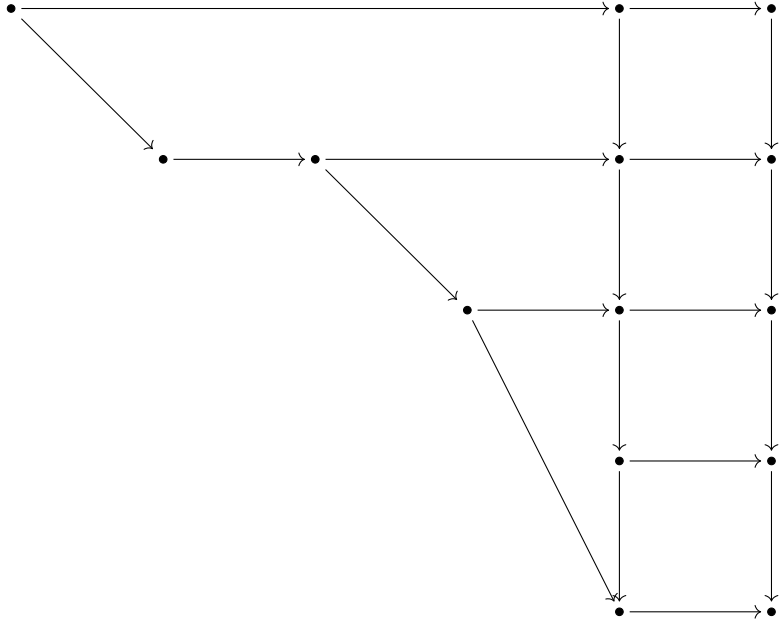

\begin{remark}
	The avid reader might complain that our definition of $\mathcal{Z}_\sigma$ only depends on $\overline{\sigma}$ and so should be denoted by $\mathcal{Z}_{\overline{\sigma}}$. The next definition will justify our notation.
\end{remark}

\begin{definition}\label{def:posetX}
	Let $\sigma:\Delta^k \to \mathcal{L}_X$ such that $\overline{\sigma}(k)=n$. We define a subsimplicial set $\mathcal{X}_{\sigma}\subset \mathcal{Z}_{\sigma}$ (with the inherited marking and scalings) consisting of those simplices $\set{(x_i,y_i)}_{i=0}^{\ell}$ satisfying at least one of the conditions below
	\begin{itemize}
		\item[i)]  We have $y_i=\overline{\sigma}(x_i)$ for $i=0,\dots,\ell$.
		\item[ii)] There exists $I \subseteq [k]$ such that $\theta_I \neq \emptyset$ with $\overline{\sigma}(\max(I))=n$ and $x_i \in I$ for $i=0,\dots,\ell$.
	\end{itemize}
\end{definition}

\begin{definition}\label{def:KI}
	Let $\sigma:\Delta^k \to \mathcal{L}_X$ such that $\overline{\sigma}(k)=n$ and suppose we are given a subset $I \subset [k]$ such that $\theta_{I}\neq \emptyset$ and such that $\overline{\sigma}(\max(I))=n$. We construct a morphism
	\[
	\func{\Delta^I \times \Delta^{[\overline{\sigma}(0),n]}\to (\Delta^n)^{\diamond}\times X_n}
	\]
	whose component at $(\Delta^n)^{\diamond}$ is given by $\Delta^I \times \Delta^{[\overline{\sigma}(0),n]} \to \Delta^{[\overline{\sigma}(0),n]} \to (\Delta^n)^{\diamond}$ and whose component at $X_n$ is given by $\Delta^I \times \Delta^{[\overline{\sigma}(0),n]} \to \Delta^I \to X_n$ where the last morphism is induced from $\theta_I$.
	
	We define a subposet $K_I \subset \Delta^I \times \Delta^{[\overline{\sigma}(0),n]}$ to be the intersection of $\Delta^I\times \Delta^{[\overline{\sigma}(0),n]}$ with $\mathcal{Z}_\sigma$. We denote $\mathcal{K}_I$ the $\bS$ simplicial set obtained by equipping $K_I$ with the decorations induced from $(\Delta^n)^{\diamond}\times X_n$.
\end{definition}

\begin{remark}
	Observe that we can construct $\mathcal{X}_\sigma$ as the union of $\Delta^k$ and every $\mathcal{K}_I$ inside of $\mathcal{Z}_\sigma$.
\end{remark}

\begin{remark}\label{rem:ftilde}
	Let $\sigma:\Delta^k \to \mathcal{L}_X$ such that $\overline{\sigma}(k)=n$. We define a morphism $\widetilde{f}_\sigma:\mathcal{X}_\sigma \to \mathcal{L}_X$ as follows:
	\begin{itemize}
		\item For simplices satisfying condition $i)$ in \autoref{def:posetX}, $\widetilde{f}_\sigma$ is simply $\sigma$.
		\item For simplices satisfying condition $ii)$ in \autoref{def:posetX}, $\widetilde{f}_\sigma$ is given by the composite
		\[
		\func{\mathcal{K}_I \to \Delta^I \times \Delta^{[\overline{\sigma}(0),n]} \to (\Delta^n)^{\diamond}\times X_n \to[u] \mathcal{L}_X}.
		\]
	\end{itemize}
	One observes that this definition is compatible in the various intersections $\mathcal{K}_I \cap \mathcal{K}_J$ thus defining the desired morphism.
\end{remark}

\begin{definition}\label{def:Xuparrow}
	Let $\sigma:\Delta^k \to \mathcal{L}_X$ such that $\overline{\sigma}(k)=n$. We define a subsimplicial subset $\mathcal{X}_\sigma^{\uparrow} \to \mathcal{Z}_\sigma$ (with the induced decorations) consisting of those simplices $\set{(x_i,y_i)}_{i=0}^{\ell}$ that are either in $\mathcal{X}_\sigma$ or satisfty the property:
	\begin{itemize}
		\item There exists $j \in [k]$ such that $x_i\neq j$ for every $i=0,\dots,\ell$ and such that $\overline{\sigma}(d^j(k-1))=n$.
	\end{itemize}
\end{definition}

\begin{remark}
	Note that we can equivalently define $\mathcal{X}_\sigma^{\uparrow}$ to consist of those simplices that are either in $\mathcal{X}_\sigma$ or that are contained in the image of
	\[
	\func{\mathcal{Z}_{d_j(\sigma)} \to \mathcal{Z}_\sigma}
	\]
	where $d_j(\sigma)(k-1)=n$.
\end{remark}

\begin{definition}\label{def:orderessential}
	We define an order on the set of essential simplices of $\mathcal{Z}_{\sigma}$ which we denote by ``$\prec$''. Let $\rho_i$ for $i=1,2$ be two essential simplices determined by the sequence of vertices $\set{(x^i_j,y^i_j)}_{j=0}^{r_i}$ for $i=1,2$. Let $\epsilon$ be the first index such that $\rho_1(\epsilon)\neq \rho_2(\epsilon)$. We say that $\rho_1 \prec \rho_2$ if precisely one the following conditions is satisfied: 
	\begin{itemize}
		\item We have that $x^1_{\epsilon}=\kappa_{\sigma}$.
		\item We have $x^i_\epsilon > \kappa_\sigma$ for $i=1,2$ and $y^1_\epsilon >y^2_\epsilon$.
	\end{itemize} 
\end{definition}

\begin{lemma}\label{lem:keylemX}
	Let $\sigma:\Delta^k \to \mathcal{L}_X$ such that $\overline{\sigma}(k)=n$. Then the following morphisms are \bS-anodyne:
	\[
	\func{\mathcal{X}_{\sigma} \to \mathcal{X}_\sigma^{\uparrow} \to \mathcal{Z}_{\sigma}.}
	\]
\end{lemma}
\begin{proof}
	If $\sigma$ factors through $\mathcal{M}_X$, then $\mathcal{X}_\sigma=\mathcal{Z}_\sigma$, so we may assume without loss of generality that $\sigma$ does not factor through $\mathcal{M}_X$. 
	
	We proceed by induction on $k$. Consider $k=1$, and note that $\mathcal{X}_\sigma=\mathcal{X}_\sigma^\uparrow$. Since $\sigma$ does not factor through $\mathcal{M}_X$, we may assume $\sigma(0)\neq n$.  There is thus some $\ell\geq 1$ such that the morphism  $\mathcal{X}_\sigma\to \mathcal{Z}_\sigma$ can be identified with the inclusion 
	\[
	\func{
		\psi_\ell:\Delta^1\coprod_{\Delta^0} (\Delta^\ell)^\diamond \to (\Delta^{\ell+1})^\dagger}
	\]
	Where $\dagger$ indicates marking where $i\to n$ is marked when $i\neq 0$, where every triangle is lean, and only those over degenerate triangles in $\Delta^n$ are thin. It follows immediately from \autoref{lem:Right_pivot} that this morphism is \textbf{MB}-anodyne.
	
	We now suppose that the lemma holds in dimension $k-1$. By the inductive hypothesis, 
	\[
	\func{\mathcal{X}_\sigma\to \mathcal{X}_\sigma^\uparrow} 
	\]
	is \bS-anodyne. Thus, it is sufficient for us to show that the morphism $\func{\mathcal{X}_\sigma^\uparrow\to \mathcal{Z}_\sigma }$ is \bS-anodyne. 
	
	We will add essential simplices of $\mathcal{Z}_\sigma$ according to the order $\prec$. Given an essential simplex $\rho$, we denote by $N_\rho$ the \textbf{MB} simplicial subset of $\mathcal{Z}_\sigma$ obtained by adding every essential simplex $\rho^\prime$ such that $\rho^\prime \preceq \rho$. We consider a pullback diagram 
	\[
	\begin{tikzcd}[ampersand replacement=\&]
		Q_\rho \arrow[r] \arrow[d] \& \Delta^r \arrow[d,"\rho"] \\
		N_{\rho^\prime} \arrow[r] \& N_\rho
	\end{tikzcd}
	\]
	and we turn our attention to proving that the top horizontal morphism is \bS-anodyne. Let us fix the notation $\rho=\set{(x_j,y_j)}_{j=0}^{r}$ and denote by $\theta$ the first index such that $x_\theta=\kappa_\sigma$.
	
	We define three types of vertices $\epsilon\in [r]$ of $\rho$. 
	\begin{itemize}
		\item[] \emph{Anterior vertices} are those $\epsilon$ which have $x_\epsilon<\kappa_\sigma$.
		\item[] \emph{Recumbent vertices}  are those $\epsilon\in [r]$ which have $x_\epsilon>\kappa_\sigma$ and $x_{\epsilon-1} <x_\epsilon$. Note that this necessarily implies $y_{\epsilon-1}=y_\epsilon$
		\item[] \emph{Plumb vertices} are those $\epsilon\in [r]$ which have $x_\epsilon\geq \kappa_\sigma$ and $x_{\epsilon-1}=x_\epsilon$. Note that this necessarily implies that $y_{\epsilon-1}<y_\epsilon$. Note also that every $\rho$ has at least one plumb vertex. 
	\end{itemize}
	Notice that the only vertex which is not anterior, recumbent or plumb is $\theta$. We call a vertex $\epsilon\in [r]$ a \emph{downturn} if $\epsilon$ is either recumbent or $\epsilon=\theta$ \textbf{and} $x_{\epsilon+1}=x_\epsilon$. Note that $\rho$ is uniquely determined by its set of downturns and the fact that it is essential.
	
	We will prove three claims about these  types of vertices, which then will enable us to complete the proof. 
	
	\vspace{1em}
	
	\noindent\textsc{Claim 1:} If $\epsilon$ is an anterior vertex, then $d_\epsilon(\Delta^r)\subset Q_\rho$. 
	
	\begin{subproof}[Subproof]
		Since $\epsilon$ is anterior, it is the only vertex of $\rho$ whose first coordinate is $x_\epsilon$. Consequently, $d_\epsilon(\rho)$ factors through $\mathcal{Z}_{d_\epsilon(\sigma)}$.
	\end{subproof} 
	
	\vspace{1em}
	
	\noindent\textsc{Claim 2:} Let $\epsilon$ be a recumbent vertex. Then $d_\epsilon(\Delta^r)\subset Q_\rho$. 
	
	\begin{subproof}[Subproof]
		There are two cases. If $x_{\epsilon+1}>x_\epsilon$, then as before $d_\epsilon(\rho)$ factors through $Z_{d_j(\sigma)}$ for some face operator $d_j$. If, on the other hand, $x_{\epsilon+1}=x_\epsilon$, then $d_\epsilon(\rho)$ factors through a previous essential simplex.  
	\end{subproof}
	
	\vspace{1em} 
	
	\noindent\textsc{Claim 3}: Let $X$ be a set of plumb vertices. Then $d_X(\Delta^r)\nsubseteq Q_\rho$. 
	
	\begin{subproof}[Subproof]
		Since $d_X(\rho)$ contains a point with first coordinate $j$ for every $j\in [k]$, we see that $d_X(\rho)$ cannot factor through $\mathcal{Z}_{d_j(\sigma)}$. Similarly, if $d_X(\rho)$ factors through $\mathcal{K}_I$ for $I\subset [k]$, then $I=[k]$, which would imply that $\sigma$ factors through $\mathcal{M}_X$. Moreover, $d_X(\rho)$ cannot factor through $\sigma$, since it contains the vertex $(x_\theta,y_\theta)$. 
		
		Finally, if $\gamma\prec \rho$ is a previous simplex in our factorization, then $d_X(\rho)$ cannot factor through $\gamma$, since $\gamma$ and $\rho$ are uniquely determined by their sequences of downturns, and the only sequence of downturns containing $d_X(\rho)$ determine simplices greater than $\rho$ under the order $\prec$.   
	\end{subproof}
	
	To finish the proof we will consider two different cases. Each of this cases will be solved used inner-dull subsets (resp. right-dull) subsets. It is important to remark that since we can assume that $\dim(\sigma) > 1$ it follows that all the decorations in $\mathcal{Z}_{\sigma}$ factor through $\mathcal{X}^{\uparrow}_\sigma$.
	
	The first case is given precisely when the vertex $r$ is recumbent. In this situation it follows that we can use \autoref{lem:innerpivot} where the pivot point is given by the biggest plumb vertex. Since $r \neq \theta$ it follows that if $r$ is not recumbent it must be plumb. In this cases the claim follow from \autoref{lem:Right_pivot}.
\end{proof}

\begin{definition}\label{defn:Bsigma}
	Let $\sigma:\Delta^k\to \mathcal{L}_X$ such that $\overline{\sigma}(k)=n$ and let $\ell=n-\overline{\sigma}(\kappa_{\overline{\sigma}}-1)$. For every morphism $f_\sigma: \mathcal{Z}_{\sigma} \to \mathcal{L}_X$ such that its restriction to $\mathcal{X}_\sigma$ equals $\widetilde{f}_\sigma$ as in \autoref{rem:ftilde}, we define a $(k+\ell)$-simplex $\sf{B}(\sigma)\in \mathcal{L}_X$ to be the composite 
	\[
	\func{\Delta^{k+\ell}\to[\rho_\sigma] \mathcal{Z}_{\sigma}\to[f_\sigma] \mathcal{L}_X}
	\]
	Where 
	\[
	\func*{
		\rho_\sigma :{[k+\ell]}\to Z_{\sigma};
		j\mapsto \begin{cases}
			(j,\overline{\sigma}(j)) & 0\leq j\leq \kappa_{\overline{\sigma}}-1 \\
			(\kappa_{\overline{\sigma}},\overline{\sigma}(\kappa_{\overline{\sigma}}-1)) &  j = \kappa_{\overline{\sigma}} \\
			(\kappa_{\overline{\sigma}},j), & \kappa_{\overline{\sigma}}<j \leq \kappa_{\overline{\sigma}}+\ell \\
			(j-\ell,n), & \kappa_{\overline{\sigma}}+\ell+1 \leq j \leq k+\ell
		\end{cases}
	}
	\]
\end{definition}

\begin{remark}
	We can equivalently characterise the simplex $\sf{B}(\sigma)$ in \autoref{defn:Bsigma} as the smallest essential simplex of $\mathcal{Z}_\sigma$ under the order $\prec$ of  \autoref{def:orderessential}  that does not factor through $\mathcal{X}_{\sigma}^\uparrow$.
\end{remark}

\begin{remark}
	There are two key parameters which we will use to analyze the simplices $\sf{B}(\sigma)$, for $\sigma: \Delta^k\to \mathcal{L}_X$. One is the fundamental vertex $\kappa_\sigma$ --- the first vertex such that $\overline{\sigma}(\kappa_\sigma)=n$. The other is the \emph{terminal size} of $\sigma$: the number of vertices   $j\in[k]$ such that $\overline{\sigma}(j)=n$. We will denote the terminal size of $\sigma$ by 
	\[
	\nu_\sigma=|\set{j \in [k] \given \overline{\sigma}(j)=n}|.
	\] 
	Notice that the terminal size of $\sf{B}(\sigma)$ is \emph{always} equal to the terminal size of $\sigma$. 
\end{remark}

To make use of the simplices $\sf{B}(\sigma)$ in an inductive pushout argument, we will need to ensure we can make sufficiently compatible choices of maps 
\[
\func{f_\sigma: \mathcal{Z}_\sigma \to \mathcal{L}_X}
\]
to define our choices of $\sf{B}(\sigma)$. 

\begin{proposition}\label{prop:Compatible_collection}
	There exists a collection indexed by the simplices of $\mathcal{L}_X$ 
	\[
	\scr{I}:=\set{f_\sigma:\mathcal{Z}_\sigma\to \mathcal{L}_X \given \sigma:\Delta^k \to \mathcal{L}_X}.
	\]
	With the following properties: 
	\begin{itemize}
		\item[i)] The restriction of $f_\sigma$ to $\mathcal{X}_{\sigma}$ equals $\widetilde{f}_\sigma$ as in \autoref{rem:ftilde}.
		\item[ii)] Given a face operator $d_j:[k-1] \to [k]$ such that $\overline{\sigma}(d^j(k-1))=n$ we have that the composite
		\[
		\func{\mathcal{Z}_{d_j(\sigma)} \to \mathcal{Z}_{\sigma}\to[f_\sigma] \mathcal{L}_X}
		\]
		equals $f_{d_j(\sigma)}$.
		\item[iii)] If $\sigma\subseteq \sf{B}(\tau)$, where $\tau\subsetneq \sigma$, then $\sf{B}(\sigma)$ is degenerate on a simplex  $\rho \subseteq \sf{B}(\tau)$.
	\end{itemize}
\end{proposition}

\begin{remark}\label{rmk:terminal_sizes_agree}
	Of key importance to our argument is the fact that, if $\sigma\subseteq \sf{B}(\tau)$ and $\tau\subsetneq \sigma$, then $\overline{\sigma}^{-1}(n)$, $\overline{\tau}^{-1}(n)$, and $\overline{\sf{B}(\tau)}^{-1}(n)$ have the same cardinality, and $\sigma$, $\tau$, and $\sf{B}(\sigma)$ agree on corresponding simplex.
\end{remark}

\begin{definition}
	Set $\Xi_k=\set{\sigma: \Delta^r \to \mathcal{L}_X \given r\leq k}$. We will call a collection 
	\[
	\scr{I}:=\set{f_\sigma:\mathcal{Z}_\sigma\to \mathcal{L}_X}_{\sigma\in \Xi_k}
	\]
	a \emph{compatible $k$-collection} if it satisfies conditions i), ii), and iii) above.
\end{definition}

\begin{lemma}\label{lem:containment_nu_enough}
	Let $\scr{I}_{k-1}$ be a compatible $(k-1)$-collection, and let $\sigma:\Delta^k\to \mathcal{L}_X$ such that $\sigma$ does not factor through $\mathcal{M}_X$. Suppose that there is a simplex $\tau:\Delta^s\to \mathcal{L}_X$ with $s<k$ such that
	\begin{itemize}
		\item There is an inclusion $\sigma \subset \sf{B}(\tau)$.
		\item The terminal sizes agree, i.e. $\nu_\sigma=\nu_{\sf{B}(\tau)}$. 
	\end{itemize}
	Then there is a subsimplex $\gamma\subset \tau$ such that 
	\begin{enumerate}
		\item There is an inclusion $\gamma\subsetneq \sigma$
		\item There is an inclusion $\sigma \subset \sf{B}(\gamma)$
		\item The terminal sizes $\nu_\sigma$ and $\nu_{\sf{B}(\gamma)}$ agree.
	\end{enumerate}
\end{lemma}

\begin{proof}
	One need only restrict to the maximal sub-simplex $\tau\cap\sigma$ that factors through $\tau$ and $\sigma$ both. Conditions (1) and (3) are immediate, and it is an easy check to see that $\sigma\subset \sf{B}(\tau\cap \sigma)$. 
\end{proof}

\begin{lemma}\label{lem:tau12}
	Let $\scr{I}_{k-1}$ be a compatible $(k-1)$-collection. Let $\sigma:\Delta^k \to \mathcal{L}_X$ and assume that $\sigma$ does not factor through $\mathcal{M}_X$. Let us suppose there exists a pair $\tau_i:\Delta^{s_i} \to \mathcal{L}_X$ with $s_i <k$ for $i=1,2$; such that $\sigma \subseteq \sf{B}(\tau_i)$ and $\tau_i \subsetneq \sigma$. Then $\tau_1 \subseteq \tau_2$ or $\tau_2 \subseteq \tau_1$.
\end{lemma}

\begin{proof}
	We can partition $\Delta^{s_i}$ into two parts: $\tau_i^{-1}([0,n-1])$ and $\tau_i^{-1}(n)$. We identify each of these with subsets of $\Delta^k$. By \autoref{rmk:terminal_sizes_agree}, $\tau_1^{-1}(n)=\tau_2^{-1}(n)=\sigma^{-1}(n)$. Since $\sigma$ does not factor through $\mathcal{M}_X$, the initial vertex of $\sigma$ must factor through each $\tau_i$. Since the only vertices $j$ of $\sf{B}(\tau_i)$ which are not vertices of $\tau_i$ satisfy with $j\geq \kappa_{\tau_i}$, we see that for $j\in [k]$ such that $j<\kappa_{\tau_i}$, $\sigma(j)$ must factor through $\tau_i$. 
	
	Thus, we see that $\tau_i$ obtained from $\sigma$ by deleting the vertices $j\in [k]$ such that $\kappa_{\tau_i}<j<\kappa_\sigma$. Thus, if $\kappa_{\tau_1}\leq \kappa_{\tau_2}$, then $\tau_1\subseteq \tau_2$. 
\end{proof}

\begin{corollary}\label{cor:max_good_simp}
	Let $\scr{I}_{k-1}$ be a compatible $(k-1)$-collection, and let $\sigma: \Delta^k\to \mathcal{L}_X$ be a simplex that does not factor through $\mathcal{L}_X$. If there is any $\tau:\Delta^s\to \mathcal{L}_X$ with $\sigma\subseteq \sf{B}(\tau)$ and $\tau\subsetneq \sigma$, then there is a unique minimal such simplex. Moreover, if there is such a simplex $\tau$ with $\sigma=\sf{B}(\tau)$, then one such simplex is minimal.
\end{corollary}

\begin{proof}
	The first statement is an immediate consequence of the previous lemma. To prove the second claim suppose that we have $\rho \subseteq \tau$ such that $\sf{B}(\tau) \subseteq \sf{B}(\rho)$. First we observe that $\nu_\tau =\nu_\rho$. Let $\mu_\rho$ be the biggest element of $\sf{B}(\rho)$ that does not lie over $n$ and such that $\mu_\rho$ and $\mu_\rho -1$ lie over the same vertex of $\Delta^n$. We similarly define $\mu_\tau$ as the biggest element of $\sf{B}(\rho)$ contained in $\sf{B}(\tau)$ satisfying the analogous property as before. Note that such elements always exist by construction. An easy argument then shows that $\mu_\tau=\mu_\rho$ and our claim follows from dimension counting.
\end{proof}

\begin{definition}
	Suppose given a compatible $(k-1)$-collection $\scr{I}_{k-1}$ and $\sigma:\Delta^k\to \mathcal{L}_X$. If it exists, we call the minimal simplex of \autoref{cor:max_good_simp} the \emph{capsule} of $\sigma$. We say that $\sigma$ is \emph{encapsulated} if it admits a capsule. 
\end{definition}

There is one final fact to establish: that there is a way of choosing a compatible degeneracy to ensure condition iii). Given $\sigma: \Delta^k\to \mathcal{L}_X$ which does not factor through $\mathcal{M}_X$, we denote by $\mathcal{R}_\sigma$ the pullback 
\[
\begin{tikzcd}
	\mathcal{R}_\sigma \arrow[r]\arrow[d] & \Delta^{k+\ell}\arrow[d,"\rho_\sigma"] \\
	\mathcal{X}_\sigma^\uparrow \arrow[r] & \mathcal{Z}_\sigma 
\end{tikzcd}
\]
Given a compatible $(k-1)$-collection $\scr{I}_{(k-1)}$, the compatibilities (1) and (2) allow us to define a map 
\[
\func{
	\tilde{f}_\sigma^\uparrow: \mathcal{X}_\sigma^\uparrow \to \mathcal{L}_X 
}
\]
for each $\sigma:\Delta^k\to \mathcal{L}_X$ which 
extends $\tilde{f}_\sigma$, and which agrees with $f_{d_j(\sigma)}$ for each face operator $d_j$ such that $d_j(\sigma)(k-1)=n$. 

\begin{lemma}
	Let $\scr{I}_{k-1}$ be a compatible $(k-1)$-collection, and suppose that $\sigma:\Delta^k\to \mathcal{L}_X$ is encapsulated with capsule $\tau$. Then for each $\zeta:\Delta^r\to \mathcal{R}_\sigma$ such that $\zeta$ hits both $\rho_\sigma(\kappa_\sigma-1)$ and $\rho_\sigma(\kappa_\sigma)$, then $\widetilde{f}_\sigma^\uparrow(\zeta)$ is degenerate on those vertices. 
\end{lemma}

\begin{proof}
	Note that our assumption means that $\zeta$ does \emph{not} factor through $\sigma$. 
	
	First suppose that $\zeta$ factors through $\mathcal{Z}_{d_j(\sigma)}$ for $j\leq \kappa_\tau-1$. Then we note that $d_j(\sigma)\subset \sf{B}(d_j(\tau))$, and $\nu_{d_j(\sigma)}=\nu_{\sf{B}(d_j(\tau))}$, so by \autoref{lem:containment_nu_enough} and the fact that $\scr{I}_{k-1}$ satisfies iii${}^\prime$) we see that $\tilde{f}_\sigma^\uparrow(\zeta)$ is degenerate. An identical argument holds when $j>\kappa_\sigma$.
	
	If $\zeta$ factors through $\mathcal{Z}_{d_j(\sigma)}$ for $\kappa_\tau\leq j\leq \kappa_\sigma$, then $\sigma(j)$ is not in $\tau$. Thus $\tau \subset d_j(\sigma)$, $d_j(\sigma)\subset \sf{B}(\tau)$, and so since $\scr{I}_{k-1}$ satisfies condition iii${}^\prime$), $\zeta$ is degenerate. 
	
	Finally, suppose that $\zeta$ factors through $\scr{K}_I^\sigma$ for some $\theta_I^\sigma\neq \varnothing$. Then $I\cap [s]=J$ has $\theta^\tau_J\neq \varnothing$, and we can factor $\tilde{f}_\sigma^\uparrow(\zeta)$ through $u$ as 
	\[
	\func{\Delta^r\to \Delta^I\times \Delta^{[\sigma(0),n]}\to \Delta^n\times X_n}
	\]
	By construction, the first factor of this simplex is degenerate at the desired vertex. The second factor can be equivalently factored through $\Delta^J\times \Delta^{[\tau(0),n]}$ and thus is degenerate at the desired vertex as well. Thus $\tilde{f}_\sigma^\uparrow(\zeta)$ is degenerate.
\end{proof}

\begin{corollary}\label{cor:Can_force_Bsigma_degenerate}
	Let $\scr{I}_{k-1}$ be a compatible $(k-1)$-collection. Suppose that $\sigma:\Delta^k\to \mathcal{L}_X$ is encapsulated, and let $\tau$ be the capsule for $\sigma$. Then there is a subsimplex $\gamma\subset \sf{B}(\tau)$ and a degeneracy operator $s_\alpha$ such that the diagram 
	\[
	\begin{tikzcd}
		\mathcal{R}_\sigma \arrow[r]\arrow[d] & \Delta^{k+\ell}\arrow[d,"s_\alpha(\gamma)"] \\
		\mathcal{X}_\sigma^\uparrow \arrow[r,"\tilde{f}_\sigma^\uparrow"'] & \mathcal{L}_X 
	\end{tikzcd}
	\] 
	commutes.
\end{corollary}

With this corollary in hand we can now return to  \autoref{prop:Compatible_collection}. 

\begin{proof}[Proof (of \autoref{prop:Compatible_collection})]
	First let us observe that the choices of $f_\sigma$ for every $\sigma: \Delta^k \to \mathcal{M}_X \subset \mathcal{L}_X$ are already made since $\mathcal{X}_\sigma= \mathcal{Z}_\sigma$. It is also easy to check that the rest of the conditions hold for those choices. Therefore we can restrict our attention to producing the choices for simplices $\sigma:\Delta^k \to \mathcal{L}_X$ that do not factor through $\mathcal{M}_X$.
	
	We will inductively compatible $k$-collections $\scr{I}_k$ for every $k\geq 1$. Before commencing our argument we will make a preliminary definition. Given $\sigma:\Delta^{k} \to \mathcal{L}_X$ we define $\mathcal{Y}^{\uparrow}_\sigma$ to be the subsimplicial subset (with the inherited decorations) of $\mathcal{Z}_\sigma$ whose simplices are those of $\mathcal{X}^{\uparrow}_\sigma$ in addition to the simplex $\sf{B}(\sigma)$. It follows from the argument given in \autoref{lem:keylemX} that the inclusion $\mathcal{Y}^{\uparrow}_\sigma \to \mathcal{Z}_\sigma$ is \bS-anodyne.
	
	For every $e:\Delta^1 \to \mathcal{L}_X$ we fix the choice of $f_e$ which is guaranteed by \autoref{lem:keylemX}. In this ground case, there are no conditions to check. Let us consider a triangle $\sigma:\Delta^2 \to \mathcal{L}_X$. Using the previous choices the can extend the map $\widetilde{f}_\sigma$ to a morphism
	\[
	\func{f^\uparrow_\sigma: \mathcal{X}_\sigma^{\uparrow} \to \mathcal{L}_X}
	\]
	We distinguish now two cases. Suppose that $\sigma$ is not contained in some $\sf{B}(e)$ for $e:\Delta^1 \to \mathcal{L}_X$. Then we define $f_\sigma$ to be an extension of $f^{\uparrow}_\sigma$ to $\mathcal{Z}_\sigma$. If $\sigma \subseteq \sf{B}(e)$ we can assume that $e \subset \sigma$ since otherwise we have $\sigma \in \mathcal{M}_X$. We extend $f^{\uparrow}_\sigma$ to a map $\mathcal{Y}^\uparrow_\sigma \to \mathcal{L}_X$ by sending $\sf{B}(\sigma)$ to the following simplex:  Let $\sigma_e: \Delta^r \to \mathcal{L}_X$ be the simplex obtained by forgetting every vertex $j$ in $\sf{B}(e)$such that $j\leq \kappa_{\overline{\sigma}}-1$ and such that is not in $\sigma$. We can now map $\sf{B}(\sigma)$ to $s_\alpha(\sigma_e)$ where $\alpha=\kappa_{\overline{\sigma}}-1$ and consequently condition $iii)$ is satisfied. This means that we can construct a compatible 1-collection $\scr{I}_1$.
	
	Now suppose we have a compatible $(k-1)$-collection $\scr{I}_{k-1}$. Let $\sigma: \Delta^k\to \mathcal{L}_X$ be a simplex. If $\sigma$ is not encapsulated, then we may define $f_\sigma$ by solving the lifting problem 
	\[
	\begin{tikzcd}
		\mathcal{X}_\sigma^\uparrow \arrow[d] \arrow[r,"\tilde{f}_\sigma^\uparrow"] & \mathcal{L}_X \\
		\mathcal{Z}_{\sigma}
	\end{tikzcd}
	\]
	using \ref{lem:keylemX}. If $\sigma$ is encapsulated with capsule $\tau$, we can use \autoref{cor:Can_force_Bsigma_degenerate} to define a map 
	\[
	\func{
		\mathcal{Y}_\sigma^\uparrow \to \mathcal{L}_X
	}
	\] 
	which sends $\sf{B}(\sigma)$ to the degenerate simplex described in \autoref{cor:Can_force_Bsigma_degenerate}. Solving the corresponding lifting problem yields an $f_\sigma$ satisfying i), ii), and iii). Thus, we can extend $\scr{I}_{k-1}$ to a compatible $k$-collection, as desired. 
\end{proof}

\begin{proposition}
	The cofibration $u:\mathcal{M}_X \to \mathcal{L}_X$ is \textbf{MB}-anodyne.
\end{proposition}
\begin{proof}
	We say that a simplex $\sigma:\Delta^k \to \mathcal{L}_X$ is \emph{wide} if it is not contained in the image of $u$. Let $\sigma: \Delta^k \to \mathcal{L}_X$ and recall the definition $\nu_\sigma=|\set{j \in [k] \given \overline{\sigma}(j)=n}|$. We produce a filtration
	\[
	\mathcal{M}_X \to S^1 \to S^2 \to \cdots \to \mathcal{L}_X
	\]
	where $S^\ell$ consists of those simplices $\sigma$ in $\mathcal{L}_X$ that either factor through $\mathcal{M}_X$ or they satisfy $\nu_\sigma \leq \ell$. We will fix the convention $S^0=\mathcal{M}_X$. We will show that each step in the filtration is \bS-anodyne. Let us fix once and for all a choice of $f_\sigma:\mathcal{Z}_\sigma \to \mathcal{L}_X$ for every $\sigma:\Delta^k \to \mathcal{L}_X$ with the properties listed in \autoref{prop:Compatible_collection}. First, let us observe that given $\sigma:\Delta^k \to S^\ell$ it follows that the morphisms $f_\sigma$ also factor through $S^\ell$. We can now define $S^{(\ell,s)}$ to consist in those simplices contained in $S^{\ell}$ in addition to the simplices $B(\sigma)$ for $\sigma:\Delta^k \to \mathcal{L}_X$ wide and non-degenerate, such that $k \leq s$ and $\nu_\sigma=\ell+1$. This produces a filtration
	\[
	S^{\ell-1} \to S^{(\ell-1,\ell)} \to  S^{(\ell-1,\ell+1)} \to \cdots \to S^\ell
	\]
	We fix the convention $S^\ell=S^{(\ell,\ell)}$. Let us consider a pullback diagram
	\[
	\begin{tikzcd}[ampersand replacement=\&]
		A_{\sigma} \arrow[r] \arrow[d] \& \mathcal{Z}_{\sigma} \arrow[d,"f_\sigma"] \\
		S^{(\ell,s-1)} \arrow[r] \& S^{(\ell,s)}
	\end{tikzcd}
	\]
	where $\sigma:\Delta^s \to \mathcal{L}_X$ does not factor through $S^{(\ell,s-1)}$. Then it follows by construction that $A_{\sigma}$ contains every simplex of $\mathcal{Z}_\sigma$ except the simplex $\sf{B}(\sigma)$. To check that the top horizontal morphism is \bS-anodyne, it suffices to apply \autoref{lem:innerpivot} with pivot point $\kappa_{\overline{\sigma}}$ after observing that the restriction of $\sf{B}(\sigma)$  to $A_{\sigma}$ consists precisely in the union of the following $(s+\ell-1)$-dimensional faces:
	\begin{itemize}
		\item The face that misses the vertex $j$ for $0\leq j \leq \kappa_{\overline{\sigma}}-1$. This is because this simplex either factors through $\mathcal{Z}_{d_j(\sigma)}$ or it is contained in $\mathcal{M}_X$.
		\item The face that misses the vertex $j$ for $\kappa_{\overline{\sigma}}+\ell \leq j \leq s + \ell$. This is because those faces have strictly smaller parameter $\nu_{d_j(\sigma)}$ if $\nu_\sigma >1$ or they are already in $\mathcal{M}_X$ if $\nu_\sigma=1$.
	\end{itemize}

	To finish the proof we observe that given $\sigma:\Delta^s \to \mathcal{L}_X$ such that $\sigma$ factors through $S^{(\ell,s-1)}$ but $\nu_\sigma=\ell+1$ then it follows by condition $iii)$ in \autoref{prop:Compatible_collection} that $\sf{B}(\sigma)$ is already contained in $S^{(\ell,s-1)}$. This together with previous discussion implies that $S^{(\ell,s-1)} \to S^{(\ell,s)}$ is \bS-anodyne.
\end{proof}

We can distill the key upshot of the preceding technical arguments into a single, simple corollary.

\begin{corollary}\label{cor:hmtpy_pushout_mapsimp}
	For any 2-Cartesian fibration $X\to \Delta^n_\flat$, the square 
	\[
	\begin{tikzcd}
		X_n \times (\Delta^{n-1})^\diamond \arrow[r]\arrow[d] & X_n\times (\Delta^n)^\diamond\arrow[d] \\
		X|_{\Delta^{n-1}} \arrow[r]& X
	\end{tikzcd}
	\] 
	is homotopy pushout. 
\end{corollary}

\subsubsection{The equivalence over a simplex}

Having now established the necessary preliminaries, we turn to the proof that the straightening is an equivalence over the minimally-scaled simplex. With few exceptions, the arguments from here on out are standard, and follow the general shape of the analogous arguments given in \cite{HTT} and \cite{LurieGoodwillie}. We begin with a lemma, which allows us to more easily apply the straightening to our homotopy pushout. 

\begin{lemma}\label{lem:diamond_MS_an}
	Consider the inclusion $(\Delta^{n-1})^\diamond \to (\Delta^n)^\diamond$ as a morphism in $(\mbsSet)_{\Delta^n_\flat}$. Then for every $0\leq i<n$, the induced morphism 
	\[
	\func{
		\psi:\ST_{\Delta^n_\flat}((\Delta^{n-1})^\diamond)(i) \to \ST_{\Delta^n_\flat}((\Delta^{n})^\diamond)(i)
	}
	\]
	is an equivalence of marked-scaled simplicial sets. 
\end{lemma}

\begin{proof}
	To begin, we examine the morphism on underlying marked simplicial sets. Consider the pushouts
	\[
	\begin{tikzcd}
		(\Delta^{n-1})^\diamond \arrow[r]\arrow[d] & ((\Delta^{n-1})^\diamond)^{\triangleright}\arrow[d] \\
		\Delta^n \arrow[r] & X
	\end{tikzcd}
	\]
	and 
	\[
	\begin{tikzcd}
		(\Delta^{n})^\diamond \arrow[r]\arrow[d] & ((\Delta^{n})^\diamond)^{\triangleright}\arrow[d] \\
		\Delta^n \arrow[r] & Y
	\end{tikzcd}
	\]
	and the induced map 
	\[
	\func{\phi:\mathfrak{C}^{\sc}[X](i,\ast) \to \mathfrak{C}^{\sc}[Y](i,\ast)}
	\]
	We first note that, since $i<n$, we have that $\mathfrak{C}^{\sc}[X](i,\ast)=\mathfrak{C}[((\Delta^{n-1})^\diamond)^{\triangleright}](i,\ast)$. 
	
	From the definition, we then have that
	\[
	\mathfrak{C}^{\sc}[X](i,\ast)\cong  \Nerv(\mathbb{P}(\{i+1,\ldots,n-1\}))^\flat
	\]
	and 
	\[
	\mathfrak{C}^{\sc}[Y](i,\ast)\cong \Nerv(\mathbb{P}(\{i+1,\ldots,n\}))^\dagger 
	\]
	where $\dagger$ indicates the marking in which precisely the non-degenerate morphisms $S\to S\cup \{n\}$ are marked. 
	
	We note that, on underlying marked simplicial sets, this means that $\phi$ can be identified with the morphism 
	\[
	\func{
		\mathfrak{C}^{\sc}[X](i,\ast)\times \{0\} \to \mathfrak{C}^{\sc}[X](i,\ast)\times (\Delta^1)^\sharp. 
	}
	\]
	We will show that this yields an equivalence of marked-scaled simplicial sets by showing that both scalings are equivalent to the maximal scaling. 
	
	We claim that the morphisms 
	\[
	\func{f_n^i:\ST_{\Delta^n_\flat}((\Delta^n)^\diamond)(i) \to (  \ST_{\Delta^n_\flat}((\Delta^n)^\diamond)(i))_\sharp}
	\]
	and 
	\[
	\func{
		g_n^i :\ST_{\Delta^n_\flat}((\Delta^{n-1})^\diamond)(i) \to ( \ST_{\Delta^n_\flat}((\Delta^{n-1})^\diamond)(i))_\sharp
	}
	\]
	are \textbf{MS}-anodyne. To show that $f_n^i$ is \textbf{MS}-anodyne it suffices to apply \autoref{lem:saturation}. The argument for $g_n^i$ is similar and left as an exercise.  We thus obtain, for any $i<n$ a commutative diagram 
	\[
	\begin{tikzcd}
		\ST_{\Delta^n_\flat}((\Delta^{n-1})^\diamond)(i) \arrow[r,"\phi"]\arrow[d,"g"',"\sim"] & \ST_{\Delta^n_\flat}((\Delta^{n})^\diamond)(i)\arrow[d,"f","\sim"']\\
		\ST_{\Delta^n_\flat}((\Delta^{n-1})^\diamond)(i)_\sharp \arrow[r,"\phi_\sharp","\sim"'] & \ST_{\Delta^n_\flat}((\Delta^{n})^\diamond)(i)_\sharp
	\end{tikzcd}
	\]
	Showing that $\phi$ is an equivalence of marked-scaled simplicial sets by 2-out-of-3. 
\end{proof}

\begin{lemma}\label{lem:FibresOfStraightening}
	Let $X\to \Delta^n_\flat$ be a 2-Cartesian fibration, and denote by $X_i$ the fibre over $i$. Let $\ST_*$ denote the straightening over $\Delta^0$. Then the map 
	\[
	\func{
		\psi_i^X:\ST_*(X_i)\to \ST_{\Delta^n_\flat}(X)(i)
	}
	\] 
	is an equivalence of marked-scaled simplicial sets. 
\end{lemma}

\begin{proof}
	Following \cite[3.2.3.3]{HTT}, we proceed by induction on $n$. We have already shown the case $n=0$ in \autoref{prop:SToverPt}
	
	By construction, $\psi_n$ is an isomorphism. For $i< n$, we get a canonical commutative diagram
	\[
	\begin{tikzcd}
		& \ST_{\Delta^{n}_\flat}(X|_{\Delta^{n-1}})(i)\arrow[dr,"\gamma_i"]  & \\
		\ST_*(X_i)\arrow[ur]\arrow[rr,"\psi_i^X"'] & & \ST_{\Delta^n_\flat}(X)(i)
	\end{tikzcd}
	\]
	We can identify the upper-left map with $\psi_i^{X|_{\Delta^{n-1}}}$, and so by the inductive hypothesis, it is an equivalence. It thus suffices for us to show that $\gamma_i$ is an equivalence. 
	
	By \autoref{cor:hmtpy_pushout_mapsimp}, we get a homotopy pushout diagram
	\[
	\begin{tikzcd}
		X_n\times (\Delta^{n-1})^\diamond \arrow[r]\arrow[d] & X_n\times (\Delta^n)^\diamond \arrow[d]\\ 
		X|_{\Delta^{n-1}} \arrow[r] & X
	\end{tikzcd}
	\]
	in $(\mbsSet)_{/\Delta^n_\flat}$. Applying the left Quillen functor $\ST_{\Delta^n_\flat}$ yields a homotopy pushout diagram 
	\[
	\begin{tikzcd}
		\ST_{\Delta^n_\flat}(X_n\times (\Delta^{n-1})^\diamond) \arrow[r]\arrow[d] &\ST_{\Delta^n_\flat}( X_n\times (\Delta^n)^\diamond) \arrow[d]\\ 
		\ST_{\Delta^n_\flat}(X|_{\Delta^{n-1}}) \arrow[r,"\gamma"'] & \ST_{\Delta^n_\flat}(X)
	\end{tikzcd}
	\]
	We have a commutative diagram 
	\[
	\begin{tikzcd}
		\ST_{\Delta^n_\flat}(X_n\times (\Delta^{n-1})^\diamond) \arrow[r]\arrow[d] &\ST_{\Delta^n_\flat}( X_n\times (\Delta^n)^\diamond) \arrow[d]\\
		\ST_{\ast}(X_n)\boxtimes \ST_{\Delta^n_\flat}( (\Delta^{n-1})^\diamond)\arrow[r]  & \ST_{\ast}(X_n)\boxtimes \ST_{\Delta^n_\flat}((\Delta^n)^\diamond)
	\end{tikzcd}
	\]
	where the vertical maps are equivalences of marked-scaled simplicial sets by \autoref{thm:product}. It thus suffices to note that, by \autoref{lem:diamond_MS_an}, the induced morphism 
	\[
	\func{
		\psi_i:\ST_{\Delta^n_\flat}((\Delta^{n-1})^\diamond)(i) \to \ST_{\Delta^n_\flat}((\Delta^{n})^\diamond)(i)
	}
	\]
	is an equivalence for any $i<n$. 
\end{proof}

Before continuing, we fix some notation to ease the coming discussion. We will in the following theorem denote the straightening-unstraightening equivalence over the point by 
\[
\begin{tikzcd}
	S: &[-3em] (\mbsSet)_{/\Delta^0_\flat} \arrow[r,shift left] & (\Set_\Delta^{\mathbf{ms}})^{\mathfrak{C}^{\sc}[\Delta^0_\flat]^\op} \arrow[l,shift left] &[-3em] : U 
\end{tikzcd}
\]

\begin{proposition}
	The Quillen adjunction 
	\[
	\begin{tikzcd}
		\ST_{\Delta^n_\flat}: &[-3em] (\mbsSet)_{/\Delta^n_\flat} \arrow[r,shift left] & (\Set_\Delta^{\mathbf{ms}})^{\mathfrak{C}^{\sc}[\Delta^n_\flat]^\op} \arrow[l,shift left] &[-3em] : \UN_{\Delta^n_\flat} 
	\end{tikzcd}
	\]
	is a Quillen equivalence.
\end{proposition}

\begin{proof}
	As in \cite[Lem. 3.2.3.2]{HTT}, we see that $\UN_{\Delta^n_\flat}$ reflects weak equivalences between the images of fibrant objects. It is thus sufficient to show that the derived adjunction unit 
	\[
	\on{Id}\Rightarrow \mathsf{R}(\UN_{\Delta^n_\flat})\circ \ST_{\Delta^{n}_\flat} 
	\]
	is an equivalence. Since $\mathsf{R}(\UN_{\Delta^n_\flat})$ preserves weak equivalences and $\ST_{\Delta^{n}_\flat}$ preserves trivial cofibrations, it is sufficient to check this for fibrant objects. 
	
	Let $X\to \Delta^n_\flat$ be a 2-Cartesian fibration, and let 
	\[
	\func{\ST_{\Delta^{n}_\flat}(X)\to[\sim] \scr{F}} 
	\]
	be a fibrant replacement in $(\Set_\Delta^{\mathbf{ms}})^{\mathfrak{C}^{\sc}[\Delta^n_\flat]^\op}$. We are thus left to show that the induced map 
	\[
	\func{ X\to \UN_{\Delta^n_\flat}(\scr{F})
	}
	\]
	is an equivalence in $(\mbsSet)_{/\Delta^n_\flat}$. Since both objects are fibrant, it suffices to show that this map is a fibrewise equivalence. 
	
	We can identify $\UN_{\Delta^n_\flat}(\scr{F})$ with $U(\scr{F}(i))$. Using the equivalence of \autoref{prop:SToverPt}, we see that the map 
	\[
	\func{X_i\to U(\scr{F}(i))}
	\]
	is an equivalence if and only if the adjoint map 
	\[
	\func{
		S(X_i) \to \scr{F}(i) 
	}
	\]
	is an equivalence. However, we can factor this map as 
	\[
	\begin{tikzcd}
		& \ST_{\Delta^{n}_\flat}(X)(i)\arrow[dr] & \\
		S(X_i)\arrow[ur,"\psi_i^X"]\arrow[rr] & & \scr{F}(i)
	\end{tikzcd}
	\]
	The upper-right map is an equivalence since $\scr{F}$ was a fibrant replacement, and $\psi_i^X$ is an equivalence by \autoref{lem:FibresOfStraightening}. The proposition is thus proven. 
\end{proof}

\begin{corollary}
	Consider the scaled simplicial set $(\Delta^2)_\sharp:=\Nsc([2])$. Then Quillen adjunction 
	\[
	\begin{tikzcd}
		\ST_{\Delta^n_\sharp}: &[-3em] (\mbsSet)_{/\Delta^n_\sharp} \arrow[r,shift left] & (\Set_\Delta^{\mathbf{ms}})^{\mathfrak{C}^{\on{sc}}[\Delta^n_\sharp]^\op} \arrow[l,shift left] &[-3em] : \UN_{\Delta^n_\sharp} 
	\end{tikzcd}
	\]
	is a Quillen equivalence.
\end{corollary}
\begin{proof}
	The key point to note is that base change along the cofibration
	\[
	(\Delta^2)^\sharp_{\flat\subset \sharp}\to (\Delta^2)^\sharp_{\sharp\subset \sharp}
	\]
	induces a fully faithful inclusion 
	\[
	(\mbsSet)_{/\Delta^2_\flat}^\circ \to (\mbsSet)_{/\Delta^2_\sharp}^\circ
	\]
	and similarly, composition with the induced map $\mathfrak{C}^{\sc}[\Delta^2_\flat]\to \mathfrak{C}^\sc[\Delta^2_\sharp]$ induces a fully faithful inclusion 
	\[
	\left((\Set_\Delta^{\mathbf{ms}})^{\mathfrak{C}^\sc[\Delta^2_\sharp]^\op}\right)^\circ \to \left((\Set_\Delta^{\mathbf{ms}})^{\mathfrak{C}^\sc[\Delta^2_\flat]^\op}\right)^\circ 
	\]
	and so we obtain a commutative diagram 
	\[
	\begin{tikzcd}
		\left((\Set_\Delta^{\mathbf{ms}})^{\mathfrak{C}^\sc[\Delta^2_\sharp]^\op}\right)^\circ \arrow[d,hookrightarrow]\arrow[r,"\UN"] & (\mbsSet)_{/\Delta^2_\sharp}^\circ\arrow[d]\\
		\left((\Set_\Delta^{\mathbf{ms}})^{\mathfrak{C}^\sc[\Delta^2_\flat]^\op}\right)^\circ\arrow[r,"\UN"'] & (\mbsSet)_{/\Delta^2_\flat}^\circ
	\end{tikzcd}
	\]
	of simplicial categories. 
	
	The remainder of the proof is, \emph{mutatis mutandis}, that of \cite[Prop. 3.8.7]{LurieGoodwillie}.
\end{proof}

\subsection{Straightening in general}\label{subsec:STequivGen}

We now prove the main theorem of this paper.

\begin{theorem}\label{thm:QE_general}
	Let $S\in \scsSet$ be a scaled simplicial set, and let $\phi:\mathfrak{C}^{\sc}[S]\to \mathcal{C}$ be an equivalence of $\msSet$-enriched categories. The Quillen adjunction 
	\[
	\begin{tikzcd}
		\ST_\phi: &[-3em] (\mbsSet)_{/S} \arrow[r,shift left] & (\mssSet)^{\mathcal{C}^\op}\arrow[l,shift left] &[-3em] :\UN_\phi 
	\end{tikzcd}
	\]
	is a Quillen equivalence.
\end{theorem}

Coupled with the fact, discussed immediately hereafter, that $\UN_\phi$ is a $\msSet$-enriched functor, this will immediately imply a stronger result --- the functor of $\msSet$-enriched categories of fibrant-cofibrant objects induces an equivalence of $\infty$-bicategories.

The argument from here on out is standard, and follows the same path as \cite[Section 3.8]{LurieGoodwillie}. Our first aim will be to show that, for any scaled simplicial set $S$, the functor 
\[
\func{
	\UN_{\phi}: (\Set_\Delta^{\mathbf{ms}})^{\mathfrak{C}^{\sc}[S]^\op} \to (\mbsSet)_{/S}
}
\]
is, in fact, an $\Set_\Delta^+$-enriched functor. 

The $\Set_\Delta^+$-enrichment on $\UN_{\phi}$ is given as follows. Let $\scr{F},\scr{G}: \mathcal{C}^\op\to \Set_\Delta^{\mathbf{ms}}$ be $\Set_\Delta^+$-enriched functors, and $K\in \msSet$. A map 
\[
\func{
	K\to \Map^+(\scr{F},\scr{G})
}
\]
is equivalently a map $\scr{F}\otimes K\to \scr{G}$, where $(\scr{F}\otimes K)(s):=\scr{F}(s)\times K_\sharp$. We then have a natural map 
\[
\UN_\phi(\scr{F})\times K_{\sharp\subset \sharp}\to \UN_\phi(\scr{F})\times \UN_*(K_\sharp) 
\]
Where the second component is induced by the natural transformation $\alpha:\ST_*\Rightarrow L$. We can then write down a natural composite map 
\[
\UN_\phi(\scr{F})\times K_{\sharp\subset \sharp}\to \UN_S(\scr{F})\times \UN_*(K_\sharp) \to \UN_\phi(\scr{F}\otimes K)\to \UN_\phi(\scr{G})
\]
Which is equivalently a map $K\to \Map^{\on{th}}(\UN_S(\scr{F}),\UN_S(\scr{G}))$. The naturality guarantees that this defines a map of simplicial sets 
\[
\Map^+(\scr{F},\scr{G})\to \Map^{\on{th}}(\UN_{\phi}(\scr{F}),\UN_\phi(\scr{G})).
\]
Similarly, since the composition maps in both cases are defined via the diagonal $\Delta^n\to \Delta^n\times \Delta^n$, naturality ensures that this defines an enriched functor. A wholly analogous argument shows that $\UN_S$ can also be viewed as a simplicially-enriched functor. 

\begin{proof}[Proof (of \autoref{thm:QE_general})]
	The proof is now nearly identical to that of \cite[Prop. 3.8.4]{LurieGoodwillie}. The argument hangs on the claim that the functor
	\[
	\func*{
		F: (\scsSet)^{\op}\to \Cat_\Delta;
		S \mapsto ((\Set_\Delta^{\mathbf{ms}})_f^{\mathfrak{C}^{\sc}[S]^\op})[\scr{W}_S^{-1}] 
	}
	\]
	sends pushouts along cofibrations to homotopy pullbacks, and sends transfinite composites of cofibrations to homotopy limits, which follows from the argument given in loc. cit.  
\end{proof}

\begin{corollary}
	Let $S\in \scsSet$ be an $\infty$-bicategory. The $\Set_\Delta^+$-enriched functor $\UN_S$ induces an equivalence of $\infty$-bicategories 
	\[
	\func{\Nsc\left(\left((\Set_\Delta^{\mathbf{ms}})^{\mathfrak{C}^{\sc}[S]^\op}\right)^\circ\right) \to \Nsc\left(\left((\Set_\Delta^{\mathbf{mb}})_{/S}\right)^\circ\right). }
	\]
\end{corollary}

\begin{proof}
	This follows immediately from \autoref{thm:product}, \autoref{thm:QE_general}, and \cite[A.3.1.10]{HTT}.
\end{proof}

One final step is left: to interpret this result internally to marked-scaled simplicial sets. 

\begin{definition}
	The $\Set_\Delta^+$-enrichment on $\mssSet$ equips the full subcategory $(\mssSet)^\circ$ of fibrant-cofibrant objects with the structure of a fibrant $\msSet$-enriched category. We denote by $\bcat{B}\!\on{icat}_\infty:=\Nsc((\mssSet)^\circ)$ the homotopy-coherent scaled nerve of this $\msSet$-category (considered as a scaled simplicial set). We refer to $\bcat{B}\!\on{icat}_\infty$ \glsadd{BicatInfty} as the \emph{$\infty$-bicategory of $\infty$-bicategories.} 
	
	Similarly, for $S\in \scsSet$, we denote by $2\bcat{C}\!\on{art}(S):=\Nsc\left(\left((\Set_\Delta^{\mathbf{mb}})_{/S}\right)^\circ\right)$ the \emph{$\infty$-bicategory of 2-Cartesian fibrations over $S$}. \glsadd{CartINfty} 
\end{definition} 

\begin{remark}
	Formally, considering $\mssSet$ as the category of all $\scr{U}$-small marked-scaled simplicial sets for some Grothendieck universe $\scr{U}$, the marked-scaled simplicial set $\bcat{B}\!\on{icat}_\infty$ is no longer small. We thus resort to fixing a new Grothendieck universe $\scr{V}$ in which $\scr{U}$, and thus $\bcat{B}\!\on{icat}_\infty$, becomes $\scr{V}$-small. 
\end{remark}  

\begin{proposition}
	Let $\scr{C}$ be a small $\msSet$-enriched category, $S$ a small scaled simplicial set,  $\phi:\mathfrak{C}^{\sc}[S]\to \scr{C}$ an equivalence of $\msSet$-enriched categories, and $\mathbf{A}$ a combinatorial, $\msSet$-enriched model category. Endow $\mathbf{A}^{\scr{C}}$ with the projective model structure. Then the functor 
	\[
	\Nsc((\mathbf{A}^{\scr{C}})^\circ )\to \Fun(S,\Nsc(\mathbf{A}^\circ))
	\]
	is a bicategorical equivalence of scaled simplicial sets. 
\end{proposition}

\begin{proof}
	The proof is that of \cite[Prop. 4.2.4.4]{HTT}. The only thing that changes is the exchange of $\Set_\Delta$ for $\mssSet$, and as both of these are excellent model categories, no further emendation is necessary.  
\end{proof}

\begin{corollary}\label{cor:intrinsincbicat}
	Let $S\in \scsSet$. There is an equivalence of $\infty$-bicategories
	\[
	2\bcat{C}\!\on{art}(S)\simeq \Fun(S^\op,\bcat{B}\!\on{icat}_\infty).
	\]
\end{corollary}

Many of the corollaries of the $(\infty,1)$-categorical straightening-unstraightening equivalence generalize directly. In particular, the following results generalize Corollary 3.3.1.1, Corollary 3.3.1.2, and Proposition 3.3.1.7 from \cite{HTT}.

\begin{corollary}
	Let $f:T\to S$ be a bicategorical equivalence between scaled simplicial sets. Then the pullback functor
	\[
	\func{
		f^\ast: (\mbsSet)_{/S} \to (\mbsSet)_{/T}	
		}
	\]
	is a right Quillen equivalence. 
\end{corollary} 
\begin{proof}
	In this case, three of the four functors in the commutative square of \autoref{rmk:basechange} induce equivalences on underlying $\infty$-categories. As a result, the right derived functor of $f^\ast$ is an equivalence. 
\end{proof}

\begin{corollary}
	Let $p:X\to S$ be a 2-Cartesian fibration between scaled simplicial sets, and $f:S\to T$ a bicategorical equivalence of scaled simplicial sets. Then there is a 2-Cartesian fibration $q: Y\to T$ and an equivalence of 2-Cartesian fibrations over $S$ $X\simeq S\times_T Y$. 
\end{corollary}

\begin{corollary}
	Let $S$ be a scaled simplicial set, and let $p:X\to S$ be a fibrant object in $(\mbsSet)_{/S}$. Then the underlying map of scaled simplicial sets is a bicategorical fibration. 
\end{corollary}

\begin{proof}
	The proof is nearly verbatim that of \cite[Prop. 3.3.1.7]{HTT}. 
\end{proof}

\section{The Relative 2-Nerve}\label{app:Relnerve}

There is a special case of most $\infty$-categorical Grothendieck constructions in which the computation of the right adjoints can be greatly simplified. When the base is suitably strict, it is possible to define a \emph{relative nerve}, which computes the Grothendieck construction of a functor. The aim of this appendix is to provide a relative nerve construction which takes as input a $\on{Set}_\Delta^+$-enriched functor 
\[
\func{
	F:\CC^{\op}\to \Set_\Delta^{\mathbf{ms}}
}
\]
and yields as output a 2-Cartesian fibration $\cchi_{\CC}(F)\to \Nsc(\CC)$. In form, this relative nerve will actually seem slightly \emph{more} complicated than the associated straightening functor. However, it will enable us to more easily make the comparison with the strict 2-categorical relative nerve construction of \cite{Buckley}. The particular virtue of our relative 2-nerve construction in this regard is that, given a strict 2-functor 
\[
\func{F:\CC^\op\to 2\!\Cat,}
\]
we can compute the relative 2-nerve in terms of strict 2-functors into $\CC$ and $F(x)$, without first passing to simplicial sets. 

In our previous papers \cite{AGSRelNerve} and \cite{AGSQuillen}, we defined two variants of the relative 2-nerve, which provided $\infty$-bicategories fibred in $(\infty,1)$-categories. In this section, we will upgrade the later of these constructions to provide the desired $\cchi_\CC$. 

\begin{remark}
	Our choice of notation $\cchi_\CC$ for the relative 2-nerve of a functor $F:\CC^{\op} \to \on{Set}_\Delta^{\mathbf{ms}}$ does in fact collide with the choice of notation in \cite{AGSRelNerve} and \cite{AGSQuillen}. An ideal choice of notation would involved a superscript $\cchi_\CC^{\epsilon}$ where $\epsilon$ denotes one of the four variances for bicategorical fibrations. We will use this rather abusive notation to improve readibility since we will only consider the \emph{outer Cartesian} variance.
\end{remark}

\begin{definition}
	Given a totally ordered set $I$, the 2-category $\OO^{I}_{i \upslash}$ has 
	\begin{itemize}
		\item Objects given by subsets $S \subseteq I$ such that $\min(S)=i$.
		\item Each mapping category $\OO^{I}_{i \upslash}(S,T)$ is a poset whose objects $\func{\mathcal{U}:S \to T}$ are given by subsets $\mathcal{U} \subseteq I$ such that
		\[
		\min(\mathcal{U})=\max(S), \enspace \max(\mathcal{U})=\max(T), \enspace   S \cup \mathcal{U}\subseteq T,
		\] 
		ordered by inclusion.
		\item Composition is given by union.
	\end{itemize}
	These lax slice categories piece together into a 2-functor 
	\[
	\func*{
		(\OO^I)^\op\to 2\!\Cat;
		i\mapsto \OO^I_{i\upslash}
	}
	\]
	so that, in particular for any $J\subset I$ with $i=\min(I)$ and $j=\min(J)$, we have 2-functors 
	\[
	\func{
		\omega_{I,J}: \OO^I(i,j)\times \OO^J_{j\upslash} \to \OO^I_{i\upslash}
	}
	\]
	given on objects by the union of sets.  It is an easy check that these functors are injective on objects, 1-morphisms, and 2-morphisms.
\end{definition}

The 2-categories $\OO^I_{i\upslash}$ play a central role in our relative nerve construction. 

\begin{construction}
	Let 
	\[
	\func{
		F: \CC^\op\to \Set_\Delta^{\mathbf{ms}}
	}
	\]
	be a $\Set_\Delta^+$-enriched functor. We define a marked-biscaled simplicial set $\cchi_\CC(F)$ as follows. An $n$-simplex $\Delta^n\to \cchi_\CC(F)$ consists of 
	\begin{itemize}
		\item A simplex $\sigma:\Delta^n_\flat\to \Nsc(\CC)$. 
		\item For every $\varnothing\neq I\subset [n]$ with $\min(I)=i$, a map of marked-scaled simplicial sets 
		\[
		\func{
			\theta_I: \Nms(\OO^I_{i\upslash})^\flat \to F(\sigma(i)) 
		}
		\] 
		such that, for every $\varnothing \neq J\subset I\subset [n]$ with $\min(J)=j$ and $\min(i)=i$, the diagram 
		\[
		\begin{tikzcd}
			\Nms(\mathbb{O}^I(i,j)) \times \Nms(\OO^J_{j\upslash}) \arrow[r]\arrow[d,"\Nerv(\sigma)\times\theta_J"'] & \Nms(\OO^I_{i\upslash}) \arrow[d,"\theta_I"] \\
			\Nms(\CC(\sigma(i),\sigma(j))) \times F(\sigma(j))\arrow[r,"F(-)"']  & F(\sigma(i))  
		\end{tikzcd}
		\]
		commutes.
	\end{itemize}
	We then define markings and scalings on $\rho_\CC(F)$.
	\begin{itemize}
		\item A 1-simplex $\Delta^1\to \cchi_\CC(F)$ is marked if the corresponding map $\theta_{[1]}:\Nms(\OO^1_{0\upslash})\to F(\sigma(i))$ descends to a map 
		\[
		\func{
			\Nms(\OO^1_{0\upslash})^\sharp\to F(\sigma(0)).
		}
		\]  
		\item A 2-simplex $\Delta^2\to \cchi_\CC(F)$ is lean if the corresponding map 
		\[
		\func{\Nms(\OO^2_{0\upslash})\to F(\sigma(0))}
		\]
		descends to a map 
		\[
		\func{\Nms(\OO^2_{0\upslash})_\sharp\to F(\sigma(0))}.
		\]
		\item A 2-simplex  $\Delta^2\to \cchi_\CC(F)$ is thin if and only if it is lean \emph{and} the corresponding 2-simplex $\sigma:\Delta^2 \to \Nsc(\CC)$ is thin. 
	\end{itemize}
	Note that there is a canonical forgetful functor
	\[
	\func*{
		\cchi_{\CC}(F)\to \Nsc(\CC);
		(\sigma,\{\theta_I\})\mapsto \sigma 
	}
	\]
	which sends thin triangles to thin triangles. 
\end{construction}

\begin{definition}
	The \emph{relative bicategorical nerve} \glsadd{RhoC} over a 2-category $\CC$ is the functor 
	\[
	\func*{
		\cchi_\CC: \left(\Set_\Delta^{\mathbf{ms}}\right)^{\CC^\op} \to (\Set_\Delta^{\mathbf{mb}})_{/\Nsc(\CC)};
		F \mapsto \cchi_\CC(F) 
	}
	\]
	By the adjoint functor theorem, $\cchi_\CC$ admits a left adjoint, which we will denote by 
	\[
	\func{
		\PPhi_\CC: (\Set_\Delta^{\mathbf{mb}})_{/\Nsc(\CC)} \to \left(\Set_\Delta^{\mathbf{ms}}\right)^{\CC^\op}.
	}
	\]
\end{definition}

%
%

\begin{lemma}
	The functor $\cchi_\CC$ preserves trivial fibrations. 
\end{lemma}

\begin{proof}
	We need only check that the lifting problems 
	\[
	\begin{tikzcd}
		A\arrow[d,"f",hookrightarrow]\arrow[r] & \cchi_\CC(F)\arrow[d,"\cchi_\CC(\mu)"] \\
		B\arrow[r] & \cchi_\CC(G)
	\end{tikzcd}
	\]
	have solutions when $\mu:F\Rightarrow G$ is a projective (pointwise) trivial fibration and $f:A\to B$ is a generating cofibration of marked-biscaled simplicial sets. The proof is virtually identical to the proof of \cite[Prop. 3.0.11]{AGSQuillen}.
\end{proof}

\begin{corollary}
	The functor $\PPhi_{\CC}$ preserves cofibrations.
\end{corollary}

\subsection{Identifying $\PPhi_{\OO^n}$}

Let $\CC$ be a 2-category. Then we can define a 2-functor
\[
\func{
	\CC^\op \to 2\!\Cat;
	c \mapsto \CC_{c\upslash}
}
\]
that maps a 1-morphism $f:c \to d$ to the functor $f^*: \CC_{d\upslash} \to \CC_{c\upslash}$ given by precomposition with $f$. It easy to verify that given a 2-morphism $\alpha:f \Rightarrow g$ we can construct a natural transformation $f^* \Rightarrow g^*$ whose component at an object $u:d \to x$ is given by $\alpha*u$. Passing to $\msSet$-enriched categories we thus obtain, for any strict 2-category $\CC$, a $\msSet$-enriched functor 
\[
\func{
	\CC_{-\upslash}:\CC^\op \to \Set_\Delta^{\mathbf{ms}};
	c \mapsto \Nms(\CC_{c\upslash}) 
}
\]
\begin{definition}
	In the particular case where $\CC=\OO^n$, we will denote the functor constructed above by 
	\[
	\func{
		\mathfrak{O}^n:(\OO^n)^\op \to \Set_\Delta^{\mathbf{ms}} 
	}
	\]
\end{definition}

\begin{notation}
	The canonical normal lax functor $\xi_n:[n]\to \OO^n$ gives rise to an inclusion of scaled simplicial sets which we denote by
	\[
	\func{p_n:\Delta^n_\flat \to \Nsc(\OO^n).}
	\]
	We will equip $\Delta^n$ with the minimal marking and lean scaling, and conventionally view $p_n$ as an object in $\left(\Set_\Delta^{\mathbf{mb}}\right)_{/\Nsc(\OO^n)}$. 
\end{notation}

\begin{lemma}
	Let $F:(\OO^n)^\op\to \Set_\Delta^{\mathbf{ms}}$ be a $\Set_\Delta^+$-enriched functor. There is a natural bijection 
	\[
	\on{Nat}_{\CC^\op}(\mathfrak{O}^n,F)\cong \Hom_{\left(\Set_\Delta^{\mathbf{mb}}\right)_{/\Nsc(\OO^n)}}(p_n, \cchi_\OO(F)).
	\]
	Consequently, we have an equivalence of $\Set_\Delta^+$-enriched functors $\PPhi_{\OO^n}(p_n)\cong \mathfrak{O}^n$. 
\end{lemma}

\begin{proof}
	Follows immediately from unwinding the definitions.
\end{proof}

\begin{corollary}
	Denote by 
	\[
	\func{
		p_1^\sharp: (\Delta^1)^\sharp \to \Nsc(\OO^1),
	}
	\]
	\[
	\func{(p_2)_{\flat\subset \sharp}:(\Delta^2)^\flat_{\flat\subset \sharp} \to \Nsc(\OO^2)}, \text{ and} 
	\]
	\[
	\func{(p_2)_{\sharp}:(\Delta^2)^\flat_{\sharp} \to \Nsc(\OO^2)_\sharp}
	\]
	the obvious decorated versions of the $p_n$. Then 
	\[
	\func{
		\PPhi_{\OO^{1}}(p_1^\sharp): (\OO^1)^\op \to (\Set_\Delta^{\mathbf{ms}});
		i \mapsto \Nms(\OO^1_{i\upslash})^\sharp
	}
	\]
	\[
	\func{
		\PPhi_{\OO^{2}}((p_2)_{\flat\subset\sharp}): (\OO^2)^\op \to (\Set_\Delta^{\mathbf{ms}});
		i \mapsto \Nms(\OO^2_{i\upslash})_\sharp 
	}, \text{ and}
	\]
	\[
	\func{
		\PPhi_{\OO^{2}}((p_2)_{\flat\subset\sharp}): (\OO^2)^\op_\sharp \to (\Set_\Delta^{\mathbf{ms}});
		i \mapsto \Nms(\OO^2_{i\upslash})^\dagger_\sharp 
	}
	\]
	where $\dagger$ denotes the marking in which the unique morphism $02\to 012$ is marked. 
\end{corollary}

\begin{proof}
	All identifications except the last are immediate from the definitions. The additional marking in the final case follows from the necessity that the functor have source $\OO^2_\sharp$.  
\end{proof}

\begin{notation}
	We will denote the three functors above by $(\mathfrak{O}^1)^\sharp$, $(\mathfrak{O}^2)_{\flat\subset \sharp}$, and $(\mathfrak{O}_2)_\sharp$, respectively. 
\end{notation}

\subsection{Identifying $\ST_{\OO^n}$}

Our comparison will be with a very specific version of the straightening functor:

\begin{notation}
	For a 2-category $\CC$, we view $\CC$ as a $\Set_\Delta^+$-enriched category. The counit $\epsilon_\CC:\mathfrak{C}^{\sc}(\Nsc(\CC))\to \CC$ is an equivalence of $\Set_\Delta^+$-enriched categories. We will denote by 
	\[
	\func{
		\ST_\CC:(\mbsSet)_{/\Nsc(\CC)}\to (\Set_\Delta^{\mathbf{ms}})^{\CC^\op} 
	}    
	\]
	the relative straightening functor $\ST_{\epsilon_\CC}$. 
\end{notation}

We now unravel the definitions to characterize $\ST_{\OO^n}(p_n)$. By construction, the underlying functor to $\Set_\Delta^+$ is given by the Yoneda embedding on the $\Set_\Delta^+$-enriched category 
\[
\OO^n \coprod_{\mathfrak{C}^{\sc}(\Nsc(\OO^n))} \mathfrak{C}^{\sc}(\Nsc(\OO^n)) \coprod_{\mathfrak{C}^{\sc}(\Delta^n_\flat)} \mathfrak{C}^{\sc}((\Delta^n_\flat)^\triangleright)
\]
We note that $\OO^n=\mathfrak{C}^{\sc}(\Delta^n_\flat)$, and by the triangle identities for the adjunction $\mathfrak{C}^{\sc}\dashv \Nsc$, we see that the induced map 
\[
\mathfrak{C}^{\sc}(\Delta^n_\flat)\to \OO^n
\]
is simply the identity. The pushout above thus collapses to simply $\mathfrak{C}^{\on{sc}}((\Delta^n_\flat)^\triangleright)$.  We can then describe the marked-scaled simplicial set $\ST_{\OO^n}(\Delta^n_\flat)(i)$ as the poset $\mathcal{L}^n_\flat(i)$ described in \autoref{def:Lni}. To ease the notation let us denote $\mathcal{L}^n_\flat(i)$ simply by $\mathcal{L}^n_i$.

\begin{construction}
	We construct a morphism of marked-scaled simpliial sets $\eta^n_i:\mathcal{L}^n_i\to \Nms(\OO^n_{i\upslash})$ whose underlying map of simplicial sets is given by (the nerve of) a normal lax functor defined as follows:
	\begin{itemize}
		\item On objects, $S\mapsto S$. 
		\item On morphisms $S\subset T$ is sent to the mophism $\{\max(S),\max(T)\}:S\to T$. 
	\end{itemize}
	The fact that, for $S\subset T\subset V$, we have $\{\max(S),\max(V)\}\subset \{\max(S),\max(T),\max(V)\}$ gives us our compositors. The fact that if $S=T$, we have $\{\max(S),\max(T)\}=\{\max(S)\}$ gives strict unitality. Since both marked-scaled simplicial sets carry the minimal marking we only need to check that $\eta^n_i$ preserves the scaling. Let $S_0\subset S_1\subset S_2$ be a 2-simplex in the source. If there are $i,j$ such that $\max(S_i)=\max(S_j)$, then it follows immediately that $\{\max(S_0),\max(S_1),\max(S_2)\}=\{\max(S_0),\max(S_2)\}$. 
\end{construction} 

The following lemma follows immediately from our definitions.

\begin{lemma}
	The maps $\eta^n_i$ define natural transformations of $\Set_\Delta^+$-enriched functors $\eta^n:\ST_{\OO^n}(\Delta^n_\flat)\to \mathfrak{O}^n$. 
\end{lemma}

\begin{proposition}\label{prop:CompareEquivonFlat}
	The morphisms $\eta_i^n: \ST_{\OO^n}(\Delta^n_\flat)(i)\to \mathfrak{O}^n(i)$ are equivalences of marked-scaled simplicial sets. 
\end{proposition}

We will prove this proposition in a series of lemmata. Since both simplicial sets are equipped with the minimal marking, it suffices to show that the map is an equivalence on underlying scaled simplicial sets by \autoref{thm:QE_mb_ms_sc}. Since $\mathfrak{O}^n(i)=\Nsc(\OO^n_{i\upslash})$, it suffices to show that the induced map
\[
\func{
	\xi^n_i:\mathfrak{C}^{\sc}[(\ST_{\OO^n}(\Delta^n_\flat))(i)]\to \OO^n_{i\upslash} 
}
\]
is an equivalence of $\Set_\Delta^+$-enriched categories. Since this map is clearly bijective on objects, it suffices to check that the induced morphisms on mapping spaces are equivalences.

In both cases, the mapping spaces are nerves of posets. 
\begin{itemize}
	\item For $S,T\in \OO^n_{i\upslash}$, the mapping space $\mathfrak{O}^n_i(S,T)$  is the poset of chains $U\subset [n]$ such that $\min(U)=\max(S)$, $\max(U)=\max(T)$, and $S\cup U\subset T$. Equivalently, this is the poset $\OO^T(\max(S),\max(T))$, equipped with the minimal marking.
	\item For $S,T\in Q^n_i$, the mapping space $\mathfrak{C}^{\sc}[(\ST_{\OO^n}(\Delta^n_\flat))(i)](S,T)$ is the poset of chains 
	\[
	S\subset S_1\subset \cdots \subset S_k\subset T
	\]
	in $\mathcal{L}^n_i$. An inclusion $\vec{S}\subset \vec{U}$ is marked if and only if, for every $S_i\subset S_{i+1}$ in $\vec{S}$, every $U_\ell\in \vec{U}$ between $S_i$ and $S_{i+1}$, either $\max(U_\ell)=\max(S_i)$ or $\max(U_\ell)=\max(S_{i+1})$. Notice that $T\setminus S$ could have elements lower than $\max(S)$.  
\end{itemize}

The map 
\[
\xi^n_i: \mathfrak{C}^{\sc}[(\ST_{\OO^n}(\Delta^n_\flat))(i)](S,T) \to \mathfrak{O}^n_i(S,T)
\]
sends a chain $S\subset S_1\subset \cdots \subset S_k\subset T$ to the chain 
\[
\xi^n_i(\vec{S})=\bigcup_{V\in \vec{S}} \{\max(V)\}.
\]

\begin{definition}
	For ease of notation, we define
	\[
	\mathbb{L}^n_i:=\mathfrak{C}^{\sc}[(\ST_{\OO^n}(\Delta^n_\flat))(i)] 
	\]
	For any $S,T\in \mathbb{L}^n_i$, we define a full subposet $\mathbb{B}^{n}_i(S,T)\subset \mathbb{L}^n_i(S,T)$ consisting of chains 
	\[
	S\subset [T,s_1]\subset \cdots \subset [T,s_k]\subset T 
	\]
	where we define for every $s \in T$ the subset $[T,s]=\set{r \in T \given r \leq s}$.
\end{definition}

\begin{lemma}
	An inclusion $\vec{S}\subset \vec{U}$ represents a marked morphism in $\mathbb{L}^n_i(S,T)$ if and only if its image under $\xi^n_i$ is degenerate. 
\end{lemma}

\begin{proof}
	Immediate from the definition.
\end{proof}

\begin{lemma}\label{lem:BBprojEquiv}
	The restriction of the map $\xi^{n}_i$ 
	\[
	\func{\xi^n_i:\mathbb{B}^n_i(S,T)\to \mathbb{O}^T(\max(S),\max(T))}
	\]
	is an equivalence of marked simplicial sets.
\end{lemma}

\begin{proof}
	We define a map 
	\[
	\func*{\gamma:\OO^T(\max(S),\max(T))\to \BB^n_i(S,T)
	}
	\]
	which sends $\max(S)<s_1<\cdots <s_k<\max(T)$ to the chain
	\[
	S\subset [T,s_1]\subset \cdots \subset [T,s_k]\subset T.
	\]
	We then note that $\xi^n_i\circ \gamma=\id$. We claim that $\gamma\circ \xi^n_i\leq \id$, which yields a marked homotopy $\gamma\circ \xi^n_i$ to $\id$. To prove the claim we note that $\gamma \circ \xi^n_i(\vec{S})$ is given by $\vec{S}$ if $S_1 \neq [T,s_0]$ with $s_0=\max(S)$ or by $\vec{S}\setminus [T,s_0]$ in which case the existence of the marked morphism $\gamma \circ \xi^n_i(\vec{S}) \to \vec{S}$ follows immediately.
\end{proof}

\begin{lemma}\label{lem:BBequivasubposet}
	The inclusion $\iota:\mathbb{B}^n_i(S,T)\to \mathbb{L}^n_i(S,T)$ is an equivalence of marked simplicial sets. 
\end{lemma}

\begin{proof}
	Let $s_j \in T$. We define $\mathbb{L}^s \subset \mathbb{L}$ as the full subposet consisting of those chains 
	\[
	\vec{S}=S \subset S_1 \subset \cdots \subset S_k \subset T,
	\] 
	such that $S_i=[T,s_i]$ whenever $s_i \geq s_j$. Note that if $s_j\leq s_0=\max(S)$ then it follows that $\mathbb{L}^s=\mathbb{B}$. Let $T_S=\set{s_j \in T \given s_j \geq s_0}$ and consider a filtration
	\[
	\mathbb{B}=\mathbb{L}^{s_0} \subset \mathbb{L}^{s_1}\subset \cdots \subset \mathbb{L}^{s_m} \subset \mathbb{L}^{s_m+1}=\mathbb{L}, \enspace \text{ with }s_m=\max(T)
	\]
	Our goal is to show that each step in the filtration is a weak equivalence of marked simplicial sets. We denote by $\iota_j: \mathbb{L}^{j} \to \mathbb{L}^{j+1}$ for $j=0,\dots,s_m$. Let $\vec{S}=S \subset S_1 \subset \cdots \subset S_k \subset T$ be an object of $\mathbb{L}^{j+1}$ we construct a new chain $\pi_j(\vec{S})$ by replacing each $S_\ell$ with $s_{\ell} \geq s_{j}$ with its corresponding $[T,s_\ell]$. This definition yields a functor
	\[
	\func{\pi_j: \mathbb{L}^{j+1} \to \mathbb{L}^j; \vec{S} \mapsto \pi_j(\vec{S})}
	\]
	such that $\pi_j \circ \iota_j=\on{id}$. Let $\zeta_j=\iota_j \circ \pi_j$. We construct a functor
	\[
	\func{\theta_j:\mathbb{L}^{j+1} \to \mathbb{L}^{j+1}}
	\]
	that appends to each chain $\vec{S} \in \mathbb{L}^{j+1}$ the object $[T,s_j]$ if there exists some $S_\ell \in \vec{S}$ such that $\max(S_{\ell})=s_j$ or leaves the chain untouched otherwise. Note that if $s_j=s_m$ then this functor is the identity. We also observe that we have a natural transformation $\on{id} \leq \theta_j$ and $\zeta_j \leq \theta_j$ whose components are marked. It follows that each $\iota_j$ is a weak equivalence and consequently so is $\iota$.
	
\end{proof}

\begin{proof}[Proof (of \autoref{prop:CompareEquivonFlat})]
	We simply apply \autoref{lem:BBequivasubposet}, \autoref{lem:BBprojEquiv}, and 2-out-of-3. 
\end{proof}

Turning now to the cases $(p_1)^\sharp$, $(p_2)_{\flat\subset \sharp}$, and $(p_2)_\sharp$, we see that the corresponding straightenings are obtained from $\ST_{\OO^1}(p_1)$ and $\ST_{\OO^2}(p_2)$ by maximally marking or maximally scaling the values of the functors, respectively. We then have the following 

\begin{corollary}\label{cor:compareequivnotflat}
	The transformations $\xi^n$, $n=1,2$ induce equivalences of enriched functors 
	\[
	\func{(\xi^1)^\sharp:\ST_{\OO^1}(p_1^\sharp)\to (\mathfrak{O}^1)^\sharp},
	\]
	\[
	\func{(\xi^2)_{\flat\subset \sharp}:\ST_{\OO^2}((p_2)_{\flat\subset \sharp})\to (\mathfrak{O}^2)_{\flat\subset\sharp}}, \text{ and}
	\]
	\[
	\func{(\xi^2)_{ \sharp}:\ST_{\OO^2_\sharp}((p_2)_{ \sharp})\to (\mathfrak{O}^2)_{\sharp}}
	\]
\end{corollary}

\begin{proof}
	The morphism $(\xi^1)^\sharp$ is an isomorphism, and it is a quick check to extend the previous arguments to cover the case $(\xi^2)_{\flat\subset \sharp}$. One then notes that, for each $i\in \OO^2$, the $i$-component of $(\xi^2)_\sharp$ is a pushout of the $i$-component of $(\xi^2)_{\flat\subset \sharp}$ along the inclusion $(\Delta^1)^\flat\to (\Delta^1)^\sharp$, and thus is an equivalence.
\end{proof}

\begin{remark}
	As in \cite[Prop. 4.1.1]{AGSRelNerve} any 2-functor $f:\CC\to \DD$ yields diagrams 
	\[
	\begin{tikzcd}
		(\Set_\Delta^{\mathbf{ms}})^{\DD^\op}\arrow[r,"f^\ast"]\arrow[d,"\cchi_\DD"'] & (\Set_\Delta^{\mathbf{ms}})^{\CC^\op}\arrow[d,"\cchi_\CC"] \\
		(\mbsSet)_{/\Nsc(\DD)}\arrow[r,"\Nsc(f)^\ast"'] & (\mbsSet)_{/\Nsc(\CC)}
	\end{tikzcd}
	\]
	and 
	\[
	\begin{tikzcd}
		(\Set_\Delta^{\mathbf{ms}})^{\DD^\op}\arrow[r,leftarrow,"f_!"]\arrow[d,leftarrow,"\PPhi_\DD"'] & (\Set_\Delta^{\mathbf{ms}})^{\CC^\op}\arrow[d,leftarrow,"\PPhi_\CC"] \\
		(\mbsSet)_{/\Nsc(\DD)}\arrow[r,leftarrow,"\Nsc(f)_!"'] & (\mbsSet)_{/\Nsc(\CC)}
	\end{tikzcd}
	\]
	which commute up to natural isomorphism.
\end{remark}

\begin{theorem}\label{thm:nat_equiv_St_rho}
	There exists a unique family of natural weak equivalences 
	\[
	\func{
		\xi^\CC(X): \ST_\CC\nat \PPhi_\CC 
	}
	\]
	indexed by pairs $(\CC,X)$ consisting of a 2-category $\CC$ and $X\in (\mbsSet)_{/\Nsc(\CC)}$ with the following properties. 
	\begin{enumerate}
		\item On the maps $p_n$ for $n\geq 0$, $p_1^\sharp$, $(p_2)_{\flat\subset \sharp}$, and $(p_2)_\sharp$, the transformations $\xi^\CC(X)$ coincide with the transformations $\xi^n$ from \autoref{prop:CompareEquivonFlat} and \autoref{cor:compareequivnotflat}. 
		\item For every map $g:X\to Y$ in $(\mbsSet)_{\Nsc(\CC)}$, the diagram 
		\[
		\begin{tikzcd}
			\ST_\CC(X)\arrow[d,"\ST_\CC(g)"']\arrow[r,"\xi^\CC(X)"] & \PPhi_\CC(X)\arrow[d,"\PPhi_\CC(g)"]\\
			\ST_\CC(Y)\arrow[r,"\xi^\CC(Y)"'] & \PPhi_\CC(Y)
		\end{tikzcd}
		\]
		commutes 
		\item For every 2-functor $f:\CC\to \DD$, the diagram 
		\[
		\begin{tikzcd}
			f_!\ST_\CC(X) \arrow[r,"f_! \xi^\CC(X)"]\arrow[d,"\cong"] & f_! \PPhi_\CC(X)\arrow[d,"\cong"]\\
			\ST_\DD(f_!X) \arrow[r,"\eta_\CC(f_!X)"'] & \PPhi_\CC(f_!X)
		\end{tikzcd}
		\]
		commutes.
	\end{enumerate}
\end{theorem}

\begin{proof}
	This is identical to the proofs of \cite[Prop 4.3.1 and 4.3.1]{AGSRelNerve}.
\end{proof}

\begin{corollary}\label{cor:relative2}
	The adjunction 
	\[
	\begin{tikzcd}
		\PPhi_\CC: &[-3em] (\mbsSet)_{/\Nsc(\CC)} \arrow[r,shift left] & (\Set_\Delta^{\mathbf{ms}})^{\CC^\op}\arrow[l,shift left] &[-3em]: \cchi_\CC 
	\end{tikzcd}
	\]
	is a Quillen equivalence. 
\end{corollary}

\begin{proof}
	Since $\PPhi_\CC$ preserves cofibrations, and is naturally weakly equivalent to $\ST_\CC$, it preserves trivial cofibrations, and thus is left Quillen. Moreover, the left-derived functors of $\ST_\CC$ and $\PPhi_\CC$ agree, and the former is an equivalence. 
\end{proof}

\subsection{Comparison to the strict case}

We now establish a comparison result with the strict 2-categorical case, as worked out by Buckley in \cite{Buckley}. We will heavily leverage two facts to ease the proof of this comparison results
\begin{itemize}
	\item For a strict 2-functor $F:\CC^\op \to 2\!\Cat$, we can describe $\cchi_\CC(F)$ \emph{entirely} in terms of 2-functors into $\CC$ and $F(x)$, for $x\in \CC$. 
	\item The Duskin 2-nerve $\Nerv_2(\CC)$ of any strict 2-category $\CC$ is 3-coskeletal. 
\end{itemize}
Making use of these two facts allows us to construct a comparison map by checking a finite number of cases by hand. Once the comparison is established, we can work with strict 2-categories to prove that it is a fibre-wise equivalence. 

Let us now introduce the 2-categorical Grothendieck construction we wish to compare with. Appropriately dualizing Buckley's construction\footnote{Buckley defines a construction that takes as input a functor $F:\CC^{(\op,\op)}\to 2\!\Cat$ where both 1- and 2-morphisms have been reversed.}, the \emph{strict 2-categorical Grothendieck construction} of a 2-functor 
\[
\func{F:\CC^\op\to 2\!\Cat}
\]
is the 2-category $\on{El}(F)$ which has 
\begin{itemize}
	\item Objects: pairs $(x,x_-)$ with $x\in \CC$ and $x_-\in F(x)$. 
	\item Morphisms:
	\[
	\func{(f,f_-):(x,x_-)\to (y,y_-)}
	\]
	where $f:x\to y$ in $\CC$, and $f_-:x_-\to F(f)(y_-)$. 
	\item 2-Morphisms: $(\alpha,\alpha_-):(f,f_-)\Rightarrow (g,g_-)$, where $\alpha:f\Rightarrow g$ is a 2-morphism in $\CC$, and $\alpha_-$ fits in the diagram 
	\[
	\begin{tikzcd}
		& & |[alias= T]| F(f)(y_-)\arrow[dd,"F(\alpha)_{y_-}"] \\
		x_-\arrow[urr,"f_-"]\arrow[drr,"g_-"',""{name=U}] & & \\ 
		& & F(g)(y_-)\arrow[from=T, to=U,Rightarrow, "\alpha_-",shorten <= 1em, shorten >=1em]
	\end{tikzcd}
	\]
\end{itemize}
The resulting functor $\on{El}(F)\to \CC$ is a 2-Cartesian fibration, where 
\begin{itemize}
	\item $(f,f_-)$ is Cartesian if $f_-$ is an equivalence, and 
	\item $(\alpha,\alpha_-)$ is coCartesian if $\alpha_-$ is an isomorphism. 
\end{itemize} 

Our aim is to prove the following 

\begin{theorem}\label{thm:BuckleyComp}
	Let 
	\[
	\func{F:\CC^{(\op,-)}\to 2\!\Cat}
	\]
	be a 2-functor, and let $\tilde{F}$ denote the composite 
	\[
	\func{\CC^{(\op,-)}\to 2\!\Cat\to \Set_\Delta^{\mathbf{ms}}}
	\]
	Then there is an equivalence 
	\[
	\begin{tikzcd}
		(\Nerv_2(\on{El}(F)),M,T\subset C)\arrow[dr]\arrow[rr, leftarrow, "\simeq"] & & \cchi_\CC(\tilde{F})\arrow[dl]\\
		& \Nsc(\CC)& 
	\end{tikzcd}
	\]
	of 2-Cartesian fibrations over $\Nsc(\CC)$.  
\end{theorem}

We begin by showing there is a map in one direction. For ease of notation, given a morphism $\phi:x\to y$ in $\CC$, we will write $\phi^\ast:=F(\phi)$. We will employ the same convention for 2-morphisms.  

Since the 2-nerve of a 2-category is 3-coskeletal, it suffices to define a map 
\[
\func{\on{sk}_3(\cchi_\CC(F))\to N_2(\on{El}(F))}
\]
which is compatible with markings and scalings. 

On 0- and 1-simplices, the data specified by the simplices in both constructions is identical. A 2-simplex in $\cchi_\CC(F)$ consists of the following data:
\begin{itemize}
	\item A 2-simplex 
	\[
	\begin{tikzcd}
		& |[alias=T]|x^1\arrow[dr,"\phi_{1,2}"]& \\
		x^0\arrow[ur,"\phi_{01}"]\arrow[rr,"\phi_{02}"',""{name=B}]& & x^2\arrow[from=B,to=T,Rightarrow,"\alpha"]
	\end{tikzcd}
	\]
	\item Three 1-simplices 
	\[
	\func{f_{12}:x^1_-\to \phi_{12}^\ast(x^2_-)}
	\]
	\[
	\func{
		f_{01}:x^0_-\to \phi_{01}^\ast(x^1_-)
	}
	\]
	and 
	\[
	\func{
		f_{02}:x^0_-\to \phi_{02}^\ast(x^2_-)
	}
	\]
	\item A diagram
	\[
	\begin{tikzcd}[row sep=4em,column sep=3em]
		|[alias=X]|\phi^\ast_{01}(x_-^1)\arrow[rr,"\phi_{01}^\ast f_{12}"] & & \phi^\ast_{01}\phi^\ast_{12}(x^2_-)\\
		& & \\
		x^0_-\arrow[rr,"f_{02}"']\arrow[uu,"f_{01}"]\arrow[uurr,bend left,""{name=U}]\arrow[uurr,bend right,""{name=L}] & & |[alias=Y]| x^2_-\arrow[uu,"(\alpha^\ast)_{x^2_-}"']\arrow[from=L,to=U,Rightarrow,"\mu", shorten <=1em, shorten >=1em]\arrow[from=U,to=X,phantom,"\circlearrowleft"{description}]\arrow[from=L,to=Y,phantom,"\circlearrowleft"{description}]
	\end{tikzcd}
	\]
	in $\CC$.
\end{itemize}
These data are identical to the data of a 2-simplex in $\on{El}(F)$. It is immediate that markings and scalings coincide under these correspondences.

Finally, we note that 3-simplices in $\on{El}(F)$ are simply compatibility conditions on 2-morphisms. It is a long but easy check to see that, given a 3-simplex in $\cchi_\CC(F)$, the corresponding 2-simplices in $\on{El}(F)$ are compatible.  We have thus shown

\begin{proposition}
	There is a morphism of 2-Cartesian fibrations 
	\[
	\func{\tau:\cchi_\CC(F)\to \on{El}(F).}
	\]
\end{proposition}

The final ingredient in our proof will involve a comparsion of 2-functors. 

\begin{definition}
	We denote by $\pi^n:\OO^n_{0\upslash}\to \OO^n$ the canonical projection. This sends $S\mapsto \max(S)$, and sends a morphism $\scr{U}:S\to T$ to the set $\scr{U}$. 
	
	Given a 2-category $\scr{C}$, we call a 2-functor
	\[
	f: \OO^n_{0\upslash}\to \CC 
	\]
	\emph{peripatetically constant} if it sends every morphism $\scr{U}:S\to T$ where $\max(S)=\max(T)$ to an identity, and every 2-morphism between such morphisms to an identity as well. We denote the \emph{set} of peripatetically constant functors $\OO^n_{0\upslash} \to \CC$ by 
	\[
	\Hom^{\on{PC}}(\OO^n_{0\upslash},\CC).
	\]
\end{definition}

\begin{lemma}\label{lem:corep_rho_over_pt}
	For any $n\geq 0$, the 2-functor $\pi^n:\OO^n_{0\upslash}\to \OO^n$ induces a bijection 
	\[
	\begin{tikzcd}
		(\pi^n)^\ast: &[-3em] \Hom(\OO^n,\CC) \arrow[r,"\cong"]& \Hom^{\on{PC}}(\OO^n_{0\upslash},\CC)
	\end{tikzcd}
	\]
\end{lemma}

\begin{proof}
	We can define a strict 2-functor 
	\[
	\func*{
		s^n:\OO^n \to \OO^n_{0\upslash};
		j \mapsto {[0,j]}
	}
	\]
	which acts as the identity on 1- and 2-morphisms. Since $\pi^n\circ s^n=\id$, we have $(s^n)^\ast\circ (\pi^n)^\ast=\id$, and thus, $\pi^n_\ast$ is injective. It is immediate from unraveling the definitions that the image of $(\pi^n)^\ast$ is precisely the peripatetically constant functors.  
\end{proof}

\begin{corollary}\label{cor:fibre_of_rho}
	Let $\DD$ be a 2-category, and denote by $\ast$ the terminal 2-category. There is an isomorphism  
	\[
	\cchi_\ast(\DD)\cong \Nsc(\DD).
	\]
\end{corollary}

\begin{proof}
	An $n$-simplex in $\cchi_{\ast}(\DD)$ consists of 2-functors  
	\[
	\func{\theta_I:\OO^I_{i\upslash}\to \DD}
	\]
	for every non-empty $I\subset [n]$, such that for every $J\subset I \subset [n]$, the diagram
	\[
	\begin{tikzcd}
		\OO^I(i,j)\times \OO^J_{j\upslash} \arrow[r]\arrow[d] & \OO^I_{i\upslash}\arrow[d] \\
		\ast \times \DD\arrow[r] & \DD 
	\end{tikzcd}
	\]
	commute. Such functors are uniquely determined by the map $\theta_{[n]}:\OO^n_{0\upslash}\to \DD$, and the commutativity of the diagrams above is equivalent to requiring that $\theta_{[n]}$ be peripatetically constant. 
	
	Consequently, we obtain a bijection on sets of $n$-simplices $\Nsc(\DD)_n\cong (\cchi_\ast(\DD))_n$ by pulling back along $\pi^n$. The corollary follows from checking directly that these bijections respect face and degeneracy maps.   
\end{proof}

\begin{proof}[Proof (of \autoref{thm:BuckleyComp})]
	By \cite[Proposition 3.35]{AGS_CartI}, it will suffice for us to show that this morphism is an equivalence on fibres. By construction, the fibre of $\on{El}(F)$ over $x\in \CC$ is precisely the 2-category $F(x)$, and the fibre of $\cchi_\CC(F)$ over $x$ is \emph{also} precisely $\cchi_{x}(F(x))\cong \Nsc(F(x))$. 
	
	It is a quick explicit check that, on 0-,1-,2-, and 3-simplices, the map $\tau:\cchi_x(F(c))\to \on{El}_x(F(x))\cong \Nsc(F(x))$ agrees with the isomorphism of \autoref{cor:fibre_of_rho}. The theorem then follows from 3-coskeletalness. 
\end{proof}

%
%

\end{document}